\documentclass[11pt, a4paper, leqno]{amsart}
\usepackage{tikz}
\usepackage[utf8]{inputenc}
 
\usepackage[euler-digits]{eulervm}

\usepackage{amsfonts, amsthm, amssymb, amsmath}
\usepackage{changepage}
\usepackage{booktabs}
\usepackage{float}
\usepackage{graphicx}
\usepackage{xcolor,colortbl}
\usepackage{mathtools}
\usepackage{mathrsfs,array}
\usepackage{eucal,fullpage,times,color,enumerate,accents}
\usepackage{url}
\usepackage{comment}
\usepackage{array}

\usepackage{hyperref}
\hypersetup{
  colorlinks   = true,          
  urlcolor     = blue,          
  linkcolor    = teal,          
  citecolor   = orange             
}

\usepackage{color}
\usepackage{mathrsfs}
\usepackage{amssymb}
\usepackage{tikz}
\usepackage{tikz-cd}
\usepackage{bm}
\usepackage{enumerate}
\usetikzlibrary{calc}
\usepackage{subfigure}
\usepackage{placeins}

\theoremstyle{plain}
\newtheorem{theorem}{Theorem}[section]

\theoremstyle{definition}
\newtheorem{example}[theorem]{Example}
\newtheorem{observation}[theorem]{Observation}

\newtheorem{quasi-theorem}[theorem]{Quasi-Theorem}

\newtheorem{rem1}[theorem]{Remark}
\newenvironment{remark}{\begin{rem1}}{\end{rem1}}

\newtheorem{not1}[theorem]{Notation}

\theoremstyle{plain}

\newtheorem{corollary}[theorem]{Corollary}

\theoremstyle{definition}
\newtheorem{definition}[theorem]{Definition}

\theoremstyle{plain}
\newtheorem{lemma}[theorem]{Lemma}
\newtheorem{proposition}[theorem]{Proposition}
\numberwithin{equation}{section}
\newtheorem*{proposition*}{Proposition}
\newtheorem*{definition*}{Definition}
\newtheorem*{remark*}{Remark}
\newtheorem*{observation*}{Observation}
\theoremstyle{plain}
\newtheorem*{lemma*}{Lemma}
\newtheorem*{theorem*}{Theorem}
\newtheorem*{question*}{Question}
\newtheorem*{corollary*}{Corollary}
\newtheorem*{claim*}{Claim}

\usepackage[noend]{algorithmic}

\usepackage[linesnumbered,longend,ruled,vlined]{algorithm2e}

\usepackage{minted}
\setminted{breaklines=true}
\definecolor{lbcolor}{rgb}{0.9,0.9,0.9}
\setminted{bgcolor=lbcolor}
\setminted{fontsize=\footnotesize}

\usepackage{tabularx,float,capt-of}

\newcommand{\tmfloatcontents}{}
\newlength{\tmfloatwidth}
\newcommand{\tmfloat}[5]{
	\renewcommand{\tmfloatcontents}{#4}
	\setlength{\tmfloatwidth}{\widthof{\tmfloatcontents}+1in}
	\ifthenelse{\equal{#2}{small}}
	{\setlength{\tmfloatwidth}{0.45\linewidth}}
	{\setlength{\tmfloatwidth}{\linewidth}}
	\begin{minipage}[#1]{\tmfloatwidth}
		\begin{center}
			\tmfloatcontents
			\captionof{#3}{#5}
		\end{center}
\end{minipage}}

\title{Cohomology Vanishing, Koszul Cohomology and Multigraded Regularity on Projective Varieties}
\author{ Raneeta Dutta}
\address{Department of Mathematics, University of Kansas,
Snow Hall, 1460 Jayhawk Blvd, Lawrence, Kansas 66045, USA}
\email{rdutta2021@ku.edu}

\begin{document}

\begin{abstract}
In this article, we prove a general cohomology vanishing theorem on arbitrary projective varieties within the framework of multigraded Castelnuovo--Mumford regularity. In particular, we apply this vaishing theorem to prove the vanishing of Koszul cohomology groups $K_{p,q}(X;F,L)$, where $L=B_1^{w_1}\otimes\cdots\otimes B_t^{w_t}$, the line bundles $B_1,\ldots,B_t$ are globally generated, and $F$ is a vector bundle.  We also introduce the $K_{p,q}$-hierarchy, providing a unified perspective on the vanishing criteria for Properties $(N_{p})$ and $(M_{q})$ while shedding light on the vanishing of mixed-weight syzygies. These results generalize earlier work in \cite{HeringSchenckSmith}, \cite{GallegoPurnaprajnaII}, and \cite{Basu}.  Furthermore, we give a complete description of minimal multigraded regularities of line bundles on arbitrary products of projective spaces. We also recover Green's vanishing theorem for projective space via a direct regularity argument, giving an alternative proof.  Finally, we compute the complete graded Betti table of the canonical image of a hyperelliptic curve by combining our vanishing theorem with Green's duality. 
\end{abstract} 

\maketitle
\vspace{-0.8cm}

\tableofcontents

\section{\textbf{Introduction}}
\label{section_introduction}
The topic of Koszul cohomology groups and syzygies has attracted a lot of attention among geometers and algebraists alike for the last 45 years. Much of the result have been regarding property $N_p$ and more recently on the so called property $M_q$. There are variation of these themes in the multigraded setting. In this article, we prove a general cohomology vanishing theorem on arbitrary projective varieties within the framework of multigraded Castelnuovo--Mumford regularity.
In particular, we study syzygies related to globally generated line bundle $L$ on a projective variety $X$. The methods in this paper illustrates multigraded regularity serves a useful technical tool for understanding the structure of minimal free resolutions and the interplay between geometry, cohomology, and syzygies.

We need the following set up to discuss our results. 

Let $L$ be a globally generated line bundle on a projective variety $X$. The complete linear series $|L|$ defines a morphism
$
\varphi_L:X\longrightarrow \mathbb P(H^0(X,L))=\mathbb P^r,$ where $r=h^0(X,L)-1.
$
Let
$S:=\operatorname{Sym}^{\bullet}(H^0(X,L))$
be the homogeneous coordinate ring of $\mathbb P^r$, and let
$$
R:=R(X,L)=\bigoplus_{j\ge0}H^0(X,L^j)$$
be the section ring of $L$. Then $R$ is naturally a graded $S$-module with minimal graded free resolution
\[
\cdots\longrightarrow E_{i+1}\longrightarrow E_i\longrightarrow\cdots
\longrightarrow E_1\longrightarrow E_0\longrightarrow R\longrightarrow0.
\]

A theorem of Green identifies the terms of this resolution as
\[
E_i=\bigoplus_{j\ge0}
K_{i,j}(X,L)\otimes S(-i-j),
\]
 Thus, the weight-$j$ syzygies at the $i$-th stage are encoded by the Koszul cohomology groups $K_{i,j}(X,L)$.

Associated to $L$ is the evaluation sequence
\[
0
\longrightarrow
M_L
\longrightarrow
H^0(X,L)\otimes\mathcal O_X
\overset{\mathrm{ev}_L}{\longrightarrow}
L
\longrightarrow
0,
\]
where $M_L$ denotes the syzygy bundle of $L$. A fundamental feature of $M_L$ is that the vanishing of suitable cohomology groups involving its tensor powers governs the vanishing of Koszul cohomology groups, and consequently the syzygies of the image $\varphi_L(X)$.

The property $(N_p)$ of a line bundle $L$ describes the beginning of the minimal resolution by answering when the first few terms in $E_{\bullet}$ are as simple as possible.

\begin{definition}
A globally generated line bundle $L$ on a projective variety $X$ satisfies property $(N_p)$, where $p\ge0$, if
\[
E_0=S,
\qquad
E_i=\bigoplus S(-i-1),
\quad
1\le i\le p.
\]
\end{definition}

If $L$ is ample, then property $(N_p)$ implies that the first $p$ steps of the resolution are linear. In general, $\varphi_L(X)$ is projectively normal if and only if $L$ satisfies $(N_0)$ and $\varphi_L(X)$ is normal.

Multigraded regularity was introduced in \cite{MaclaganSmith} and further studied by Hering–Schenck–Smith in \cite{HeringSchenckSmith}. They proved Property $(N_{p})$ of line bundles of the form
\[
L = B^{\vec{w}} := B_1^{w_1}\otimes \cdots \otimes B_t^{w_t},
\qquad \vec{w}\in \mathbb{N}^t,
\]
 in terms of conditions on the weight vector $\vec{w}$ where $B_{1},\cdots, B_{t}$ are fixed $t$ globally generated line bundles on a projective variety $X$.

Most of the results concern property $(N_{p})$ which deal with beginning of the resolution until the $p$th stage. But in order to get more understanding of the resolution one has to see what happens at the end of the resolution as well, this motivates the following notion of the property $(M_q)$.

\begin{definition}
A globally generated line bundle $L$ on a projective variety $X$ is said to satisfy property $(M_q)$ if
\[
K_{p,1}(X,L)=0
\qquad\text{for all } p\ge r-q.
\]
\end{definition}

 While property $(N_p)$ has been studied extensively over the past several decades, comparatively less is known about property $(M_q)$, particularly in higher dimensions. Apart from earlier work on curves \cite{GreenLazarsfeldI} and abelian varieties \cite{AproduLombardi}, the first systematic study for higher-dimensional varieties was carried out by Basu, who established $(M_q)$ criteria for powers of ample line bundles and adjoint line bundles on several classes of varieties.

In this regard, we recall the definition of multigraded regularity.

\begin{definition}[Multigraded Castelnuovo–Mumford regularity]
Let $B_1,\dots,B_t$ be fixed globally generated line bundles on a projective variety $X$. Let $\mathscr{F}$ be a coherent sheaf on $X$, and let $F$ be a vector bundle. We say that $\mathscr{F}$ is $F$-regular (with respect to $B_1,\dots,B_t$) if
\[
H^i\left(X,\mathscr{F}\otimes F\otimes B^{-\vec{u}}\right)=0
\quad \text{for all } i>0 \text{ and all } \vec{u}\in \mathbb{N}^t \text{ with } |\vec{u}|=i.
\]
\end{definition}
    
{\textbf{Minimal Multigraded Regularity on Products of Projective Spaces:}} Understanding multigraded regularity explicitly is a subtle problem, even for products of projective spaces. We investigate the multigraded regularity of line bundles on arbitrary products of projective spaces. Our first main result provides a complete and explicit description of the \emph{minimal} multigraded regularities in this setting. More precisely, we show that the \emph{minimal regularity vectors} ( see page ) are naturally indexed by the permutations of the factors, yielding exactly $n!$ minimal regularity vectors for $\mathbb{P}^{m_1}\times\cdots\times\mathbb{P}^{m_n}$.

\begin{theorem}
\label{thm:minimal-regularity} ( see Theorem \ref{theorem:minimal-regularity} )
Let
$
X=\mathbb{P}^{m_1}\times \mathbb{P}^{m_2}\times \cdots \times \mathbb{P}^{m_n},$
and let $B_1,\dots,B_n$ denote the pullbacks of 
$\mathcal{O}_{\mathbb{P}^{m_i}}(1)$ under the natural projections.
For $\vec{a}=(a_1,\dots,a_n)\in \mathbb{Z}^n$, set
\[
L=B^{\vec{a}}=B_1^{a_1}\otimes \cdots \otimes B_n^{a_n}.
\]

Then the minimal $\vec{r}=(r_1,\dots,r_n)\in \mathbb{Z}^n$ for which 
$L$ is $B^{\vec{r}}$-regular (with respect to $B=(B_1,\dots,B_n)$) 
are indexed by permutations of $\{1,\dots,n\}$. More precisely, for each permutation $\sigma\in S_n$, the corresponding 
minimal regularity vector $\vec{r}^{(\sigma)}$ is given by
\[
r^{(\sigma)}_{\sigma(i)}
=
-a_{\sigma(i)} + \sum_{j>i} m_{\sigma(j)},
\qquad \text{for } i=1,\dots,n.
\]

\end{theorem}

Besides furnishing an effective computational tool for determining multigraded regularity (see Remark \ref{remark_geometry-of-minimal-regularity}), this theorem reveals an uniform combinatorial and geometric structure underlying multigraded regularity on products of projective spaces, making it a result of independent interest.

In addition, Section \ref{section-5} investigates the behavior of multigraded regularity under restriction to hypersurfaces and complete intersections.

{\textbf{General Cohomology Vanishing Theorem:}} We now come to the main result of this paper, which is set in the framework of multigraded regularity. This result provides a general cohomology vanishing which acts as a single unifying mechanism from which several corollaries on syzygies follow.

\begin{theorem} (\textbf{Main Theorem}, see Theorem \ref{Theorem_6.1}) \par
\label{general_theorem}
  Let $X$ be a projective variety and let $\vec{m} \in \mathbb{Z}^t$ and $\vec{w_1},...,\vec{w}_{k+1} \in \mathbb{N}^t$ be arbitrary vectors. Let $F$ be any vector bundle on $X$ and $F$ is $B^{\vec{m}-\vec{e_j}}$-regular for all $1\le j\le t$. Then the following is true: $$H^1(M_{B^{\vec{w_{1}}}} \otimes M_{B^{\vec{w_2}}} \otimes...\otimes M_{B^{\vec{w}_{k+1}}} \otimes F \otimes B^{\vec{m_{k}}})=0$$ when $\vec{m_k}\ge \vec{m} + (k-1)\vec{\delta}$ and
  $$H^i({M_{B^{\vec{w_1}}}} \otimes...\otimes {M_{B^{\vec{w_{k+1}}}}} \otimes F \otimes {B^{\vec{m_k}}})=0, \text{ for all } i > 0$$
  when $\vec{m_k}\ge \vec{m}+k\vec{\delta}$.
\end{theorem}

\textbf{Vanishing of Koszul Cohomology of all weights:} We now start with an important corollay of the main theorem \ref{general_theorem}. This result utilizes positivity conditions on line bundles in order to achieve the vanishing of Koszul Cohomologies of all weights $q\ge 2$. 

 \begin{corollary} \label{Kp,q_vanishing}
 (see Corollary \ref{CorollaryKoszulCohomology})
    Let $X$ be a projective variety and $F$ be any vector bundle on $X$ which is $B^{\vec{m}-\vec{e_j}}$-regular for some $\vec{m}\in \mathbb{N}^t$ and for all $1\le j \le t$. Let $L=B^{\vec{w}}$ be any line bundle on $X$ where $\vec{w}\in \mathbb{N}^t$. Then the Koszul Cohomology group $K_{p,q}(X,F;L)=0$ when $(q-1)\vec{w}\ge \vec{m}+(p-1)\vec{\delta}$. 
\end{corollary}

One of the notable consequences of Corollary \ref{Kp,q_vanishing} is that it reveals a natural hierarchy for vanishing of Koszul Cohomology groups of different weights. For a fixed stage $p$ of $E_{\bullet}$ the vanishing criterion 
$$
\vec{w}\ge \frac{\vec{m}+(p-1)\vec{\delta}}{q-1}
$$
shows that the positivity $(\vec{w})$ required for the vanishing of $K_{p,q}(X,F;L)$ decreases as the weight $q$ increases. So, if we denote the space of line bundles satisfying the above inequality as $R_{p,q}$, then it follows that for fixed $p$, 
\begin{equation}
\label{hierarchy_q}
    R_{p,q-1}\subset R_{p,q}\subset R_{p,q+1}
\end{equation}
 And similarly, if we keep the weight $q$ fixed and let the index $p$ vary, then also we get 
 \begin{equation}
 \label{hierarchy_p}
     R_{p+1,q}\subset R_{p,q}\subset R_{p-1,q}
 \end{equation}
We call the inequalities (\ref{hierarchy_q}) and (\ref{hierarchy_p}) the $q-$\textbf{hierarchy} and $p-$\textbf{hierarchy}, respectively.

The following figure illustrates the $K_{p,q}$-hierarchy when we specialize our results for $t=2$ and $F=\mathcal{O}_{X}$.

 \begin{figure}[ht!]
 \begin{adjustwidth*}{}{-0.5em} 
  \begin{minipage}{0.3\textwidth}
\begin{tikzpicture}[row sep = 1.3, column sep = 0.7]
\fill[cyan!15!] (2,3) rectangle (4, 4);
\fill[green!15!] (2-0.5,3-0.5) rectangle (2, 4);
\fill[green!15!] (2-0.5,3-0.5) rectangle (4, 3);
\fill[yellow!90!] (2-1.5,3-1.5) rectangle (2-0.5, 4);
\fill[yellow!90!] (2-1.5,3-1.5) rectangle (4, 3-0.5);
\draw[<->] (-.7, 1) -- (4, 1);
\node[scale=0.8] at (4, 0.5) {$B_{1}$};
\draw[<->] (0, 1-.7) -- (0, 4);
\node[scale=0.8] at (-0.5, 4) {$B_{2}$};
\draw[->] (2, 3) -- (2, 4);
\draw[->] (2, 3) -- (4, 3);
\draw[->] (2-0.5, 3-0.5) -- (2-0.5, 4);
\draw[->] (2-0.5, 3-0.5) -- (4, 3-0.5);
\draw[->] (0.5, 2-0.5) -- (0.5, 4);
\draw[->] (0.5, 2-0.5) -- (4, 2-0.5);
\node[scale=0.8] at (2.75, 3.5) {$(N_{p})$};
\node[scale=0.8] at (2, 0.5) {$K_{p, q}(X, L) = 0$};
\node[scale=0.8] at (5.75, 3) {$q = 2, \ \vec{w} \geq \vec{m}+(p-1)\vec{\delta}$};
\node[scale=0.8] at (5.75, 3-0.5) {$  q = 3,  \vec{w} \geq \dfrac{\vec{m}}{2}+\dfrac{(p-1)}{2}\vec{\delta}$};
\node[scale=0.8] at (5.75, 3-1.5) {$q = 4, \vec{w} \geq \dfrac{\vec{m}}{3}+\dfrac{(p-1)}{3}\vec{\delta}$};
\draw[->, dashed] (0, 1) -- (2-0.05, 3-0.05);
\node[scale=0.75] at (1+0.2, 2-0.75) {$\vec{m}+(p-1)\vec{\delta}$};
\end{tikzpicture}
\end{minipage}
\hspace{3.25cm}
 \begin{minipage}{0.3\textwidth}
\begin{tikzpicture}[row sep = 1.3, column sep = 0.7]
\fill[cyan!15!] (2,3) rectangle (4, 4);
\fill[green!15!] (2-0.75,3-0.75) rectangle (2, 4);
\fill[green!15!] (2-0.75,3-0.75) rectangle (4, 3);
\fill[yellow!90!] (2-1.5,3-1.5) rectangle (2-0.75, 4);
\fill[yellow!90!] (2-1.5,3-1.5) rectangle (4, 3-0.75);
\draw[<->] (-.7, 1) -- (4, 1);
\node[scale=0.8] at (4, 0.5) {$B_{1}$};
\draw[<->] (0, 1-.7) -- (0, 4);
\node[scale=0.8] at (-0.5, 4) {$B_{2}$};
\draw[->] (2, 3) -- (2, 4);
\draw[->] (2, 3) -- (4, 3);
\draw[->] (2-0.75, 3-0.75) -- (2-0.75, 4);
\draw[->] (2-0.75, 3-0.75) -- (4, 3-0.75);
\draw[->] (2-1.5, 3-1.5) -- (2-1.5, 4);
\draw[->] (2-1.5, 3-1.5) -- (4, 3-1.5);
\node[scale=0.8] at (3, 3.5) {$(K_{p, q} = 0)$};
\node[scale=0.8] at (3, 2.5) {$(K_{p-1, q} = 0)$};
\node[scale=0.8] at (3, 1.75) {$(K_{p-2, q})$};
\node[scale=0.8] at (2, 0.5) {$K_{p, q}(X, L) = 0$};
\node[scale=0.8] at (5.75, 3) {$ \vec{w} \geq \dfrac{\vec{m}+(p-1)\vec{\delta}}{q-1}$};
\node[scale=0.8] at (5.75, 3-0.75) {$\vec{w} \geq \dfrac{\vec{m}+((p-1)-1)\vec{\delta}}{q-1}$};
\node[scale=0.8] at (5.75, 3-1.5) {$\vec{w} \geq \dfrac{\vec{m}+((p-2)-1)\vec{\delta}}{q-1}$};
\draw[->, dashed] (2-1.5, 3-1.5) -- (2-0.75, 3-0.75);
\node[scale=0.75] at (2-0.75-0.45, 3-0.75) {$\vec{\delta}$};
\draw[->, dashed] (2-0.75, 3-0.75) -- (2, 3);
\node[scale=0.75] at (2-0.45, 3) {$\vec{\delta}$};
\end{tikzpicture}
\end{minipage}
\caption{\emph{Hierarchy of vanishing regions for Koszul cohomology groups ${K_{p,q}(X,L)}$ of various weights $q$, for fixed stage $p$, with $L = B^{\vec{w}}$ and $F=\mathcal{O}_X$ }}
\label{IntBound}
\end{adjustwidth*}
\end{figure}
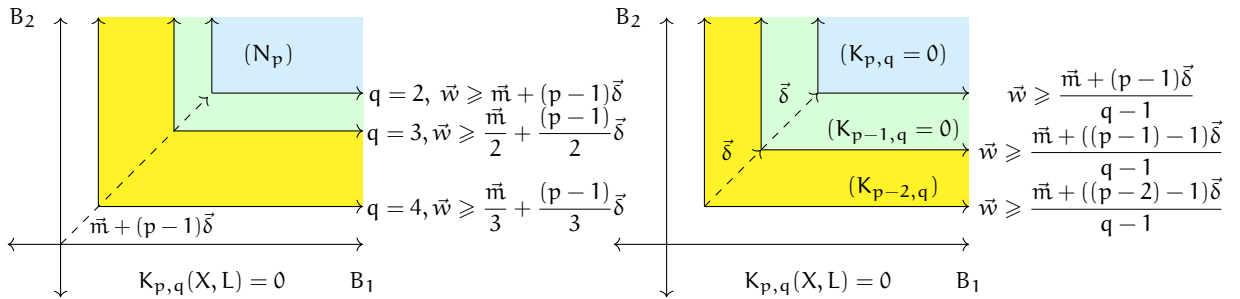

Also, we demonstrate in Example \ref{Example: Sharpness of Higher Koszul Cohomology} that our bound mentioned above is sharp for weight $3$ syzygies of the canonical line bundle on a  hyperelliptic curve, which in turn shows the inclusion $R_{g-2,3}\subsetneq R_{g-3,3}$. Thus, the inclusion mentioned above is in general strict.

.

\paragraph{\textbf{Slope interpretation:}}
The vanishing criterion of Corollary~\ref{CorollaryKoszulCohomology} also admits a natural geometric interpretation. In the one-dimensional setting, the condition $(q-1)w \ge m+(p-1)\delta$ determines a family of straight lines in the $(w,p)$-plane, parametrized by the weights $q$. For each fixed coordinate, these boundary lines have slope $1/(q-1)$. As the weight $q$ increases, the slopes decrease and the corresponding vanishing regions become progressively larger, as illustrated in Figure~8. The appeal of this interpretation lies in the fact that it reduces the seemingly intricate multigraded setting to a simple two-dimensional picture, where the hierarchy of vanishing is governed by the elementary notion of the slope of a line.

An alternative interpretation of Corollary \ref{Kp,q_vanishing} yields explicit vanishing criteria for Koszul cohomology on products of projective spaces and complete intersections. As a special case, we obtain explicit criteria for complete intersections in products of projective spaces (Corollaries \ref{corollary_koszul_prod_projsp} and \ref{Syz_complete_intersection}).

\textbf{New proof of a theorem of Mark Green:} As a further application of Corollary~\ref{CorollaryKoszulCohomology}, we recover a classical vanishing theorem of Mark Green for Koszul cohomology on projective space (Theorem \ref{theorem:green-koszul-cohomology-projective-space}). More precisely, we show that Green's theorem
\[
K_{p,q}\bigl(\mathbb{P}^{r},\mathcal{O}_{\mathbb{P}^{r}}(k),\mathcal{O}_{\mathbb{P}^{r}}(d)\bigr)=0
\qquad \text{whenever } k+(q-1)d\ge p,
\]
arises as a direct specialization of our general vanishing criterion. While Green's original proof relies on a detailed analysis of the Koszul complex through homological and representation-theoretic techniques, our derivation follows formally from Castelnuovo--Mumford regularity. This further illustrates the unifying nature of the multigraded regularity framework developed in this paper.

\textbf{Property $(M_q)$: Vanishing of weight-one syzygies from the tail of the resolution:} The Theorem \ref{general_theorem} also yields criteria for Property $(M_q)$, which governs the vanishing of weight-one syzygies from the tail of the minimal free resolution. While Corollary~\ref{Kp,q_vanishing} is most effective for vector bundles with sufficient positivity, Green’s Koszul duality allows us to extend these results to situations with weaker positivity assumptions.

We now specialize this general framework to obtain a criterion for Property $(M_{q})$. 

\begin{corollary} (see Corollary \ref{CorollaryM_q})
\label{corollaryMq}
Let $X$ be a projective variety of dimension $n\ge 2$ with rational singularities and whose canonical sheaf $\omega_X$ is a line bundle. Let $L=B^{\vec{w}}$ be an ample line bundle on $X$ with $\vec{w}\in \mathbb{N}^t$. Moreover, assume that $\omega_X$ is $B^{\vec{m}-\vec{e_j}}$-regular for all $1 \le j \le t$ for some $\vec{m}\in \mathbb{N}^t$. Then $L$ satisfies Property $(M_q)$ if
\[
(n-1)\vec{w}\ge \vec{m}+(q-n-1)\vec{\delta}.
\]
\end{corollary}

This result extends Theorem 2.2 of Basu in \cite{Basu} to the multigraded case.

 The examples in Section 10 show that our result on Property $(M_{q})$ is sharp. We show that the sharpness is governed by the gonality of curves in the corresponding linear systems, which reflects the deep relationship between Property $(M_q)$ and the Gonality Conjecture where this conjecture is extended to certain higher dimensional varieties in \cite{Basu}.

\textbf{Property $\mathbf{(N_{p})}:$} The vanishing theorem provides information across the entire resolution. In particular, we obtain the following multigraded criterion for Property $(N_p)$, that tells about the linear strands from the beginning.  This Corollary recovers a variant of result of Hering-Schenck-Smith in \cite{HeringSchenckSmith} when we specialize Theorem \ref{general_theorem} to $F=\mathcal{O}_{X}$.

\begin{corollary} ( see Corollary \ref{Corollary_Np-property})
\label{Corollary-Np-property}
    Let $X$ be a projective variety and let $\vec{m}\in \mathbb{Z}^t$ be such that $\mathcal{O}_{X}$ is $B^{\vec{m}-\vec{e_{j}}}$-regular for $1\le j \le t$. Let $L=B^{\vec{w}}$ be any line bundle for some $\vec{w} \in \mathbb{N}^{t}$. Then for $p\ge 1, L$ satisfies property $(N_p)-$ if $\vec{w} \ge \vec{m}+ (p-1)\vec{\delta}$. 
\end{corollary}

In connection with Property $(N_p)$, we also obtain a reformulation (Theorem~\ref{theorem-HSS-general}) of the well-known theorem of Hering, Schenck, and Smith \cite{HeringSchenckSmith} on Property $(N_p)$ (Corollary~\ref{corollary-N_p-HSS}). Though the technique differs slightly from that of Theorem~\ref{general_theorem}, this reformulation gives a direct generalization of their result on Property $(N_{p})$ by utilizing the full scope of their technique.
More precisely, this reformulation shows that the proof of their theorem is not just limited to the particular line bundle $\mathcal{O}_X$ but also applies to every line bundle of the form $L=B^{\vec{v}}$, where $\vec{v}\ge\vec{0}$. A detailed study of this reformulation is presented in Section~\ref{section_9_reformulation_of_HSS}.

As a further application of our criterion for Property $(N_p)$, we obtain the following explicit multigraded bound for varieties with trivial canonical bundle. 

\begin{corollary} (see Corollary \ref{Corollary_Np-propertyII-2})
\label{Corollary_Np-propertyII}
Let $X$ be a smooth projective variety of dimension $n$ with $\omega_{X}=\mathcal{O}_{X}$. 
Let $B_{1},\cdots,B_{t}$ be fixed ample and base point free line bundles on $X$, where $t\ge 2$. 
Let $L=B^{\vec{w}}$ be an ample and base point free line bundle on $X$ for some $\vec{w}\in \mathbb{N}^{t}$. 
Then $L$ satisfies Property $(N_p)$ if
\[
\vec{w}\ge (p+n)\vec{\delta}.
\]
\end{corollary}

This criterion applies to abelian varieties, K3 surfaces and Calabi--Yau varieties. Furthermore, when $t=1$, Corollary~\ref{Corollary_Np-propertyII} generalizes Corollary~1.6 of Gallego--Purnaprajna \cite{GallegoPurnaprajnaI} for Calabi--Yau $n$-folds, thereby placing their result in a broader multigraded framework.  The Calabi-Yau case in the multigraded setting will be investigated further in the subsequent article which study the optimality of the above bound.

The examples in Section 7 show that our result is sharp, establishing that the bounds cannot, in general, be improved. Finally, if we specialize more general results of Gallego--Purnaprajna in \cite{GallegoPurnaprajnaIII,GallegoPurnaprajnaIV} to product of projective spaces, the multigraded regularity approach developed in this paper, yields substantially stronger bounds for property $(N_p)$ than their bounds.

\textbf{Betti-Table of Canonical Image of Hyperelliptic Curves} In Section \ref{Appendix_B}, we determine the complete graded Betti table of the canonical image of a hyperelliptic curve $C$ of genus $g$. While the higher-weight Koszul cohomology in weight $3$ is taken care of by our main vanishing theorem, demonstrating the sharpness of Corollary \ref{CorollaryKoszulCohomology} for $q=3$, the remaining rows are discovered using Green’s duality together with an independent geometric analysis of the double cover structure $C \to \mathbb{P}^1$. In particular, the weight $1$ and weight $2$ Koszul cohomology groups are computed explicitly via this factorization, allowing us to assemble the full Betti table. To the best of our knowledge, this complete description of the Betti table of $(C,\omega_C)$ is not previously available in the literature.

\section{\textbf{Notations And Conventions}}
\label{section_2}
\begin{itemize}
    \item[(i)] We work over a field of characteristic zero in this paper. We denote it by $\mathbb{C}$. 
    \item[(ii)] For us $\mathbb{N}$ is the set of all non negative integers. 
    \item[(iii)] We denote a t-tuple $(w_1,\cdots,w_t)\in \mathbb{N}^t$ by the vector $\vec{w}$. And we say $\vec{w}=(w_1,\cdots,w_t)$ is greater than or equal to $\vec{v}=(v_1,\cdots,v_t)$ if each $w_i\ge v_i$ for all $1\le i \le t$ and we denote it by $\vec{w}\ge \vec{v}$. 
    \item[(iv)] Most of the times, we use the phrase " Multigraded Regularity" for " Multigraded Castelnuovo-Mumford Regularity" in this paper.
\end{itemize}

\section{\textbf{Background}}
\label{section_background}

\subsection{Castelnuovo--Mumford Regularity:}

We begin this section by recalling the well-known notion of Castelnuovo--Mumford regularity, whose generalization forms the foundation for the proof of the main theorem of this paper.

\begin{definition}

\textbf{(Castelnuovo--Mumford Regularity of a coherent sheaf):}
Let $\mathcal{F}$ be a coherent sheaf on a projective space $\mathbb{P}$ and $m$ be an integer. We say that $\mathcal{F}$ is $m$-regular if
$$
H^{i}(\mathbb{P}, \mathcal{F}(m-i))=0
$$
for all $i>0$.

\end{definition}

Equivalently, the above cohomology vanishing condition can be expressed as
$$
H^{i}(\mathcal{F}\otimes \mathcal{O}_{\mathbb{P}}(1)^{\otimes (m-i)})=0.
$$
This definition of Castelnuovo--Mumford regularity gives the notion of regularity with respect to a globally generated and ample line bundle on an arbitrary projective variety. Throughout the paper, we use the terms \emph{globally generated} and \emph{base point free} interchangeably when referring to a line bundle.

\begin{definition}
\label{Definition_3.2}
\textbf{(Regularity with respect to a globally generated and ample line bundle):}
Let $X$ be a projective variety and $B$ a globally generated and ample line bundle on $X$. A coherent sheaf $\mathcal{F}$ on $X$ is said to be $m$-regular with respect to the line bundle $B$ if

\begin{center}
$H^i(X, \mathcal{F}\otimes B^{m-i})=0$ for $i>0$.
\end{center}

\end{definition}

We conclude this subsection by recalling the following well-known theorem on regularity, first proved by Mumford for projective space in \cite{MumfordCurves}.

\begin{theorem}
\label{Theorem_3.3}

Let $\mathcal{F}$ be a coherent sheaf on a projective variety $X$ which is $m$-regular with respect to a globally generated and ample line bundle $B$. Then for every $k\ge 0$,

\begin{itemize}
    \item[(i)] The maps
    $$
    H^0(X,\mathcal{F}\otimes B^{\otimes m}) \otimes H^0(X,B^{\otimes k})
    \rightarrow
    H^0(X,\mathcal{F}\otimes B^{\otimes(m+k)})
    $$
    are surjective.

    \item[(ii)] $\mathcal{F}$ is $(m+k)$-regular with respect to $B$.

    \item[(iii)] $\mathcal{F}\otimes B^{\otimes(m+k)}$ is globally generated.
\end{itemize}

\end{theorem}

\subsection{Multigraded Castelnuovo--Mumford Regularity:}

Multigraded regularity was first introduced by D. Maclagan and G. G. Smith in \cite{MaclaganSmith}. We follow the setup introduced in \cite{HeringSchenckSmith} to define multigraded regularity.

Fix a collection of globally generated line bundles $B_1,\ldots,B_t$ on $X$. For $\vec{u}:=(u_1,\ldots,u_t)\in \mathbb{Z}^t$, we set
\[
B^{\vec{u}}:=B_1^{u_1}\otimes B_2^{u_2}\otimes\cdots\otimes B_t^{u_t},
\]
and let
\[
\mathfrak{B}=\left\{B^{\vec{u}}\mid \vec{u}\in\mathbb{N}^t\right\}.
\]
Thus, for every $\vec{v}\in\mathbb{N}^t$, we have $B^{\vec{v}}\in\mathfrak{B}$.

As usual, we denote the standard basis of $\mathbb{Z}^t$ by $\vec{e_1},\ldots,\vec{e_t}$, so that $B^{\vec{e_j}}=B_j$.

\begin{definition}
\label{Definition_3.4}

\textbf{(Multigraded Regularity of a Coherent Sheaf with respect to globally generated line bundles $\mathbf{B_1},\ldots,\mathbf{B_t}$):}
Let $\mathcal{F}$ be a coherent sheaf on $X$ and let $E$ be a vector bundle on $X$. We say that $\mathcal{F}$ is $E$-regular with respect to $B_1,\ldots,B_t$ if
\[
H^i(X,\mathcal{F}\otimes E\otimes B^{-\vec{u}})=0
\]
for all $i>0$ and for all $\vec{u}\in\mathbb{N}^t$ such that $|\vec{u}|:=u_1+\cdots+u_t=i$.

\begin{remark}
Throughout this article, whenever we say that a coherent sheaf $\mathcal{F}$ is $E$-regular, we always mean regularity with respect to the fixed line bundles $B_1,\ldots,B_t$. Hence, we will not mention them explicitly.
\end{remark}

\begin{remark}
In \cite{HeringSchenckSmith}, the authors define $L$-regularity, where $L$ is any line bundle. Here, we extend that notion to an arbitrary vector bundle $E$.
\end{remark}

\end{definition}

We conclude this subsection by recalling the following theorem on multigraded regularity from \cite[\textsection~2]{HeringSchenckSmith}. It is the multigraded analogue of Theorem \ref{Theorem_3.3}.

\begin{theorem}
\label{Theorem_3.7}

Let $\mathcal{F}$ be a coherent sheaf on a projective variety $X$ and let $E$ be a vector bundle on $X$ such that $\mathcal{F}$ is $E$-regular. Then the following hold.

\begin{itemize}
    \item[(i)] $\mathcal{F}$ is $E\otimes B^{\vec{u}}$-regular for every $\vec{u}\in\mathbb{N}^t$.

    \item[(ii)] The maps
    \[
    H^0(X,\mathcal{F}\otimes E\otimes B^{\vec{u}})
    \otimes
    H^0(X,B^{\vec{v}})
    \longrightarrow
    H^0(X,\mathcal{F}\otimes E\otimes B^{\vec{u}+\vec{v}})
    \]
    are surjective for all $\vec{u},\vec{v}\in\mathbb{N}^t$.

    \item[(iii)] $\mathcal{F}\otimes E\otimes B^{\vec{u}}$ is globally generated, provided that there exists a $\vec{w}\in\mathbb{N}^t$ such that $B^{\vec{w}}$ is ample.
\end{itemize}

\end{theorem}

\subsection{\textbf{Multigraded Regularity as a Generalization of Regularity:}}

As discussed above, multigraded regularity may be viewed as a natural generalization of regularity. Since both notions occur extensively in the literature, it is useful to understand how they are related, both technically and geometrically.

Let $\mathcal{F}$ be a coherent sheaf on a projective variety $X$. Suppose that $\mathcal{F}$ is $r$-regular with respect to a globally generated line bundle $B$. Then
\[
H^i(\mathcal{F}\otimes B^{\otimes r}\otimes B^{\otimes(-i)})=0
\]
for all $i>0$. Thus, by Definition \ref{Definition_3.4}, $\mathcal{F}$ is $B^r$-regular with respect to the line bundle $B$, where $L=B^r$.

Now consider the multigraded setting. Suppose that $\mathcal{F}$ is $L=B^{\vec{r}}$-regular with respect to the fixed line bundles $B_1,\ldots,B_t$, where $\vec{r}=(r_1,\ldots,r_t)\in\mathbb{Z}^t$. Then
\[
H^i(\mathcal{F}\otimes B^{\vec{r}}\otimes B^{-\vec{u}})=0
\]
for all $i>0$, where $\vec{u}=(u_1,\ldots,u_t)\in\mathbb{N}^t$ satisfies $|\vec{u}|=i$. Equivalently,
\[
H^i\!\left(\mathcal{F}\otimes B_1^{\,r_1-u_1}\otimes\cdots\otimes B_t^{\,r_t-u_t}\right)=0
\]
for all $i>0$ with $|\vec{u}|=u_1+\cdots+u_t=i$. Using our notation, this vanishing can be written more compactly as
\[
H^i(\mathcal{F}\otimes B^{\vec{r}-\vec{u}})=0
\]
for all $i>0$ and all $\vec{u}$ satisfying $|\vec{u}|=i$.

When $t=1$, this recovers the notion of regularity given in Definition \ref{Definition_3.2}. We will see some examples of multigraded regularity in the next section.

\subsection{$\mathbf{(N_p)}$ \textbf{-Property of a Line Bundle:}}

Property $(N_p)$ was introduced to study the equations defining a projective variety and the syzygies among those equations.

Let $L$ be a globally generated line bundle on a projective variety $X$. Then $L$ determines a morphism
\[
\varphi_{L}:X\longrightarrow\mathbb{P}(H^0(X,L)).
\]
Let
\[
S=\operatorname{Sym}^{\bullet}H^0(X,L)
\]
be the symmetric algebra on $H^0(X,L)$. Thus $S$ is the homogeneous coordinate ring of the projective space $\mathbb{P}(H^0(X,L))=\mathbb{P}^r$, where $h^0(X,L)=r+1$.

Consider the section ring
\[
R=\bigoplus_{m\ge0}H^0(X,L^{\otimes m})
\]
associated to $L$. Then $R$ is naturally an $S$-graded module. Let
$E_{\bullet}=E_{\bullet}(L)$ denote the minimal graded free resolution of the $S$-graded module $R$ in (\ref{Equation_3.1}):

\begin{equation}
\label{Equation_3.1}
\begin{tikzcd}
\cdots \ar[r] & E_{i+1} \ar[r] & E_i \ar[r] & \cdots \ar[r] & E_1 \ar[r] & E_0 \ar[r] & R \ar[r] & 0.
\end{tikzcd}
\end{equation}

Each $E_i$ is a direct sum of twists of $S$:
\[
E_i=\bigoplus_{j\ge1}S(-i-j)^{\beta_{i,i+j}}.
\]
We refer to the summand $S(-i-j)^{\beta_{i,i+j}}$ as the \textbf{weight $j$ syzygy} of $L$ at the \textbf{$i$th position}.

\vspace{\baselineskip}

We say that the line bundle $L$ satisfies \textbf{Property} $\mathbf{(N_p)}$ for $p\in\mathbb{N}$ if
\[
E_0=S
\quad\text{and}\quad
E_i=\bigoplus S(-i-1)
\quad\text{for all }1\le i\le p.
\]

Thus, $L$ satisfies Property $(N_p)$ precisely when the natural map $S\to R$ is surjective and only the \textbf{weight one syzygies} occur in the modules $E_1,\ldots,E_p$. Equivalently, the homogeneous ideal of $\varphi_L(X)$ is generated by quadrics, and the first $p$ steps of its minimal graded free resolution are linear.

When $L$ is ample and globally generated, Property $(N_0)$ implies that $L$ is very ample. Furthermore, if $X$ is a smooth projective variety, then $L$ satisfies Property $(N_0)$ if and only if the embedding of $X$ by $L$ is projectively normal. This characterization need not hold for singular varieties.

\subsection{Vanishing of Weight One Syzygies or $\mathbf{(M_q)}$ property of a Line Bundle:} The $(M_q)$ property was first introduced by M. Green and R. Lazarsfeld in \cite{GreenLazarsfeldI}, in attainment to give a complete picture of a minimal graded resolution of a line bundle $L$. Complementary to the $(N_p)$ property, which addresses what happens to the beginning of the resolution, $(M_q)$ property deals with the end of a minimal resolution. \par 

Using the notation established in the previous subsection, we say that a globally generated line bundle $L$ on a projective variety $X$ satisfies \textbf{Property} $\mathbf{(M_q)}$ if the \textbf{weight one syzygy} vanishes at the \textbf{$i$th position} $E_i$ for every $i\ge r-q$ in the minimal graded free resolution (\ref{Equation_3.1}). Equivalently,
\[
E_i=\bigoplus_{j\ge2}S(-i-j)^{\beta_{i,i+j}}
\]
for all $i\ge r-q$.

If $L$ is very ample and embeds $X$ as a projectively normal variety in $\mathbb{P}^r$, then Property $(M_q)$ provides direct information about the syzygies of $X$, equivalently, about the minimal graded free resolution of the homogeneous coordinate ring $S_{X/\mathbb{P}^r}$.

\subsection{Relation between Koszul Cohomology and Property $(M_q)$:}

Let $L$ be a globally generated line bundle on a projective variety $X$. The \textit{kernel bundle}, denoted by $M_L$, is defined as the kernel of the evaluation map associated to $L$:
\begin{center}
\[
M_L:=\ker\!\left(H^0(X,L)\otimes\mathcal{O}_X\longrightarrow L\right).
\]
\end{center}

Since $L$ is a line bundle, $\operatorname{rank}(M_L)=h^0(X,L)-1=r$. Let $E$ be a vector bundle on $X$. From the short exact sequence
\begin{equation}
\label{Equation_3.2}
\begin{tikzcd}
0 \arrow[r] & M_L \arrow[r] & {H^0(X,L)\otimes \mathcal{O}_X} \arrow[r] & L \arrow[r] & 0
\end{tikzcd}
\end{equation}
we obtain the induced short exact sequence on exterior powers:
\[
\begin{tikzcd}
0 \arrow[r] & \wedge^{p+1} M_L \arrow[r] &
\wedge^{p+1}(H^0(X,L)\otimes\mathcal{O}_X)
\arrow[r] &
\wedge^{p}M_L\otimes L
\arrow[r] &
0.
\end{tikzcd}
\]

Tensoring the above sequence with the vector bundle $L^{q-1}\otimes E$ and passing to the associated long exact sequence in cohomology, we define the \textbf{twisted Koszul cohomology groups} by
\[
K_{p,q}(X,E;L):=
\ker\!\left(
H^1(X,\wedge^{p+1}M_L\otimes L^{q-1}\otimes E)
\xrightarrow{\varphi_{p,q,E,L}}
\wedge^{p+1}H^0(X,L)\otimes H^1(X,L^{q-1}\otimes E)
\right),
\]
where $p,q$ are nonnegative integers.

The term \emph{twisted} refers to the twist by the vector bundle $E$. When $E\simeq\mathcal{O}_X$, that is, in the untwisted case, we simply write
\[
K_{p,q}(X,L):=K_{p,q}(X,\mathcal{O}_X;L),
\]
and refer to these groups as the \textbf{Koszul cohomology groups}.

The Betti table associated to the minimal graded free resolution of an ample and globally generated line bundle $L$ on $X$ exhibits a direct relationship between the Koszul cohomology groups and the graded Betti numbers:
\[
\dim K_{p,q}(X,L)=\beta_{p,p+q}.
\]
Consequently, a line bundle $L$ satisfies Property $(M_q)$ if and only if
\[
\beta_{p,p+1}=0
\quad\text{for all }p\ge r-q,
\]
or equivalently,
\[
K_{p,1}(X,L)=0
\quad\text{for all }p\ge r-q.
\]

\section{\textbf{ Multigraded Regularity on Products of Projective Spaces}} 
\label{section_4_multireg_product_projectivespace}

Now we present some examples of multigraded regularity. The first natural example is the product
$\mathbb{P}^{m}\times\mathbb{P}^{n}$. Since
\[
\mathrm{Pic}(\mathbb{P}^{m})\cong\mathbb{Z},
\]
the Picard group is generated by the single base point free line bundle
$\mathcal{O}_{\mathbb{P}^{m}}(1)$. Thus, to obtain a variety whose Picard group has more than one base point free generator, it is natural to consider the product
$\mathbb{P}^{m}\times\mathbb{P}^{n}$: where $\mathrm{Pic}(\mathbb{P}^{m}\times\mathbb{P}^{n})=\mathbb{Z}\oplus \mathbb{Z}$ and the two base point free generators are pullbacks of $\mathcal{O}_{\mathbb{P}^{m}}(1)$ and $\mathcal{O}_{\mathbb{P}^{n}}(1)$ under projections. 

The multigraded regularity region for line bundles on
$\mathbb{P}^{m}\times\mathbb{P}^{n}$ can be computed explicitly, as shown below. Although multigraded Castelnuovo--Mumford regularity has been extensively studied in the literature (see, for example, \cite{MaclaganSmith}), this explicit computation of the minimal multigraded regularity region for the line bundle $B^{(a,b)}$ does not appear to have been recorded in this form.

\begin{example}
\label{PmxPn_example}
Let $X = \mathbb{P}^{m} \times \mathbb{P}^{n}$ with projections 
$p_1 : X \to \mathbb{P}^{m}$ and $p_2 : X \to \mathbb{P}^{n}$, and define
\[
B_1 = p_1^*(\mathcal{O}_{\mathbb{P}^{m}}(1)), \quad 
B_2 = p_2^*(\mathcal{O}_{\mathbb{P}^{n}}(1)), \quad B=(B_1,B_2).
\]

For $(a,b)\in \mathbb{Z}^2$, let $L = B^{(a,b)} = B_1^a \otimes B_2^b$.  
As in the previous example, we compute the minimal multigraded regularities 
of $L$ with respect to $B$.

Since $\dim(X) = m+n$, if $L$ is $B^{(r,s)}$–regular, then for every 
$(u,v)\in\mathbb{N}^2$ with $u+v = i$ and $1 \le i \le m+n$, we have
\[
H^i(L \otimes B^{(r,s)} \otimes B^{(-u,-v)}) = 0.
\]

As before, this is equivalent to the vanishing
\[
H^i(\mathcal{O}_{\mathbb{P}^{m}\times \mathbb{P}^{n}}(c,d)) = 0, 
\quad \text{where } c = a+r-u, \ d = b+s-v.
\]

By the K\"unneth formula,
\[
H^i(\mathcal{O}_{\mathbb{P}^{m}\times \mathbb{P}^{n}}(c,d))
= \bigoplus_{p+q=i} H^p(\mathcal{O}_{\mathbb{P}^{m}}(c)) \otimes 
H^q(\mathcal{O}_{\mathbb{P}^{n}}(d)).
\]

Crucially, only the cohomology groups corresponding to $i = m+n,\, n,\, m$ can be non-vanishing; all other groups vanish by the virtue of the K\"unneth formula, as it accounts for all partitions of $i$ into two non-negative integers.  
This gives rise to three conditions that must be simultaneously satisfied to ensure vanishing for all $i$.

\paragraph{\textbf{Vanishing of $H^{m+n}$}}
Since $H^i(\mathcal{O}_{\mathbb{P}^{r}}(t)) = 0$ for $1 \le i \le r-1$, 
the K\"unneth formula yields
\[
H^{m+n}(\mathcal{O}_{\mathbb{P}^{m} \times \mathbb{P}^{n}}(c,d))
= H^m(\mathcal{O}_{\mathbb{P}^{m}}(c)) \otimes H^n(\mathcal{O}_{\mathbb{P}^{n}}(d)).
\]
By Serre duality, $H^r(\mathcal{O}_{\mathbb{P}^{r}}(t)) = 0$ if and only if $t > -r-1$.  
It follows that
\[
H^{m+n}(\mathcal{O}_{\mathbb{P}^{m} \times \mathbb{P}^{n}}(c,d)) = 0 
\quad \text{if and only if} \quad 
c > -m-1 \ \text{or} \ d > -n-1.
\]
Applying $c = a+r-u$, $d = b+s-v$ with $u+v = m+n$, we obtain
\[
r > -a + n - 1 \quad \text{or} \quad s > -b + m - 1.
\]

\paragraph{\textbf{Vanishing of $H^{n}$}}
By the K\"unneth formula,
\[
H^{n}(\mathcal{O}_{\mathbb{P}^{m} \times \mathbb{P}^{n}}(c,d))
= H^0(\mathcal{O}_{\mathbb{P}^{m}}(c)) \otimes H^n(\mathcal{O}_{\mathbb{P}^{n}}(d)),
\]
which vanishes if and only if
\[
c < 0 \ \text{or} \ d > -n-1.
\]
Substituting $c = a+r-u$, $d = b+s-v$ with $u+v = n$, we obtain
\[
r < -a \quad \text{or} \quad s > -b-1.
\]

\paragraph{\textbf{Vanishing of $H^{m}$}}
Similarly,
\[
H^{m}(\mathcal{O}_{\mathbb{P}^{m} \times \mathbb{P}^{n}}(c,d))
= H^m(\mathcal{O}_{\mathbb{P}^{m}}(c)) \otimes H^0(\mathcal{O}_{\mathbb{P}^{n}}(d)),
\]
which vanishes if and only if
\[
c > -m-1 \ \text{or} \ d < 0.
\]
Applying $c = a+r-u$, $d = b+s-v$ with $u+v = m$, we obtain
\[
r > -a-1 \quad \text{or} \quad s < -b.
\]

\paragraph{\textbf{Minimal regularities}}
Combining the three conditions above, the line bundle $L = B^{(a,b)}$ is $B^{(r,s)}$–regular if and only if
\[
\text{(i) } r > -a + n - 1 \ \text{and} \ s > -b - 1 
\quad \text{or} \quad
\text{(ii) } r > -a - 1 \ \text{and} \ s > -b + m - 1.
\]

Passing to the minimal integer values satisfying these inequalities, we obtain
\[
\bm{(r,s) \ge (-a+n,-b)} \quad \text{or} \quad \bm{(r,s) \ge (-a,-b+m)}.
\]

In particular, the minimal regularities of $L$ are
\[
\bm{(-a+n,-b)} \quad \text{and} \quad \bm{(-a,-b+m)}.
\]
\end{example}

   Below we demonstrate the multigraded regularity of $L = \mathcal{O}_{\mathbb{P}^{1} \times \mathbb{P}^{1}}(2,3)$ with respect to $B$ as defined before, via the following lattice points of the shaded convex set.

\begin{example}[Lattice picture for minimal regularity]
Consider the line bundle
\[
L = \mathcal{O}_{\mathbb{P}^{1} \times \mathbb{P}^{1}}(2,3)
\]
with respect to $B=(B_1,B_2)$. The minimal $B$-regularity vectors are
\[
(-2,-2) \quad \text{and} \quad (-1,-3).
\]

We visualize the multigraded regularity using lattice points in $\mathbb{Z}^2$. 
The shaded convex region represents all $(r,s)$ such that $L$ is $B^{(r,s)}$-regular. 
The green point marks the minimal vector $(-2,-3)$, while the red point marks 
the second minimal vector $(-1,-2)$.

\begin{small}
\begin{figure}[ht!]
\begin{center}
\begin{tikzpicture}[scale=0.7]

\fill[gray!30!] (-1,-3) rectangle (5,5); 
\fill[gray!30!] (-2,-2) rectangle (5,5); 

\draw[thick,->] (-3,0) -- (6,0) node[right] {$r$};
\draw[thick,->] (0,-4) -- (0,6) node[above] {$s$};

\draw[thick, red] (-1,-3) -- (-1,5); 
\draw[thick, red] (-2,-2) -- (-2,5); 
\draw[thick, red] (-1,-3) -- (6,-3); 
\draw[thick, red] (-2,-2) -- (6,-2); 

\foreach \x in {-2,-1,0,1,2,3,4,5} {
    \foreach \y in {-3,-2,-1,0,1,2,3,4,5} {
        \filldraw[color=black] (\x,\y) circle (2pt);
    }
}

\foreach \y in {-2,-1,0,1,2,3,4,5} {
    \filldraw[color=orange] (-1,\y) circle (2pt); 
    \filldraw[color=orange] (-2,\y) circle (2pt); 
}
\foreach \x in {-2,-1,0,1,2,3,4,5} {
    \filldraw[color=orange] (\x,-3) circle (2pt); 
    \filldraw[color=orange] (\x,-2) circle (2pt); 
}

\filldraw[green] (-1,-3) circle (4pt)
    node[below, scale=0.9] {$(-1,-3)$};

\filldraw[blue] (-2,-2) circle (4pt)
    node[above left, scale=0.9] {$(-2,-2)$};

\filldraw[red] (-2,-3) circle (4pt);

\draw[color=red,->] (-1,-3.1) -- (-1,6);
\draw[color=red,->] (-1.1,-3) -- (6,-3);
\draw[color=red,->] (-2,-2.1) -- (-2,6);
\draw[color=red,->] (-2.1,-2) -- (6,-2);

\node[scale=0.9] at (6.2,-3.2) {$s=-3$};
\node[scale=0.9] at (-1,6.25) {$r=-1$};
\node[scale=0.9] at (6.2,-2.2) {$s=-2$};
\node[scale=0.9] at (-2.5,6.25) {$r=-2$};

\end{tikzpicture}
\end{center}
\caption{Multigraded regularity lattice for $\mathcal{O}_{\mathbb{P}^{1} \times \mathbb{P}^{1}}(2,3)$. The green and blue points indicate minimal regularities $(-1,-3)$ and $(-2,-2)$. Bold red lines show axes corresponding to these minimal regularities, orange points represent admissible points along these axes, and gray shading indicates all regularity vectors, the red bold point is an example of a non-admissible point.}
\label{regularityP1xP1_refined}
\end{figure}
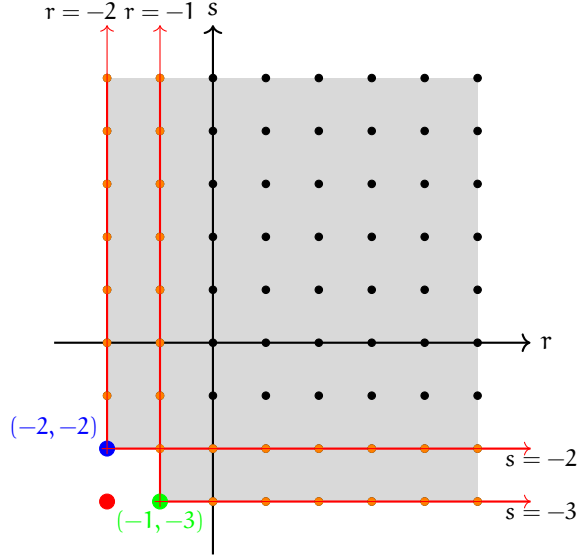
\end{small}
\end{example}

The preceding example already exhibits the essential features of 
multigraded regularity on a product of two projective spaces. 
However, the structure becomes significantly richer when one 
passes to products with more factors. In particular, while the two-factor case leads to a small number 
of vanishing conditions arising from the K\"unneth decomposition, 
the situation for three factors demonstrates the explicit pattern for minimal multigraded regularities. 

We illustrate this phenomenon in the case of a product of three 
projective spaces.

\begin{example}
\label{PmxPnxPl-example}
Consider 
\[
X=\mathbb{P}^{m} \times \mathbb{P}^{n} \times \mathbb{P}^{\ell},
\]
and let $B_1,B_2,B_3$ denote the pullbacks of 
$\mathcal{O}_{\mathbb{P}^{m}}(1)$, 
$\mathcal{O}_{\mathbb{P}^{n}}(1)$, and 
$\mathcal{O}_{\mathbb{P}^{\ell}}(1)$, respectively. 
For $(a,b,c)\in \mathbb{Z}^{3}$, set $L=B^{(a,b,c)}$.

We outline the computation of the minimal 
$(r,s,t)\in \mathbb{Z}^{3}$ such that $L$ is 
$B^{(r,s,t)}$-regular. Equivalently, one requires the vanishing
\[
H^i\big(L \otimes B^{(r,s,t)} \otimes B^{(-u,-v,-w)}\big)=0
\]
for all $(u,v,w)\in \mathbb{N}^3$ with $u+v+w=i$ and for all $1\le i \le m+n+l$.
Combining these conditions, one obtains a finite collection 
of $6$ minimal regularity vectors. More precisely, the minimal 
$(r,s,t)$ for which $L$ is $B^{(r,s,t)}$-regular are given by

\[
\begin{aligned}
&(-a+n+\ell,-b+\ell,-c), \quad 
(-a+n+\ell,-b,-c+n), \\
&(-a+\ell,-b+m+\ell,-c), \quad 
(-a,-b+m+\ell,-c+m), \\
&(-a,-b+m,-c+m+n), \quad 
(-a+n,-b,-c+m+n).
\end{aligned}
 \]
\end{example}

It is clear that as the number of factors increases, 
the combinatorics of partitions of the cohomological index becomes 
progressively more intricate, leading to a larger collection of vanishing 
conditions and minimal regularity vectors. These examples indicate the existence of a uniform pattern describing 
the minimal multigraded regularity for arbitrary products of projective 
spaces. We now formulate this in general. 

\begin{theorem}
\label{theorem:minimal-regularity}
Let 
\[
X=\mathbb{P}^{m_1}\times \mathbb{P}^{m_2}\times \cdots \times \mathbb{P}^{m_n},
\]
and let $B_1,\dots,B_n$ denote the pullbacks of 
$\mathcal{O}_{\mathbb{P}^{m_i}}(1)$ under the natural projections. 
For $\vec{a}=(a_1,\dots,a_n)\in \mathbb{Z}^n$, set 
\[
L=B^{\vec{a}}=B_1^{a_1}\otimes \cdots \otimes B_n^{a_n}.
\]

Then the minimal $\vec{r}=(r_1,\dots,r_n)\in \mathbb{Z}^n$ for which 
$L$ is $B^{\vec{r}}$-regular (with respect to $B=(B_1,\dots,B_n)$) 
are indexed by permutations of $\{1,\dots,n\}$. 

More precisely, for each permutation $\sigma\in S_n$, the corresponding 
minimal regularity vector $\vec{r}^{(\sigma)}$ is given by
\begin{equation}
\label{min_reg_product_proj_space}
    r^{(\sigma)}_{\sigma(i)}
=
-a_{\sigma(i)} + \sum_{j>i} m_{\sigma(j)},
\qquad \text{for } i=1,\dots,n.
\end{equation}
In particular, there are exactly $n!$ minimal regularity vectors.
\end{theorem}

\begin{proof}
By Definition~\ref{Definition_3.4}, the line bundle $L$ is $B^{\vec{r}}$-regular if and only if 
\[
H^k\left(L \otimes B^{\vec{r}} \otimes B^{-\vec{u}}\right) = 0
\]
for every $\vec{u} = (u_1, \ldots, u_n) \in \mathbb{N}^n$ with $|\vec{u}| = k$, where $1 \le k \le \dim X = \sum_{i=1}^{n} m_i$. Setting $c_i = a_i + r_i - u_i$, this is equivalent to requiring that 
\[
H^k\left(X, \mathcal{O}_X(c_1, \ldots, c_n)\right) = 0
\]
for all admissible $k$ and $\vec{u}$. Applying the K\"unneth formula yields
\[
H^k(X, \mathcal{O}_X(c_1, \ldots, c_n)) = \bigoplus_{\varepsilon_1 + \cdots + \varepsilon_n = k} \bigotimes_{i=1}^{n} H^{\varepsilon_i}(\mathbb{P}^{m_i}, \mathcal{O}_{\mathbb{P}^{m_i}}(c_i)).
\]

Since
\[
H^{q}(\mathbb P^{m_i},\mathcal O_{\mathbb P^{m_i}}(c_i))=0
\]
for \(0<q<m_i\), the only possible nonzero contributions in the
K\"unneth decomposition occur when each
\(\varepsilon_i\) is either \(0\) or \(m_i\).
Therefore, the possible nonvanishing degrees are precisely of the form
\[
k=\sum_{j\in J}m_j,
\]
where \(J\subseteq\{1,\ldots,n\}\) is nonempty.

We analyze the corresponding vanishing conditions according to the number of factors contributing top cohomology.

\paragraph{\textbf{The case of $n$ factors ($k = \sum_{i=1}^n m_i:$  { $|J|=n$)}}}
This corresponds to the top cohomological degree where $\varepsilon_i = m_i$ for all $i$. By Serre duality, the tensor product vanishes if and only if $c_i > -m_i - 1$ for at least one index $i$. Considering the extremal case \(u_i=m_i\), this provides $\binom{n}{n} = 1$ compound condition consisting of $n$ alternative inequalities:
\[
r_1 \ge \sum_{i=2}^n m_i - a_1 \quad \text{or} \quad \cdots \quad \text{or} \quad r_n \ge \sum_{j=1}^{n-1} m_j - a_n.
\]

In a compact form the above condition can be expressed as: 

\[
\bigcup_{i=1}^n \left( r_i \ge \sum_{j \neq i} m_j - a_i \right).
\]

\paragraph{\textbf{The case of $n-1$ factors ($k = \sum_{j \neq i} m_j:$ {$|J|=n-1$)}}}
Here, exactly one factor has $\varepsilon_i = 0$ and the remaining $n-1$ factors have $\varepsilon_j = m_j$. For each chosen index $i$, the summand vanishes if $c_i < 0$ or if $c_j > -m_j - 1$ for some $j \neq i$. Since the vanishing conditions must hold for every admissible
\(\vec{u} \), it is enough to test the extremal choice that makes the
inequalities weakest. For each $i \in \{1, \dots, n\}$, setting $J = \{1, \dots, n\} \setminus \{i\}$ yields $\binom{n}{n-1} = n$ conditions of the form:
\[
r_i \le -a_i - 1 \quad \text{or} \quad \left( \text{some } r_j \ge \sum_{k \neq i, j} m_k - a_j \text{ for } j \neq i \right).
\]
In a compact form it is therefore the following region:
\[
\left( r_i \le -a_i - 1 \right) \quad \cup \quad \bigcup_{j \neq i} \left( r_j \ge \sum_{k \neq i, j} m_k - a_j \right).
\]

\paragraph{\textbf{The general intermediate case of $p$ factors ($k = \sum_{j \in J} m_j$ with $|J| = p$):}}
For a fixed subset \(J\subseteq\{1,\ldots,n\}\) with \(|J|=p\),
the corresponding K\"unneth summand is
\[
\bigotimes_{j\in J}
H^{m_j}(\mathbb P^{m_j},\mathcal O(c_j))
\otimes
\bigotimes_{i\notin J}
H^0(\mathbb P^{m_i},\mathcal O(c_i)).
\]
Hence this summand vanishes if and only if either $c_j>-m_j-1$ for some \(j\in J\), or $c_i<0$ for some \(i\notin J\).
Evaluating at the extremal choice $u_j = m_j$ for $j \in J$ and $u_i = 0$ for $i \notin J$ yields a total of $\binom{n}{p}$ combinatorial conditions.
Precisely, for each subset $J$ of size $p$, maximizing the shifts over $u$ yields $\binom{n}{p}$ conditions:
\[
\bigcup_{j \in J} \left( r_j \ge \sum_{k \notin J} m_k - a_j \right) \quad \cup \quad \bigcup_{i \notin J} \left( r_i \le -a_i - 1 \right).
\]

\paragraph{\textbf{The case of $1$ factor ($k = m_i$:  $|J|=1)$}}
In the final step, only a single factor satisfies $\varepsilon_i = m_i$ while all other $\varepsilon_j = 0$. This yields $\binom{n}{1} = n$ final constraints. Setting $J = \{i\}$ for each $i \in \{1, \dots, n\}$ yields the final $\binom{n}{1} = n$ baseline conditions one for each $i$:
\[
\left( r_i \ge -a_i \right) \quad \cup \quad \bigcup_{j \neq i} \left( r_j \le -a_j - 1 \right).
\]

Combining these steps, the conditions that force the vanishing of all cohomologies are given by the simultaneous satisfaction of these  $\sum_{p=1}^n \binom{n}{p} = 2^n - 1$ subset conditions. 
To show that the minimal regularity vectors are indexed by permutations, we construct them inductively from this system of subset inequalities.  Let $m = m_{1}+\cdots+m_{n},$ 
$$H^{\ast} = \bigoplus_{1 \leq k \leq m}\bigoplus_{|\vec{u}| = k} H^{k}(X, L \otimes B^{\vec{r}} \otimes B^{-\vec{u}})$$ $$H^{m_{i}} = H^{m_{i}}(\mathbb{P}^{m_{i}}, \mathcal{O}_{\mathbb{P}^{m_{i}}}(a_{i}+r_{i}-u_{i}))$$
Consider the following Kunneth decomposition corresponding to different values of $|J| = n, n-1, n-2, \cdots$
\begin{small}
\begin{align*}
    H^{\ast} = H^{m}\oplus\bigoplus_{1 \leq i \leq n}H^{m_{1},\cdots,\widehat{m}_{i},\cdots,m_{n}}\oplus\bigoplus_{1 \leq i< j \leq n}H^{m_{1},\cdots,\widehat{m}_{i}, \cdots ,\widehat{m}_{j},\cdots,m_{n}}\oplus \cdots\\
    H^{m} = H^{m_{1}} \otimes \cdots \otimes H^{m_{i}} \otimes \cdots \otimes H^{m_{n}} \hspace{0.1cm} \text{ for } |J| = n \\
    H^{m_{1},\cdots,\widehat{m}_{i},\cdots,m_{n}} = \bigoplus_{|\vec{u}| = m-m_{i}} H^{m_{1}} \otimes \cdots \otimes H^{\widehat{m}_{i}} \otimes \cdots \otimes H^{m_{n}} \hspace{0.1cm}  \text{ for } |J| = n-1 \\
    H^{m_{1},\cdots,\widehat{m}_{i}, \cdots ,\widehat{m}_{j},\cdots,m_{n}} = \bigoplus_{|\vec{u}| = m-m_{i}-m_{j}}H^{m_{1}} \otimes \cdots \otimes H^{\widehat{m}_{i}} \otimes \cdots \otimes H^{\widehat{m}_{j}} \otimes \cdots \otimes H^{m_{n}} \hspace{0.1cm}  \text{ for } |J| = n-2 \\
    \vdots \hspace{5cm} \ddots \hspace{5cm} \vdots
\end{align*}
\end{small}
where the hat above the indices  means the index is zero, for example $\widehat{m}_{i} = 0$. \par
Now, without loss of generality, suppose $\sigma(1)$ in $\left\{1, 2, \cdots, n \right\}$ be such an integer for which we choose the condition
\begin{equation}
    \label{dummy_condition_|J|=n}
    r_{\sigma(1)} \ge \sum_{k \neq 1}m_k-a_{\sigma(1)}.
\end{equation}
when $|J| = n$. This condition is the condition for $H^{m_{\sigma(1)}} = 0.$ Now, due to this choice when $|J| = n-1,$ we get vanishing cohomologies $H^{m_{1}+\cdots+\widehat{m}_{i}+\cdots+m_{k}} = H^{m_{1}} \otimes \cdots \otimes H^{m_{i}} \otimes \cdots \otimes H^{m_{n}} = 0$ whenever $i \neq \sigma(1).$ Thus, having made the choice (\ref{dummy_condition_|J|=n}), there are only $n-1$ non-redundant choices of conditions  left in the case $|J| = n-1.$ Similarly, again without loss of generality if $\sigma(2)$ be such that we choose $r_{\sigma(2)} > \sum_{k \neq 1, 2}m_{k}-a_{\sigma(2)},$ then $H^{m_{1},\cdots,\widehat{m}_{i}, \cdots ,\widehat{m}_{j},\cdots,m_{n}} = 0$ whenever $(i, j) \neq (1, 2)$ and so we have only $n-2$ choices left when $|J| = n-2$ and so on. \par
Continuing in this manner successively determines an ordered sequence $$(\sigma(1),\sigma(2),\ldots,\sigma(n))$$ 
Thus, we at once see that $\sigma$ is a permutation of the indices and by our count of the number of choices of the conditions for each $|J|,$ we see that there are total $$n(n-1)(n-2) \cdots 3 
\cdot 2 \cdot 1 = n!$$
choices of such $\sigma$ and hence for every permutation corresponds to a unique set of conditions.
Consequently, the minimal regularity vectors are precisely those indexed by permutations, namely
\[
r^{(\sigma)}_{\sigma(i)} = -a_{\sigma(i)} + \sum_{j>i}m_{\sigma(j)}, \qquad i=1,\ldots,n.
\]

\end{proof}

\begin{remark}
\label{remark_geometry-of-minimal-regularity}
For a fixed permutation $\sigma\in S_n$, consider the partial products
\[
X_{\sigma,\ge i}
:=
\mathbb{P}^{m_{\sigma(i)}} \times \cdots \times \mathbb{P}^{m_{\sigma(n)}}
\subseteq X.
\]
Then
\[
\sum_{j>i} m_{\sigma(j)} = \dim X_{\sigma,>i},
\]
so that the expression appearing in the theorem can be rewritten as
\[
r_{\sigma(i)} = -a_{\sigma(i)} + \dim X_{\sigma,>i}.
\]

Equivalently, $\dim X_{\sigma,>i}$ is the codimension in $X$ of the complementary
subproduct
\[
\mathbb{P}^{m_{\sigma(1)}} \times \cdots \times \mathbb{P}^{m_{\sigma(i)}},
\]
so that each component of the minimal regularity vector measures the 
codimension of a natural partial subproduct determined by the ordering $\sigma$.

In this way, each permutation $\sigma$ determines a flag of subvarieties
\[
X_{\sigma,\ge 1} \supset X_{\sigma,\ge 2} \supset \cdots \supset X_{\sigma,\ge n},
\]
and the collection of minimal regularity vectors is naturally indexed by 
these flags. This provides a geometric stratification of the set of minimal 
regularities, reflecting the combinatorial structure of the symmetric group.
\end{remark}

\section{\textbf{ Multigraded Regularity of Hypersurfaces and Complete Intersections:}}
\label{section-5}

Let $X$ be a smooth projective variety and let $B = \left\{ B_{1}, \cdots, B_{t} \right\}$ be a set of basepoint free line bundles on $X$. For any vector $\vec{d} = (d_{1}, \cdots, d_{t}) > \vec{0}$ we call a smooth divisor $H$ in $|B^{\vec{d}}|$ a \textit{hypersurface of multidegree} $\vec{d}$ with respect to $B.$ \par
This definition generalizes the case of hypersurfaces in products of projective spaces. Namely, if $X = \prod^{n}_{k = 1}\mathbb{P}^{m_{k}}$ is a product of projective space and Let $B = (B_{1}, \cdots, B_{n})$ (so $t = n$) be the line bundles $B_{k} = p^{\ast}_{k}\mathcal{O}_{\mathbb{P}^{m_{k}}}(1), $ for any integer $k$ satisfying $1 \leq k \leq n.$ This set $B$ generates the Picard group of $X.$ We say that this $B$ is the \textit{standard basis} of $\mathrm{Pic}X.$  Then any $H \subset X$ a smooth hypersurface of $X$ has multidegree $\vec{d} = (d_{1}, \cdots, d_{n})$ with respect to this $B,$ this hypersurface is called a \textit{hypersurface of multi-degree} $\vec{d}$. For any $1 \leq j \leq n$ the multidegree of $B_{j}$ in the standard basis of $\mathrm{Pic}X$ is the standard basis vector $\vec{e}_{j}$ of $\mathbb{Z}^{n}.$ \par
The regularity and syzygies of hypersurfaces are particularly interesting since they usually provide examples of varieties of various Kodaira dimension. In case of product of projective space $X = \prod^{n}_{k = 1}\mathbb{P}^{m_{k}}$ ampleness is equivalent to having the inequality $d_{k} > 0$ for all integer $k$ satisfying $1 \leq k \leq n.$ 
Such a hypersurface is a Fano variety if $d_{k} \leq m_{k}+1,$ for all $1 \leq k \leq n$ with one strict inequalty. It is a Calabi-Yau variety if $d_{k} = m_{k}+1$ for all $1 \leq k \leq n$ and a variety of general type if $d_{k} \geq  m_{k}+1$ for all $1 \leq k \leq n$ with one inequality strict. When $n$ is high, a hypersurface can have various Kodaira dimensions for other multidegrees.
\par
Since all $B_{i}$ are globally generated the restrictions $B_{i, H}$ of $B_{i}$ on $H$ is globally generated and therefore we get a set $B_{H} = (B_{1, H}, \cdots, B_{t, H})$ of globally generated line bundles, with respect to which we can take multigraded regularity of line bundles on $H.$  \par

In this section, we compute the multigraded regularity of restrictions $F_{H},$ of vector bundles $F$ on $X,$ with respect to $B_{H}.$  \par
Using the above notation, we have the following theorem.
\begin{theorem}
\label{multireghypersurface}
    Let $E$ and $F$ be any vector bundles on $X.$ Let $F$ be $E-$regular with respect to $B$ and $H$ is an ample hypersurface on $X$ of multidegree $\vec{d}$ with respect to $B.$ Then, $F_{H}$ is $E_{H} \otimes B_{H}^{\vec{d}-\vec{e}_{j}}-$regular for all $1 \leq j \leq t.$ \par
    In particular, if $F$ is $B^{\vec{r}}-$regular, and if $\vec{r}_{H} = \vec{r}+\vec{d},$ then $F_{H}$ is $B_{H}^{\vec{r}_{H}-\vec{e}_{j}}-$regular.
\end{theorem}
\begin{proof}
We need to show that
\begin{equation}
    H^{i}(F_{H} \otimes E_{H} \otimes B_{H}^{\vec{s}} \otimes B_{H}^{-\vec{u}}) = 0
\end{equation}
for all $|\vec{u}| = i$ and $1 \leq i \leq \mathrm{dim}(X)-1,$ provided $\vec{s} \geq \vec{d}-\vec{e}_{j}$ for all $1 \leq j \leq t.$ \par
The ideal $\mathcal{I}_{H/X}$ of $H$ in $X$ is given by $\mathcal{I}_{Z/X} = B^{-\vec{d}}$ and we have the short-exact sequence
\begin{equation}
\label{hypersurface_ses}
    0 \rightarrow B^{-\vec{d}} \rightarrow \mathcal{O}_{X} \rightarrow \mathcal{O}_{H} \rightarrow 0
\end{equation}
Let $F_{H}$ be $B^{\vec{s}}-$regular. We fix some $i$ satisfying $1 \leq i \leq \mathrm{dim}(H) = \mathrm{dim}(X)-1$ and set 
$$M = F \otimes E \otimes B^{\vec{s}} \otimes B^{-\vec{u}},$$
where $\vec{u}$ is any vector with $|\vec{u}| = i.$ Then, tensoring the short exact sequence (\ref{hypersurface_ses}) by $M$ and taking long exact sequence of cohomology, we get
$$
    H^{i}(M) \rightarrow H^{i}(M_{H}) \rightarrow H^{i+1}(M \otimes B^{-\vec{d}})
$$
Thus, $H^{i}(M_{H}) = 0$ if we have both vanishings $H^{i}(M) = 0$ and $H^{i+1}(M \otimes B^{-\vec{d}}) = 0.$ \par
Since $F$ is $E-$regular, $H^{i}(M) = H^{i}(F \otimes E \otimes  B^{\vec{s}} \otimes \otimes B^{-\vec{u}}) = 0$ if $\vec{s} \geq \vec{0}$. The other vanishing again follows since $F$ is $E-$regular with respect to $B.$ Indeed,
\begin{equation}
\label{cohomology_hypersurface}
H^{i+1}(M \otimes B^{-\vec{d}}) = H^{i+1}( B^{\vec{s}-\vec{d}+\vec{e}_{j}} \otimes F \otimes E \otimes B^{-(\vec{u}+\vec{e}_{j})}) = 0
\end{equation}
since $|\vec{u}+\vec{e}_{j}| = i+1,$ $2 \leq i+1 \leq \mathrm{dim}(X)$ for any $j$ satisfying $1 \leq j \leq t$, provided
$$\vec{s}-\vec{d}+\vec{e}_{j} \geq \vec{0}, \hspace{0.5cm} \text{ for all } 1 \leq j \leq t$$
Since $d_{i} \geq 1,$ we have $\vec{d}-\vec{e}_{j} \geq \vec{0}$ for all $1 \leq j \leq t,$ comparing two inequalities we get $\vec{s} \geq \vec{d}-\vec{e}_{j}$ for all $j$ and therefore $F_{H}$ is $B^{\vec{r}-\vec{e}_{j}}-$regular for all $1 \leq j \leq t.$
\end{proof}
\begin{remark} $\text{ }$
\begin{itemize}
    \item[(a)] This Theorem indicates that upon restriction the multigraded regularity increases almost the same as the multidegree of the hypersurface. We can see that this $B^{\vec{m}-\vec{e}_{j}}-$regularity condition is similar set up as the regularity condition of our Main Theorem (\ref{general_theorem}).
    \item[(b)]  If $F$ is any line bundle on a product of projective space $X = \prod^{n}_{k = 1} \mathbb{P}^{m_{k}}$ and $B$ be the standard basis of $\mathrm{Pic}X,$ then $F$ is of the form $F = \mathcal{O}_{X}(\vec{a}) = B^{\vec{a}},$ then its minimum multigraded regularities are known from Theorem \ref{theorem:minimal-regularity}. If $\mathrm{dim}(X) \geq 4,$ then since $H$ is ample, by Grothendieck-Lefschetz theorem, these line bundles $F|_{H}$ generate $\mathrm{Pic}(H)$ since the line bundles $F = \mathcal{O}_{X}(\vec{a}) = B^{\vec{a}}$ generate $\mathrm{Pic}(X).$ So the theory we develop below gives the multigraded regularity of all line bundles on $H$ with respect to the multigraded regularity of line bundles on $X,$ when $\mathrm{dim}(X) \geq 4.$
\end{itemize}
\end{remark}
Next, mention a corollary of Theorem \ref{multireghypersurface} which generalizes the result to complete intersections.
\begin{corollary}
    Let $E$ and $F$ be a vector bundle on $X.$ Let $F$ be $E-$regular with respect to $B$. Let $Z = H_{1} \cap \cdots \cap H_{s}$ be a complete intersection of hypersurfaces $H_{\ell}$ of multi-degrees $\vec{d}_{\ell}$ with respect to $B$ such that $\vec{d}_{\ell} > 0$ for any $1 \leq \ell \leq s.$ \\
    Then $F_{Z}$ is $E_{Z} \otimes B^{\vec{d}-\vec{e}_{j}}-$regular for all $1 \leq j \leq t,$ where
    $$\vec{d} := \vec{d}_{1}+\cdots+\vec{d}_{s}$$
\end{corollary}
\begin{proof}
    Let $Z_{0} = X, Z_{\ell} := H_{1} \cap \cdots \cap H_{\ell}$ for any $1 \leq \ell \leq s$ and $\vec{z}_{\ell} := \sum^{\ell}_{v = 1}\vec{d}_{v}$ \par
    Then, we have a stratification of $X$
    $$Z = Z_{s} \subsetneq Z_{s-1} \subsetneq \cdots \subsetneq Z_{2} \subsetneq Z_{1} = H_{1} \subsetneq Z_{0} = X$$
    where each $Z_{\ell}$ is a hypersurface on $Z_{\ell-1}$ for any $1 \leq \ell \leq s$. The result is obtained by applying Theorem \ref{multireghypersurface} successively
\end{proof}
In the following example we show that our bound in Theorem \ref{multireghypersurface} is sharp.
\begin{example}
    Let $H \subset \mathbb{P}^{2}$ be a smooth cubic in $\mathbb{P}^{2}.$ Let $B =\mathcal{O}_{\mathbb{P}^{2}}(1).$ We compute the regularity of the restriction $L|_{H} = \mathcal{O}_{H}$ of the line bundle $L = \mathcal{O}_{\mathbb{P}^{2}}$ with respect to $B_{H},$ which is the same as the regularity of $B_{H}.$  Clearly $B_{H}$ has degree $3$ on $H$ and it is well-known that a degree $3$ line bundle on a elliptic curve has regularity $2.$ and $2 = (-1)+3-1 = \mathrm{reg}(\mathcal{O}_{\mathbb{P}^{2}}(1))+\mathrm{deg}(H)-1.$
\end{example}

\section{\textbf{Building Blocks}}
\label{section_6_buildingblocks}

This section develops the technical machinery that will be used in the proofs of the main results. We begin by recalling two fundamental ingredients from the literature. The first is a lemma due to Aprodu and Lombardi \cite{AproduLombardi}, which provides a sufficient criterion for property $(M_q)$. The second is a regularity lemma in  \cite{HeringSchenckSmith}, which plays crucial role in the study of multigraded regularity and serves as one of the principal tool for establishing some of the subsequent regularity results in this section. We develope a sequence of regularity lemmas and propositions describing the behavior of multigraded regularity under tensoring with syzygy bundles.

\par \vspace{\baselineskip}

We begin by recalling the following lemma, which will be used in the study of property $(M_q)$. 

\begin{lemma}[Lemma 2.1 in \cite{AproduLombardi}]
\label{LemmaAprodu}

Let $X$ be a smooth projective variety of dimension $n\ge 2$ and let $L$ be an ample and globally generated line bundle on $X$. Set $r=h^0(L)-1$, and let $n\le q\le r-1$. Let $B$ be any line bundle such that $L-B$ is ample. If the multiplication maps of global sections
\[
\mu_k:
H^0(L)\otimes H^0(M_L^{\otimes k}\otimes \omega_X\otimes L^{n-1}\otimes B^{-1})
\longrightarrow
H^0(M_L^{\otimes k}\otimes \omega_X\otimes L^n\otimes B^{-1})
\]
are surjective for $0\le k\le q-n$, then
\[
H^1(\wedge^{k+1}M_L\otimes \omega_X\otimes L^{n-1}\otimes B^{-1})=0.
\]

\end{lemma}

\begin{remark}
In the preceding lemma, the authors essentially proved the following isomorphism:
\[
H^1(X,\wedge^{k+1}M_L\otimes\omega_X\otimes L^{n-1}\otimes B^{-1})
\simeq
H^1(X,\wedge^{r-k-n+1}M_L\otimes B)
\]
for $0\le k\le q-n$. Thus, in the language of Koszul cohomology, it gives
\[
K_{k,n}(X,\omega_X\otimes B^{-1};L)
=
K_{r-k-n,1}(X,B;L)
=
0
\]
for $0\le k\le q-n$.
\end{remark}

\begin{observation}
\label{Observation_5.3}
It is worth observing that the lemma above generalizes the Duality Theorem of M. Green, namely Corollary~(2.c.10) in \cite{GreenI}. Green proved the following duality:
\begin{equation}
\label{GreenDuality}
K_{p,q}(X,\omega_X,L)^{*}\simeq K_{r-n-p,n+1-q}(X,L).
\end{equation}
for $q\ge n+1$. Taking $B=\mathcal{O}_X$, the lemma above shows that this duality continues to hold for $q=n$ and for $0\le p\le q-n$.
\end{observation}

\begin{remark}
By taking $B=\mathcal{O}_X$ in Lemma \ref{LemmaAprodu} and combining it with Green's $K_{p,1}$ Theorem [Theorem 3.c.1 in \cite{GreenI}], we conclude that $L$ satisfies property $(M_q)$ for $q\ge n\ge2$ if
\[
H^1(\wedge^{k+1}M_L\otimes\omega_X\otimes L^{n-1})=0
\]
for $0\le k\le q-n$, which, by the Green duality \eqref{GreenDuality}, is equivalent to
\[
K_{r-k-n,1}(X,L)=0
\]
for all $0\le k\le q-n$.
\end{remark}

We next recall the following lemma in the context of multigraded regularity, which is an observation due to Hering, Schenck, and Smith \cite{HeringSchenckSmith}.

\begin{lemma}[Regularity Lemma I]
\label{LemmaI}
Let
\[
\begin{tikzcd}
0 \arrow[r] & \mathcal{F}' \arrow[r] & \mathcal{F} \arrow[r] & \mathcal{F}'' \arrow[r] & 0
\end{tikzcd}
\]
be a short exact sequence of coherent sheaves on a projective variety $X$. If $\mathcal{F}$ is $E$-regular, the map
\[
H^0(X,\mathcal{F}\otimes E\otimes B^{-\vec{e_j}})
\longrightarrow
H^0(X,\mathcal{F}''\otimes E\otimes B^{-\vec{e_j}})
\]
is surjective for all $1\le j\le t$, and $\mathcal{F}''$ is $(E\otimes B^{-\vec{e_j}})$-regular for all $1\le j\le t$, then $\mathcal{F}'$ is $E$-regular.
\end{lemma}

Let $B_1,..,B_t$ be fixed globally generated line bundles on a projective variety $X$. Let $F$ and $E$ be two vector bundles on $X$ such that $F$ is $E$-regular and let $\vec{w}=(w_1,..,w_t)\in \mathbb{N}^t$. Then the vector bundle $M_{B^{\vec{w}}} \otimes F$ is $(B^{\vec{\delta}} \otimes E)$-regular where $B^{\vec{\delta}}:= B_1\otimes...\otimes B_t$.

We now establish a regularity lemma that will play a central role in the proof of the main theorem in the next section. 

\begin{lemma}[Regularity Lemma II]

\label{Lemma II}
Let $B_1,\ldots,B_t$ be fixed globally generated line bundles on a projective variety $X$. Let $E$ and $F$ be vector bundles on $X$, and assume that $F$ is $E$-regular. Let $\vec{w}=(w_1,\ldots,w_t)\in\mathbb{N}^t$. Then
\[
M_{B^{\vec{w}}}\otimes F
\]
is $(B^{\vec{\delta}}\otimes E)$-regular, where
\[
B^{\vec{\delta}}:=B_1\otimes\cdots\otimes B_t.
\]

\end{lemma}

\begin{proof}
For simplicity, we denote the line bundle $B^{\vec{w}}$ by $L$ in this proof. Since $L$ is globally generated, we have the short exact sequence
\begin{equation}
\label{Equation_5.1}
\begin{tikzcd}
0 \arrow[r] & M_L \arrow[r] & H^0(L)\otimes \mathcal{O}_X \arrow[r,"ev"] & L \arrow[r] & 0.
\end{tikzcd}
\end{equation}

Tensoring \eqref{Equation_5.1} with $B_1\otimes\cdots\otimes B_t \otimes F \otimes E \otimes B^{-\vec{u}}$ and taking cohomology, we obtain the long exact sequence
\begin{small}
\[
\begin{tikzcd}[column sep=0.35cm]
\cdots \arrow[r] &
H^{i-1}(L\otimes F\otimes E\otimes B_1\otimes\cdots\otimes B_t\otimes B^{-\vec{u}})
\arrow[r] &
H^i(M_L\otimes F\otimes E\otimes B_1\otimes\cdots\otimes B_t\otimes B^{-\vec{u}}) \\
\arrow[r] &
H^0(L)\otimes H^i(F\otimes E\otimes B_1\otimes\cdots\otimes B_t\otimes B^{-\vec{u}})
\arrow[r] &
H^i(L\otimes F\otimes E\otimes B_1\otimes\cdots\otimes B_t\otimes B^{-\vec{u}})
\arrow[r] &
\cdots
\end{tikzcd}
\]
\end{small}

We aim to show that
\[
H^i(M_L\otimes F\otimes E\otimes B_1\otimes\cdots\otimes B_t\otimes B^{-\vec{u}})=0
\]
for all $i>0$ and all $\vec{u}\in\mathbb{N}^t$ with $|\vec{u}|=i$.

By Theorem \ref{Theorem_3.7}(i),
\[
H^i(F\otimes E\otimes B_1\otimes\cdots\otimes B_t\otimes B^{-\vec{u}})=0
\quad \text{for all } i>0.
\]

Since $|\vec{u}|>0$, write $\vec{u}=\vec{u}'+\vec{e}_k$ for some $k$, where $\vec{u}'\in\mathbb{N}^t$ and $|\vec{u}'|=i-1$. 
 Then
\[
H^{i-1}(L\otimes F\otimes E\otimes B_1\otimes\cdots\otimes B_t
\otimes B^{-\vec{e}_k}\otimes B^{-\vec{u}'})=0
\]
for all $i>1$, again by Theorem \ref{Theorem_3.7}(i). Since
\[
B^{-\vec{e}_k}\otimes B^{-\vec{u}'}=B^{-\vec{u}},
\]
it follows that
\[
H^{i-1}(L\otimes F\otimes E\otimes B_1\otimes\cdots\otimes B_t
\otimes B^{-\vec{u}})=0.
\]
Hence, by exactness of the above long exact sequence,
\[
H^i(M_L\otimes F\otimes E\otimes B_1\otimes\cdots\otimes B_t
\otimes B^{-\vec{u}})=0
\]
for all $i>1$.

\[
H^{i-1}(L\otimes F\otimes E\otimes B_1\otimes\cdots\otimes B_t\otimes B^{-\vec{u}})
=0
\]
again by Theorem \ref{Theorem_3.7}(i). Hence, by exactness, we obtain
\[
H^i(M_L\otimes F\otimes E\otimes B_1\otimes\cdots\otimes B_t\otimes B^{-\vec{u}})=0
\quad \text{for all } i>1.
\]

It remains to prove
\[
H^1(M_L\otimes F\otimes E\otimes B_1\otimes\cdots\otimes B_t\otimes B^{-\vec{e}_i})=0
\quad \text{for } 1\le i\le t.
\]

For this, consider the relevant portion of the long exact sequence:
\begin{small}
\[
\begin{tikzcd}[column sep=0.4cm]
\cdots \arrow[r] &
H^0(L)\otimes H^0(F\otimes E\otimes B_1\otimes\cdots\otimes B_t\otimes B^{-\vec{e}_i})
\arrow[r,"f"] &
H^0(L\otimes F\otimes E\otimes B_1\otimes\cdots\otimes B_t\otimes B^{-\vec{e}_i}) \\
\arrow[r] &
H^1(M_L\otimes F\otimes E\otimes B_1\otimes\cdots\otimes B_t\otimes B^{-\vec{e}_i})
\arrow[r] &
H^0(L)\otimes H^1(F\otimes E\otimes B_1\otimes\cdots\otimes B_t\otimes B^{-\vec{e}_i})
\arrow[r] &
\cdots
\end{tikzcd}
\]
\end{small}

By Theorem \ref{Theorem_3.7}(i) and the $E$-regularity of $F$, the last term vanishes. Moreover, the map $f$ is surjective by Theorem \ref{Theorem_3.7}(ii) for all $1\le i\le t$. Hence,
\[
H^1(M_L\otimes F\otimes E\otimes B_1\otimes\cdots\otimes B_t\otimes B^{-\vec{e}_i})=0.
\]

This completes the proof.
\end{proof}

\begin{remark}
We denote by
\[
\vec{\delta}=(1,\ldots,1)
\]
the diagonal vector. We say that a vector $\vec{u}\in\mathbb{N}^t$ lies above the diagonal if $\vec{u}\ge\vec{\delta}$, that is, if $u_i\ge1$ for every $1\le i\le t$.

The preceding lemma shows that if $L=B^{\vec{w}}$ for some $\vec{w}\in\mathbb{N}^t$, then $M_L\otimes F$ is $(B^{\vec{\delta}}\otimes E)$-regular whenever $F$ is $E$-regular. In general, however, one cannot replace $B^{\vec{\delta}}\otimes E$ by $B^{\vec{w}}\otimes E$ when $\vec{w}<\vec{\delta}$. Thus, $B^{\vec{\delta}}\otimes E$ may be viewed as the minimal regularity shift guaranteed by the lemma.
\end{remark}

The following proposition is obtained by iterating Regularity Lemma II~\ref{Lemma II}.

\begin{proposition}[Regularity Proposition III]
\label{Proposition III}

Let $X$ be a projective variety, and let
$\vec{w}_1,\ldots,\vec{w}_q\in\mathbb{N}^t$
be arbitrary vectors. If $F$ is $E$-regular, then
\[
M_{B^{\vec{w}_1}}
\otimes
M_{B^{\vec{w}_2}
}
\otimes
\cdots
\otimes
M_{B^{\vec{w}_q}}
\otimes
F
\]
is
\[
(B^{\vec{\delta}})^{\otimes q}\otimes E
\]
-regular.

\end{proposition}

\begin{proof}

Set
\[
F_1=M_{B^{\vec{w}_1}}\otimes F
\quad\text{and}\quad
E_1=B^{\vec{\delta}}\otimes E.
\]
By Lemma~\ref{Lemma II}, $F_1$ is $E_1$-regular. Applying Lemma~\ref{Lemma II} again, we obtain that
\[
M_{B^{\vec{w}_2}}\otimes F_1
\]
is $B^{\vec{\delta}}\otimes E_1$-regular. Substituting the definitions of $F_1$ and $E_1$, it follows that
\[
M_{B^{\vec{w}_1}}\otimes M_{B^{\vec{w}_2}}\otimes F
\]
is
\[
(B^{\vec{\delta}})^{\otimes2}\otimes E
\]
-regular. Repeating this argument $q$ times proves the proposition.

\end{proof}

We next establish a refinement of Regularity Lemma II under hypotheses that are tailored to the proof of the main theorem.

\begin{lemma}[Regularity Lemma IV]
\label{LemmaIV}

Let $B^{\vec{w}}$ and $B^{\vec{m}}$ be line bundles on a projective variety $X$, where $\vec{w}\in\mathbb{N}^t$ and $\vec{m}\in\mathbb{Z}^t$. Let $F$ be a vector bundle on $X$ that is $B^{\vec{m}-\vec{e}_j}$-regular for every $1\le j\le t$. Then
\[
M_{B^{\vec{w}}}\otimes F
\]
is $B^{\vec{m}}$-regular. 

\end{lemma}

\begin{proof}
Set $L:=B^{\vec{w}}$. We apply Lemma~\ref{LemmaI}. Consider the short exact sequence
\[
\begin{tikzcd}
0 \arrow[r] &
M_L\otimes F \arrow[r] &
H^0(L)\otimes F \arrow[r] &
L\otimes F \arrow[r] &
0.
\end{tikzcd}
\]

We verify the hypotheses of Lemma~\ref{LemmaI}.

\smallskip

\noindent
(i) Since $F$ is $B^{\vec{m}-\vec{e}_j}$-regular, it follows that $H^0(L)\otimes F$ is $B^{\vec{m}}$-regular.

\smallskip

\noindent
(ii) For every $\vec{u}\in\mathbb{N}^t$ with $|\vec{u}|=i$,
\[
H^i(L\otimes F\otimes B^{\vec{m}-\vec{e}_j}\otimes B^{-\vec{u}})
=
H^i(F\otimes B^{\vec{m}-\vec{e}_j}\otimes L\otimes B^{-\vec{u}})
=0,
\]
by the $B^{\vec{m}-\vec{e}_j}$-regularity of $F$. Hence $L\otimes F$ is $B^{\vec{m}-\vec{e}_j}$-regular for every $1\le j\le t$.

\smallskip

\noindent
(iii) The multiplication maps
\[
H^0(L)\otimes H^0(F\otimes B^{\vec{m}-\vec{e}_j})
\longrightarrow
H^0(L\otimes F\otimes B^{\vec{m}-\vec{e}_j})
\]
are surjective for every $1\le j\le t$ by Theorem~\ref{Theorem_3.7}(ii), since $F$ is $B^{\vec{m}-\vec{e}_j}$-regular.

Therefore, Lemma~\ref{LemmaI} implies that $M_L\otimes F$ is $B^{\vec{m}}$-regular.
\end{proof}

Motivated by the ideas in \cite{HeringSchenckSmith}, we isolate the following result as a separate lemma because of its independent interest, particularly in the case where $F=B^{\vec{v}}$ with $\vec{v}\in\mathbb{N}^t$.

\begin{lemma}[Regularity Lemma V]
\label{LemmaV}

Let $B^{\vec{v}}$ and $B^{\vec{m}}$ be line bundles on $X$, where
$\vec{v},\vec{m}\in\mathbb{N}^t$, such that $B^{\vec{v}}$ is $B^{\vec{m}}$-regular and $\vec{m}\ge\vec{\delta}$. Let $L=B^{\vec{w}}$, where $\vec{w}\ge\vec{v}+\vec{m}$. Then $M_L\otimes B^{\vec{v}}$ is $B^{\vec{m}}$-regular.

\end{lemma}

\begin{proof}
Consider the short exact sequence
\[
\begin{tikzcd}
0 \arrow[r] &
M_L\otimes B^{\vec{v}} \arrow[r] &
H^0(L)\otimes B^{\vec{v}} \arrow[r] &
L\otimes B^{\vec{v}} \arrow[r] &
0.
\end{tikzcd}
\]

We verify the hypotheses of Lemma~\ref{LemmaI}.

\smallskip

\noindent
(i) Since $B^{\vec{v}}$ is $B^{\vec{m}}$-regular, it follows that
$H^0(L)\otimes B^{\vec{v}}$ is also $B^{\vec{m}}$-regular.

\smallskip

\noindent
(ii) For every $\vec{u}\in\mathbb{N}^t$ with $|\vec{u}|=i$,
\[
H^i(L\otimes B^{\vec{v}}\otimes B^{\vec{m}-\vec{e}_j}\otimes B^{-\vec{u}})
=
H^i(B^{\vec{v}}\otimes B^{\vec{m}}\otimes L\otimes B^{-\vec{e}_j}\otimes B^{-\vec{u}})
=0
\]
for all $i\ge1$, since $\vec{w}\ge\vec{m}\ge\vec{\delta}$. Hence,
$L\otimes B^{\vec{v}}$ is $B^{\vec{m}-\vec{e}_j}$-regular for every
$1\le j\le t$.

\smallskip

\noindent
(iii) Since $B^{\vec{v}}$ is $B^{\vec{m}}$-regular, Theorem~\ref{Theorem_3.7}(ii) implies that the multiplication maps
\[
H^0(B^{\vec{v}+\vec{m}+\vec{p}})
\otimes
H^0(B^{\vec{q}})
\longrightarrow
H^0(B^{\vec{v}+\vec{m}+\vec{p}+\vec{q}})
\]
are surjective for all $\vec{p},\vec{q}\in\mathbb{N}^t$. Since $\vec{m}\ge\vec{\delta}$, taking
\[
\vec{q}=\vec{v}+\vec{m}-\vec{e}_j,
\qquad
\vec{p}=\vec{w}-\vec{v}-\vec{m},
\]
we obtain that the maps
\[
H^0(L)\otimes H^0(B^{\vec{v}}\otimes B^{\vec{m}-\vec{e}_j})
\longrightarrow
H^0(L\otimes B^{\vec{v}}\otimes B^{\vec{m}-\vec{e}_j})
\]
are surjective for every $1\le j\le t$. Therefore, Lemma~\ref{LemmaI} implies that
\[
M_L\otimes B^{\vec{v}}
\]
is $B^{\vec{m}}$-regular.

\end{proof}

\section{\textbf{Main Vanishing Theorem and Applications}}
\label{section_7_mainbody_vanishingtheorem}

The results established in the previous section provide the technical framework required for the proofs of the main results. In this section, we combine these ingredients to establish our general vanishing theorem and its applications.

\begin{theorem}
\label{Theorem_6.1}
Let $X$ be a projective variety, let $\vec{m}\in\mathbb{Z}^t$, and let
$\vec{w}_1,\ldots,\vec{w}_{k+1}\in\mathbb{N}^t$.
Let $F$ be a vector bundle on $X$ such that $F$ is
$B^{\vec{m}-\vec{e}_j}$-regular for every $1\le j\le t$.
Then the following statements hold for all $k\ge 1$:
\begin{align*}
H^1\left(
M_{B^{\vec{w}_1}}\otimes\cdots\otimes M_{B^{\vec{w}_{k+1}}}
\otimes F\otimes B^{\vec{m}_k}
\right)=0,
\end{align*}
whenever $\vec{m}_k\ge \vec{m}+(k-1)\vec{\delta}$, and
\begin{align*}
H^i\left(
M_{B^{\vec{w}_1}}\otimes\cdots\otimes M_{B^{\vec{w}_{k+1}}}
\otimes F\otimes B^{\vec{m}_k}
\right)=0
\end{align*}
for all $i>0$ whenever $\vec{m}_k\ge \vec{m}+k\vec{\delta}$.
\end{theorem}

\begin{proof}
For brevity, let $L_i:=B^{\vec{w}_i}$ for $1\le i\le k+1$. We first prove that
\[
M_{L_1}\otimes\cdots\otimes M_{L_{k+1}}\otimes F
\]
is $B^{\vec{m}+k\vec{\delta}}$-regular. By Lemma \ref{LemmaIV},
$M_{L_1}\otimes F$ is $B^{\vec{m}}$-regular. Applying Proposition \ref{Proposition III}, it follows that
$M_{L_1}\otimes M_{L_2}\otimes F$ is $B^{\vec{m}+\vec{\delta}}$-regular. Repeating this argument inductively, we conclude that
$M_{L_1}\otimes\cdots\otimes M_{L_{k+1}}\otimes F$
is $B^{\vec{m}+k\vec{\delta}}$-regular. We now use this regularity to establish the desired vanishing.

To prove the vanishing, consider the short exact sequence
\[
\begin{tikzcd}
0 \arrow[r] &
M_{L_1} \arrow[r] &
H^0(L_1)\otimes \mathcal{O}_X \arrow[r] &
L_1 \arrow[r] &
0.
\end{tikzcd}
\]

Tensoring this sequence with
$M_{L_2}\otimes\cdots\otimes M_{L_{k+1}}\otimes F\otimes B^{\vec{m}_k}$
and passing to cohomology yields the following long exact sequence:

\begin{small}
\[
\begin{tikzcd}[row sep = 0.1cm, column sep = 0.4cm]
\arrow[r] &
H^0(L_1)\otimes H^0(M_{L_2}\otimes\cdots\otimes M_{L_{k+1}}\otimes F\otimes B^{\vec{m}_k})
\arrow[r,"\alpha"] &
H^0(L_1\otimes M_{L_2}\otimes\cdots\otimes M_{L_{k+1}}\otimes F\otimes B^{\vec{m}_k}) & \\
\arrow[r] &
H^1(M_{L_1}\otimes\cdots\otimes M_{L_{k+1}}\otimes F\otimes B^{\vec{m}_k})
\arrow[r] &
H^0(L_1)\otimes H^1(M_{L_2}\otimes\cdots\otimes M_{L_{k+1}}\otimes F\otimes B^{\vec{m}_k})
\arrow[r] &
\cdots
\end{tikzcd}
\]
\end{small}

Since
$M_{L_2}\otimes\cdots\otimes M_{L_{k+1}}\otimes F$
is $B^{\vec{m}+(k-1)\vec{\delta}}$-regular,
\[
H^1\!\left(M_{L_2}\otimes\cdots\otimes M_{L_{k+1}}
\otimes F\otimes B^{\vec{m}+(k-1)\vec{\delta}}
\otimes B^{-\vec{e}_j}\right)=0
\]
for all $1\le j\le t$. Therefore,
\[
H^1(M_{L_2}\otimes\cdots\otimes M_{L_{k+1}}
\otimes F\otimes B^{\vec{m}_k})=0
\]
whenever $\vec{m}_k\ge \vec{m}+(k-1)\vec{\delta}$. Furthermore, the
$B^{\vec{m}+(k-1)\vec{\delta}}$-regularity of
$M_{L_2}\otimes\cdots\otimes M_{L_{k+1}}\otimes F$
also implies the surjectivity of the map $\alpha$ for every
$\vec{m}_k\ge \vec{m}+(k-1)\vec{\delta}$. This proves the first vanishing statement.

Moreover,
$M_{L_1}\otimes\cdots\otimes M_{L_{k+1}}\otimes F$
is $B^{\vec{m}+k\vec{\delta}}$-regular. Hence, by Theorem
\ref{Theorem_3.7}(i),
\[
H^i(M_{L_1}\otimes\cdots\otimes M_{L_{k+1}}
\otimes F\otimes B^{\vec{m}+k\vec{\delta}})
=
H^i(M_{L_1}\otimes\cdots\otimes M_{L_{k+1}}
\otimes F\otimes B^{\vec{m}+k\vec{\delta}}
\otimes B^{\vec{u}}\otimes B^{-\vec{u}})
=0
\]
for all $i>0$ and for every $\vec{u}\in\mathbb{N}^t$ satisfying $|\vec{u}|=i$. This completes the proof.
\end{proof}

\begin{remark}
Theorem~\ref{Theorem_6.1} does not, in general, hold for $k=0$ under the assumption that $F$ is $B^{\vec{m}-\vec{e}_j}$-regular for all $1\le j\le t$. However, if one assumes instead that $F$ is $B^{\vec{m}-\vec{\delta}}$-regular, then the conclusion of the theorem remains valid for every $k\ge 0$.

\end{remark}

Theorem~\ref{Theorem_6.1} immediately yields the following vanishing result for Koszul cohomology.

\begin{corollary}
\label{CorollaryKoszulCohomology}
Let $X$ be a projective variety, let $F$ be a vector bundle on $X$ such that $F$ is $B^{\vec{m}-\vec{e}_j}$-regular for some $\vec{m}\in\mathbb{Z}^t$ and all $1\le j\le t$, and let $L=B^{\vec{w}}$ be a line bundle on $X$, where $\vec{w}\in\mathbb{N}^t$. Then
\[
K_{p,q}(X,F;L)=0
\]
whenever $(q-1)\vec{w}\ge \vec{m}+(p-1)\vec{\delta}$. In particular,
\[
K_{p,1}(X,F;L)=0
\]
whenever $\vec{m}+(p-1)\vec{\delta}\le \vec{0}$.
\end{corollary}

\begin{proof}
Taking $L_i=L$ for all $1\le i\le p+1$ in Theorem~\ref{Theorem_6.1}, we obtain
\[
H^1(M_L^{\otimes(p+1)}\otimes F\otimes L^{q-1})=0
\]
whenever
\[
(q-1)\vec{w}\ge \vec{m}+(p-1)\vec{\delta}.
\]
By the description of the Koszul cohomology groups in Section~2.6, this is equivalent to
\begin{center}

$K_{p,q}(X,F;L)=0$
\end{center} whenever
$(q-1)\vec{w}\ge \vec{m}+(p-1)\vec{\delta}.$

Finally, setting $q=1$, we obtain $H^1(M_L^{\otimes(p+1)}\otimes F)=0,$
which implies that \par
$K_{p,1}(X,F;L)=0$ whenever
$ \vec{m}+(p-1)\vec{\delta}\le \vec{0},$
where $\vec{0}=(0,\ldots,0)\in\mathbb{N}^t$.
\end{proof}

\begin{remark}
Corollary~\ref{CorollaryKoszulCohomology} is central to the present article, as it establishes a natural hierarchy of vanishing results for Koszul cohomology groups with increasing weights. This hierarchical structure provides a systematic framework for understanding the behavior of these cohomology groups and underlies several of the subsequent developments in this paper. Accordingly, we devote the next section to a detailed study of this hierarchy, together with its consequences and applications.
\end{remark}
 We computed minimal regularities of line bundles on product of projective spaces in Theorem \ref{theorem:minimal-regularity}. Now, we apply Corollary~\ref{CorollaryKoszulCohomology} to that particular examples.
\begin{corollary}
\label{corollary_koszul_prod_projsp}
    Let $X = \prod^{n}_{k = 1}\mathbb{P}^{m_{k}}$ be a product of projective spaces and $B = \left\{ B_{1}, \cdots, B_{n} \right\}$ where $B_{k}= p^{\ast}_{k}\mathcal{O}_{\mathbb{P}^{m_{k}}}(1)$. \par Let $F = B^{\vec{a}}$ be any line bundle in $X.$ Let $\vec{r}^{(\sigma)}$ be the minimal regularity vectors of $F$ obtained in Theorem \ref{theorem:minimal-regularity}.
    \par
    Then, for any vectors of non-negative integers $\vec{w}_{1}, \cdots, \vec{w}_{k+1}$ in $\mathbb{N}^{n},$ one has
    \begin{equation}
        H^{1}(M_{B^{w_{1}}} \otimes \cdots \otimes M_{B^{w_{n}}} \otimes F \otimes B^{\otimes m_{k}}) = 0
    \end{equation}
    for all $\vec{m}_{k} \geq \vec{r}^{(\sigma)}+(k-1)\vec{\delta}$ and moreover
    \begin{equation}
        H^{i}(M_{B^{w_{1}}} \otimes \cdots \otimes M_{B^{w_{n}}} \otimes F \otimes B^{\otimes m_{k}}) = 0
    \end{equation}
    for all $\vec{m}_{k} \geq \vec{r}^{(\sigma)}+k\vec{\delta}$ \par
    In particular, $K_{p, q}(X, F, L) = 0$ for all $L = B^{\vec{w}}$ with $(q-1)\vec{w} \geq \vec{r}^{(\sigma)}+(p-1)\delta$
\end{corollary}
Next, we mention the following corollary of Theorems \ref{multireghypersurface} and \ref{Theorem_6.1} which tells us about vanishing of syzygies of hypersurfaces. Fix a set $B = \left\{ B_{1}, \cdots, B_{t} \right\}$ of basepoint free line bundles.
\begin{corollary} $\text{[Syzygies of Complete Intersections]}$ \\
\label{Syz_complete_intersection}
    Let $X$ and let $F$ be any vector bundle on $X$ which is $B^{\vec{r}}$-regular with respect to B. For any integer $s \geq 1,$ let $Z = H_{1} \cap \cdots \cap H_{s}$ be a complete intersection of hypersurface $H_{k}$ on $X$ of multidegrees $\vec{d}_{1},\cdots, \vec{d}_{s}$  respectively, with respect to $B$.
    Let $\vec{d} = d_{1}+\cdots+\vec{d}_{s}$ and
    $$\vec{r}_{Z} = \vec{r}+\vec{d}.$$ 
    Then, for any vector of non-negative integers $\vec{w}_{1}, \cdots, \vec{w}_{k+1}$ in $\mathbb{N}^{t},$ one has
    \begin{equation}
        H^{1}(M_{B_{Z}^{\vec{w}_{1}}} \otimes \cdots \otimes M_{B_{Z}^{\vec{w}_{k+1}}} \otimes F_{Z} \otimes B_{Z}^{\vec{m}_{k}}) = 0
    \end{equation}
    for all $\vec{m}_{k} \geq \vec{r}_{H}+(k-1)\vec{\delta}$ and moreover
    \begin{equation}
        H^{i}(M_{B_{Z}^{\vec{w}_{1}}} \otimes \cdots \otimes M_{B_{Z}^{\vec{w}_{k+1}}} \otimes F_{Z} \otimes B_{Z}^{\vec{m}_{k}}) = 0 \text{ for all } i \geq 1
    \end{equation}
    for all $\vec{m}_{k} \geq \vec{r}_{Z}+k\vec{\delta}.$ \par
    In particular, $K_{p, q}(Z, F_{Z}, L_{Z}) = 0$ for all $L = B^{\vec{w}}$ and $(q-1)\vec{w} \geq \vec{r}_{Z}+(p-1)\vec{\delta}$
\end{corollary}
\begin{remark} $\text{ }$
\begin{itemize} 
    \item[(a)] For $s = 1,$ $Z = H_{1}$ is a hypersurface and hence the above corollary gives vanishing of syzygies of hypersurfaces.
    \item[(b)] Notice that if $X = \prod^{n}_{k = 1}\mathbb{P}^{m_{k}}$ is a product of projective spaces and $B = \left\{B_{1}, \cdots, B_{n} \right\}$ be the standard basis of $\mathrm{Pic}X,$ that is $B_{k} = p^{\ast}_{k}\mathcal{O}_{\mathbb{P}^{m_{k}}}(1),$ and $F = B^{\vec{a}},$ then we know from Theorem \ref{theorem:minimal-regularity} that the minimal regularities of $F$ are $\vec{r}^{(\sigma)}.$ Thus, by Corollary \ref{Syz_complete_intersection}, we obtain
    $$K_{p, q}(Z, F_{Z}, L_{Z}) = 0$$
    for $(q-1)\vec{w} \geq r^{(\sigma)}+\sum^{n}_{i = 1}\vec{d}_{k}+(p-1)\vec{\delta}$
\end{itemize}
    
\end{remark}

As another immediate consequence, we obtain the following sufficient condition for a line bundle to satisfy Property $(N_{p})$.

\begin{corollary}
\label{Corollary_Np-property}
Let $X$ be a projective variety and let $\vec{m}\in\mathbb{Z}^t$ be such that $\mathcal{O}_X$ is $B^{\vec{m}-\vec{e}_j}$-regular for all $1\le j\le t$. Let $L=B^{\vec{w}}$ be a line bundle on $X$ for some $\vec{w}\in\mathbb{N}^t$. Then, for $p\ge1$, $L$ satisfies property $(N_p)$ whenever
\[
\vec{w}\ge \vec{m}+(p-1)\vec{\delta}.
\]
\end{corollary}

\begin{proof}
By the definition of property $(N_p)$, a line bundle $L$ satisfies property $(N_p)$ if
\[
H^1(M_L^{\otimes(k+1)}\otimes L^j)=0
\]
for all $0\le k\le p$ and $j\ge1$. Applying Theorem~\ref{Theorem_6.1} with $F=\mathcal{O}_X$ and $L_i=L$ for all $1\le i\le k+1$, we obtain
$
H^1(M_L^{\otimes(k+1)}\otimes L^j)=0$
whenever
$
j\vec{w}\ge \vec{m}+(k-1)\vec{\delta}.$

Since $k\le p$, it is enough to require
\[
j\vec{w}\ge \vec{m}+(p-1)\vec{\delta}.
\]
Moreover, as $j\ge1$, the weakest condition occurs for $j=1$. 

Hence, $\vec{w}\ge \vec{m}+(p-1)\vec{\delta}$
implies $H^1(M_L^{\otimes(i+1)}\otimes L^j)=0$ for all $0\le i\le p$ and $j\ge1$. 
Therefore, $L$ satisfies property $(N_p)$.
\end{proof}

Thus far, we have worked with an arbitrary projective variety. From this point onward, we additionally assume that the variety is smooth.

Before proving the next corollary, we establish the following crucial lemma. It provides a lower bound for the regularity of the canonical line bundle $\omega_X$ with respect to a collection of ample and base point free line bundles.

\begin{lemma}
\label{canonical-line-bundle-lemma}
Let $X$ be a projective variety of dimension $n$, and let $B_1,\ldots,B_t$ be fixed ample and base point free line bundles on $X$, where $t\ge2$. If $\omega_X$ is $B^{\vec m}$-regular with respect to $B=(B_1,\ldots,B_t)$ for some $\vec m\in\mathbb Z^t$, then
\[
\vec m\ge n\vec\delta.
\]
Equivalently,
\[
\vec{m}_{\min}\le n\vec{\delta},
\]
where $\vec{m}_{\min}$ denotes the minimal regularity vector of $\omega_X$ with respect to $B$.
\end{lemma}

\begin{proof}
We establish the claimed lower bound for $\vec{m}$. Observe that if $\vec{m}\ge (n+1)\vec{\delta}$, then by the Kodaira Vanishing Theorem,
\begin{equation}
H^{i}(\omega_{X}\otimes B^{\vec{m}}\otimes B^{-{\vec{u}}})=0
\quad \text{for all } i>0,
\end{equation}
since the $B_i$ are ample line bundles.

Now take $\vec{m}=n\vec{\delta}$. Again by the Kodaira Vanishing Theorem,
\begin{equation}
H^{i}(\omega_{X}\otimes B^{\vec{m}}\otimes B^{-{\vec{u}}})=0
\quad \text{for all } 0<i\le n-1.
\end{equation}
To verify the vanishing for $i=n$, it is enough to consider a vector $\vec{u}$ having $n$ as one of its components. Without loss of generality, let $\vec{u}=(n,0,\cdots,0)$. Then
\begin{equation}
H^n(\omega_{X}\otimes B^{n\vec{\delta}}\otimes B^{-(n,0,\cdots,0)})
=
H^n(\omega_{X}\otimes \mathcal{O}_{X}\otimes B_2^{\,n}\otimes\cdots\otimes B_t^{\,n})
=0.
\end{equation}

On the other hand, taking $\vec{m}=(n-1)\vec{\delta}$ does not, in general, yield the required vanishing for all $i>0$. Indeed, by the Kodaira Vanishing Theorem,
\[
H^{i}(\omega_{X}\otimes B^{(n-1)\vec{\delta}}\otimes B^{-{\vec{u}}})=0
\quad \text{for all } 0<i\le n-1,
\]
whereas for $i=n$,
\begin{equation}
H^n(\omega_{X}\otimes B^{(n-1)\vec{\delta}}\otimes B^{-(n,0,\cdots,0)})
=
H^n(\omega_{X}\otimes B_1^{-1}\otimes B_2^{\,n-1}\otimes\cdots\otimes B_t^{\,n-1}),
\end{equation}
which is nonzero in general, since $B_1^{-1}$ is not necessarily ample. Therefore,
\[
\vec{m}\ge n\vec{\delta},
\]
as claimed.
\end{proof}

\begin{remark}
The bound $\vec{m}\ge n\vec{\delta}$ in the above lemma is optimal in general. However, under additional hypotheses on $X$ and the line bundles $B_1,\ldots,B_t$, this bound can be improved.
\end{remark}

\begin{remark}\label{remark:6.8}
The above lemma also holds in the case $t=1$, but with a larger bound.

Indeed, if $B_{1}=B$ is an ample and base point free line bundle, then taking $m=n$ gives
\[
H^{n}\!\left(\omega_{X}\otimes B^{n}\otimes B^{-n}\right)
=H^{n}(\omega_{X})
=H^{0}(\mathcal{O}_{X})\neq 0.
\]

On the other hand, taking $m=n+1$ yields
\[
H^{n}\!\left(\omega_{X}\otimes B\right)=0
\]
by the Kodaira Vanishing Theorem. Therefore, the Castelnuovo--Mumford regularity of
$\omega_{X}$ with respect to $B$, denoted by $\operatorname{reg}_{B}(\omega_{X})$, is $n+1$ in general.
\end{remark}

The above lemma together with Corollary~\ref{Corollary_Np-property} yields the following consequence.

\begin{corollary}
\label{Corollary_Np-propertyII-2}
Let $X$ be a smooth projective variety of dimension $n$ with
$\omega_X=\mathcal{O}_X$. Let $B_1,\ldots,B_t$ be fixed ample and base point free line bundles on $X$, where $t\ge2$. Let $L=B^{\vec{w}}$ be an ample and base point free line bundle on $X$ for some $\vec{w}\in\mathbb{N}^t$. Then $L$ satisfies property $(N_p)$ whenever
\[
\vec{w}\ge(p+n)\vec{\delta}.
\]
\end{corollary}

\begin{proof}
Since $\omega_X=\mathcal{O}_X$, Lemma~\ref{canonical-line-bundle-lemma} implies that $\mathcal{O}_X$ is $B^{n\vec{\delta}}$-regular. Hence, $\mathcal{O}_X$ is $B^{(n+1)\vec{\delta}-\vec{e}_j}$-regular for all $1\le j\le t$. Therefore, by Corollary~\ref{Corollary_Np-property}, the line bundle $L=B^{\vec{w}}$ satisfies property $(N_p)$ whenever
\[
\vec{w}\ge (n+1)\vec{\delta}+(p-1)\vec{\delta}=(n+p)\vec{\delta}.
\]
\end{proof}

\begin{remark}
The hypothesis $\omega_X\cong\mathcal{O}_X$ is satisfied by several important classes of smooth projective varieties.

\begin{itemize}
\item Abelian varieties: every abelian variety has trivial canonical bundle.

\item K3 surfaces: smooth projective surfaces with trivial canonical bundle and
$H^1(X,\mathcal{O}_X)=0$.

\item Calabi--Yau varieties: more generally, smooth projective varieties with
$\omega_X\cong\mathcal{O}_X$ and
$H^i(X,\mathcal{O}_X)=0$ for
$0<i<\dim X$.

\item Hypersurfaces: if
$X\subset\mathbb P^{m_1}\times\cdots\times\mathbb P^{m_t}$
is a smooth hypersurface of multidegree
$(m_1+1,\ldots,m_t+1)$,
then
$\omega_X\cong\mathcal{O}_X$
by the adjunction formula.
\end{itemize}

Consequently, Corollary~\ref{Corollary_Np-propertyII} applies to each of the above classes of varieties.
\end{remark}

\begin{remark}
Corollary~\ref{Corollary_Np-propertyII} generalizes Corollary~1.6 of \cite{GallegoPurnaprajnaI} for a Calabi--Yau $n$-fold when $t=1$.
\end{remark}

A distinguishing feature of Theorem \ref{Theorem_6.1} is that it provides a unified framework for three fundamental aspects of syzygy theory: Vanishing of higher-weight Koszul cohomology, Property $(N_p)$, and Property $(M_q)$. The preceding corollaries established the first two, while the following corollary completes this picture by deriving the corresponding criterion for Property $(M_q)$.

\begin{corollary}
\label{CorollaryM_q}
Let $X$ be a smooth projective variety of dimension $n\ge2$. Let $L=B^{\vec{w}}$ be an ample line bundle on $X$ for some $\vec{w}\in\mathbb{N}^t$. Let $\omega_X$ be $B^{\vec{m}-\vec{e}_j}$-regular for all $1\le j\le t$, where $\vec{m}\in\mathbb{Z}^t$. Then $L$ satisfies Property $(M_q)$ if
\[
(n-1)\vec{w}\ge \vec{m}+(q-n-1)\vec{\delta}.
\]
\end{corollary}

\begin{proof}
Since $L$ is ample and globally generated, Lemma~\ref{LemmaAprodu} implies that $L$ satisfies property $(M_q)$ if
\[
H^1(X,\wedge^{k+1}M_L\otimes\omega_X\otimes L^{n-1})=0
\]
for all $0\le k\le q-n$. In characteristic zero, wedge powers are direct summands of tensor powers. Hence, it is enough to prove that
\[
H^1(X,M_L^{\otimes(k+1)}\otimes\omega_X\otimes L^{n-1})=0
\]
for all $0\le k\le q-n$.

Applying Theorem~\ref{Theorem_6.1} with $L_i=L$ for all $1\le i\le q-n+1$ and $F=\omega_X$, we obtain
\[
H^1(X,M_L^{\otimes(q-n+1)}\otimes\omega_X\otimes L^{n-1})=0
\]
whenever
\[
(n-1)\vec{w}\ge\vec{m}+(q-n-1)\vec{\delta}.
\]
Moreover, this condition implies
\[
(n-1)\vec{w}\ge\vec{m}+(k-1)\vec{\delta}
\]
for every $0\le k\le q-n$. Therefore, the required vanishing holds for all $0\le k\le q-n$.

Finally, by Observation~\ref{Observation_5.3} together with Green's $K_{p,1}$ Theorem, we conclude that $L$ satisfies property $(M_q)$ whenever
\[
(n-1)\vec{w}\ge\vec{m}+(q-n-1)\vec{\delta}.
\]
\end{proof}

\begin{remark}
The above corollary implies, via Koszul duality, that
\[
K_{q-n,n}(X,\omega_X;L)=K_{r-q,1}(X,L)=0
\]
for the line bundle $L=B^{\vec{w}}$ whenever
\[
(n-1)\vec{w}\ge \vec{m}+(q-n-1)\vec{\delta}.
\]
Thus, the criterion for property $(M_q)$ also yields vanishing results for the corresponding weight-one Koszul cohomology groups.
\end{remark}

\begin{remark}
Writing $\vec{w}=(w_1,\ldots,w_t)$ and $\vec{m}=(m_1,\ldots,m_t)$, the criterion in the above corollary can be expressed componentwise as follows:
\[
L=B^{\vec{w}}=B_1^{w_1}\otimes\cdots\otimes B_t^{w_t}
\]
satisfies property $(M_q)$ whenever
\[
w_i\ge \frac{m_i+(q-n-1)}{n-1}
\]
for every $1\le i\le t$.
\end{remark}

\begin{remark}
\label{remark-bound-M_q}
For fixed ample and base point free line bundles $B_1,\ldots,B_t$ with $t\ge2$, Lemma~\ref{canonical-line-bundle-lemma} together with Corollary~\ref{CorollaryM_q} implies that
\[
L=B^{\vec{w}}
\]
satisfies property $(M_q)$ whenever
\[
(n-1)\vec{w}\ge (n+1)\vec{\delta}+(q-n-1)\vec{\delta}=q\vec{\delta}.
\]
Equivalently,
\[
\vec{w}\ge \frac{q}{n-1}\vec{\delta}.
\]
\end{remark}

\begin{remark}
In the case $t=1$, when $B$ is a single ample and base point free line bundle on $X$, Corollary~\ref{CorollaryM_q} generalizes Theorem~0.2 in \cite{Basu}. Indeed, for $L=B^w$, the criterion becomes
\[
(n-1)w\ge (n+1)+(q-n-1)=q+1,
\]
and hence
\[
w\ge \frac{q+1}{n-1}.
\]

\end{remark}

\section{\texorpdfstring{\textbf{Property} $(\mathbf{N}_{p})$ \textbf{on Product of Projective Spaces: Bounds, Optimality, and Comparisons}}{Property N{p} on Product of Projective Spaces: Bounds, Optimality, and Comparisons}}
\label{section_8_property_N_p_sharpness}

In this section, we discuss the consequences of property $(N_p)$ on product of projective spaces by demonstrating through examples.

First, we will illustrate through some examples what our Corollary \ref{Corollary_Np-property} tells about bounds on $p$ for which a line bundle satisfy property $-(N_p)$. Then we will show that our bound for property $-(N_p)$ is sharp.

We begin by illustrating a direct application of our corollary on property $(N_p)$ to products of projective spaces. In particular, we compute the maximal value of $p$ that can be obtained from Corollary \ref{Corollary_Np-property}. Theorem \ref{thm:minimal-regularity} plays a crucial role in these computations, as it provides explicit minimal regularities of $\mathcal{O}_{X}$. We will see how these regularities lead to effective bounds on $p$.

\begin{example}
Let $X=\mathbb{P}^{2}\times \mathbb{P}^{3}$. Let $B_1,B_2$ denote the line bundles obtained as pullbacks along the projections to the two factors, respectively. We compute the values of $p$ for which the line bundle $L=\mathcal{O}_{X}(5,7)$ satisfies $(N_p)$.

From Example \ref{PmxPn_example}, the minimal regularities of $\mathcal{O}_{X}$ (with respect to $B=(B_{1}, B_{2})$) are $(3,0)$ and $(0,2)$. In Corollary \ref{Corollary_Np-property}, we require that $\mathcal{O}_{X}$ is $B^{\vec{m}-\vec{e}_{j}}$-regular. The minimal such choices of $\vec{m}$ are:
\[
\vec{m}=(4,1),\quad (1,3),\quad (3,2),\quad (2,3).
\]
For each of these $\vec{m}$, $\mathcal{O}_{X}$ is $B^{\vec{m}-\vec{e}_{j}}$-regular for all $j=1,2$.

\medskip

(i) For $\vec{m}=(4,1)$, the line bundle $L=\mathcal{O}_{X}(5,7)$ satisfies $(N_p)$ if
\[
(5,7)\ge (4,1)+(p-1,p-1)=(p+3,p).
\]
This gives $p\le 2$ and $p\le 7$, hence $p\le 2$. Therefore $L$ satisfies $(N_2)$ in this case.

\medskip

(ii) For $\vec{m}=(1,3)$, $L$ satisfies $(N_p)$ if
\[
(5,7)\ge (1,3)+(p-1,p-1)=(p,p+2),
\]
which gives $p\le 5$. Hence, using this regularity, $L$ satisfies $(N_p)$ for all $p\le 5$.

\medskip

(iii) For $\vec{m}=(3,2)$, $L$ satisfies $(N_p)$ if
\[
(5,7)\ge (3,2)+(p-1,p-1)=(p+2,p+1),
\]
which gives $p\le 3$. Therefore $L$ satisfies $(N_3)$ in this case.

\medskip

(iv) For $\vec{m}=(2,3)$, $L$ satisfies $(N_p)$ if
\[
(5,7)\ge (2,3)+(p-1,p-1)=(p+1,p+2),
\]
which gives $p\le 4$. Hence $L$ satisfies $(N_p)$ for all $p\le 4$ in this case.

\medskip

Combining the above cases, we conclude that $L$ satisfies $(N_p)$ for all $p\le 5$.
\end{example}

\begin{remark}
The above example illustrates how the computation of minimal regularities play a crucial role in determining the possible values of $p$ for which a line bundle $L$ satisfies $(N_p)$. Here lies the importance of the Theorem \ref{theorem:minimal-regularity}. 
\end{remark}

In the above example, although we showed that $L$ satisfies $(N_p)$ for all $p\le 5$, it is not clear whether $p=5$ is optimal. In other words, we do not know whether $L$ satisfies $(N_6)$ or higher. Thus, the sharpness (or optimality) of our bound for property $(N_p)$ of the line bundle $L$ in the previous example remains unclear.

In order to prove the sharpness of our bound for Property $(N_p)$, we first recall the following theorem from \cite{GallegoPurnaprajnaIV}.

\begin{theorem}[Gallego--Purnaprajna, {\cite[Theorem 1.3]{GallegoPurnaprajnaIV}}]
\label{theorem:anti-canonical rational surface}
Let $X$ be an anti-canonical rational surface and let $L$ be an ample and base-point-free line bundle on $X$. Then $L$ satisfies Property $(N_p)$ if and only if ${\omega_X}^{*}\cdot L\ge p+3$.
\end{theorem}

In the following example, we show the sharpness of our bound for Property $(N_p)$ by using the above theorem.

\begin{example}
Let $X=\mathbb{P}^{1}\times \mathbb{P}^{1}$. Let $B_1,B_2$ denote the line bundles obtained as pullbacks along the projections to the two factors, respectively. Take $L=\mathcal{O}_{X}(1,1)$. We study the bounds for which $L$ satisfies $(N_p)$.

By Example \ref{PmxPn_example}, the minimal regularities of $\mathcal{O}_{X}$ are $(1,0)$ and $(0,1)$. Take $\vec{m}=(1,1)$. Then $\mathcal{O}_{X}$ is $B^{\vec{m}-\vec{e}_{j}}$-regular for $j=1,2$. 
Therefore, by Corollary \ref{Corollary_Np-property}, $L$ satisfies $(N_p)$ if
\[
(1,1)\ge \vec{m}+(p-1)\vec{\delta}=(1,1)+(p-1,p-1).
\]
Hence $L$ satisfies $(N_p)$ if $p\le 1$.

We now use Theorem \ref{theorem:anti-canonical rational surface} to show that the above bound is sharp.

Since ${\omega_{X}}^{*}=\mathcal{O}_{X}(2,2)$, it follows that $X=\mathbb{P}^{1}\times \mathbb{P}^{1}$ is an anti-canonical surface. 
The Chow ring of $X$ is
\[
A^{*}(X)=\mathbb{Z}[B_1,B_2]/(B_1^{2}, B_2^{2}).
\]
Abusing notation, we compute
\[
{\omega_{X}}^{*}\cdot L=(2B_{1}+2B_{2})\cdot (B_1+B_2)=4B_{1}\cdot B_{2}=4,
\]
since a general vertical and horizontal fiber, each isomorphic to $\mathbb{P}^1$, intersect transversally in one point, the intersection multiplicity is $1$, and hence $B_1 \cdot B_2 = 1$.
Therefore, ${\omega_{X}}^{*}\cdot L=4=1+3$. Hence, by Theorem \ref{theorem:anti-canonical rational surface}, the maximum $p$ for which $L=\mathcal{O}_{X}(1,1)$ satisfies $(N_p)$ is $p_{\max}=1$. 
This proves that our bound for Property $(N_p)$ is sharp.
\end{example}

In the previous examples, the bound was seen to be either sharp or inconclusive in the given context. The next example demonstrates that our bound for Property $(N_p)$ is not always sharp, showing that the regularity-based criterion obtained from minimal multigraded regularity is not always optimal.\par

The following theorem from \cite{GallegoPurnaprajnaIV} will be used to show that our bound is not sharp.

\medskip
\noindent\textbf{Fano $n$-folds of index $(n-1)$}
\medskip

\begin{theorem}[Gallego--Purnaprajna, {\cite[Theorem 2.1]{GallegoPurnaprajnaIV}}]
\label{theorem-index(n-1)}
Let $X$ be a Fano $n$-fold. Assume there exists an ample and base-pointfree
line bundle $L$ such that ${\omega_{X}}^{*}=L^{\otimes(n-1)}$ (e.g., if $X$ is a Fano $n$-fold of index $n-1$). Then, $L$ satisfies property $-(N_p)$ iff the intersection product $L^{n}:=\underbrace{L \cdot L \cdots L}_{n\ \text{times}}\ge p+3$. 
\end{theorem}

\begin{example}
Let $X=\mathbb{P}^{1}\times \mathbb{P}^{1}\times \mathbb{P}^{1}$. Let $B_1,B_2,B_3$ be the line bundles obtained as pullbacks along the projections to the three factors. Take $L=\mathcal{O}_{X}(1,1,1)$. We determine for which values of $p$ the line bundle $L$ satisfies $(N_p)$.

From Theorem \ref{thm:minimal-regularity}, the minimal regularities of $\mathcal{O}_{X}$ (with respect to $B=(B_1,B_2,B_3)$) are
\[
(2,1,0),\ (2,0,1),\ (1,2,0),\ (0,2,1),\ (0,1,2),\ (1,0,2).
\]
To apply Corollary \ref{Corollary_Np-property}, take $\vec{m}=(2,2,1)=(2,1,0)+\vec{e}_2+\vec{e}_3$. Then
\[
\vec{m}-\vec{e}_{1}=(1,2,1)>(1,2,0),\quad 
\vec{m}-\vec{e}_{2}=(2,1,1)>(2,1,0),\quad 
\vec{m}-\vec{e}_{3}=(2,2,0)>(1,2,0),
\]
and hence $\mathcal{O}_{X}$ is $B^{\vec{m}-\vec{e}_j}$-regular for all $j=1,2,3$. Therefore, by Corollary \ref{Corollary_Np-property}, the line bundle $L=\mathcal{O}_{X}(1,1,1)$ satisfies $(N_p)$ if
\[
(1,1,1)\ge (2,2,1)+(p-1,p-1,p-1)=(p+1,p+1,p).
\]
Thus, our bound implies that $L$ satisfies $(N_0)$.

\medskip

We now use Theorem \ref{theorem-index(n-1)} to examine the sharpness of this bound.

Since $\omega_{X}^{*}=\mathcal{O}_{X}(2,2,2)=L^{2}$, we compute the intersection number $L^{3}$. The Chow ring of $X$ is
\[
A^{*}(X)=\mathbb{Z}[B_1,B_2,B_3]/(B_1^{2},B_2^{2},B_3^{2}).
\]
Writing in divisor notation,
\[
L^{3}=L\cdot L\cdot L=(B_{1}+B_{2}+B_{3})^{3}=3!\, B_{1}\cdot B_{2}\cdot B_{3}.
\]
Since $B_i^2=0$ for all $i$, and since $B_1,B_2,B_3$ are pullbacks of hyperplane classes from the three factors, a general representative of each corresponds to a fiber of one of the projections. These three divisors intersect transversally in a single point, and hence
\[
B_1 \cdot B_2 \cdot B_3 = 1.
\]
Therefore $L^{3}=6$.

It follows that $L^{3}=6\ge p+3$ whenever $p\le 3$. Hence, by Theorem \ref{theorem-index(n-1)}, the line bundle $L=\mathcal{O}_{X}(1,1,1)$ satisfies $(N_3)$, and this bound is optimal. On the other hand, our bound only gives that $L$ satisfies $(N_0)$. Therefore, our bound is not optimal in this case.
\end{example}

\medskip
\noindent\textbf{Fano $n$-folds of index $(n-2)$ and $(n-3)$}
\medskip

Some very interesting Fano $n$-folds of index $n-2$ and $n-3$ are given by products of projective spaces. When specialized to these Fano $n$-folds, our result on Property $(N_p)$ (Corollary \ref{Corollary_Np-property}) yields better bounds, in general, than Corollary \ref{theorem-fano-n-fold-index(n-2)} of \cite{GallegoPurnaprajnaIII} and Theorem \ref{theorem-fano-n-fold-index(n-3)} of \cite{GallegoPurnaprajnaII}. 

The following two examples, one for each case, illustrate this improvement.

\begin{corollary}
[Gallego--Purnaprajna, {\cite[Corollary 3.5]{GallegoPurnaprajnaIII}}]
\label{theorem-fano-n-fold-index(n-2)}
Let $X$ be a Fano $n$-fold with index $(n-2)$ and let ${\omega_{X}}^{*}=L^{\otimes (n-2)}$ with $L$ a big and base-point-free line bundle such that $L^{n}\ge 4$. If $s\ge p+1\ge 2$, then $L^{\otimes s}$ satisfies property $(N_p)$. 
\end{corollary}\par

We will give the following example to illustrate that our bound is sharper than the bound given by Corollary \ref{theorem-fano-n-fold-index(n-2)} in \cite{GallegoPurnaprajnaIII}.

\begin{example}
Let $X=\mathbb{P}^{1}\times \mathbb{P}^{2}$, and let $B_1,B_2$ denote the pullbacks of the hyperplane classes from the two factors. 

We have
\[
\omega_X^* = \mathcal{O}_X(2,3).
\]

 Since $\gcd(2,3)=1$, the line bundle $\omega_X^{*}$ cannot be written as a nontrivial tensor power of a different ample line bundle. Hence $X$ is a Fano $3$-fold of index $1=3-2$. Denote
\[
L := \omega_X^{*} = \mathcal{O}_X(2,3).
\]

We first compute the top self-intersection $L^3 = L \cdot L \cdot L$. In divisor notation,
\[
L = 2B_1 + 3B_2.
\]
The Chow ring of  $X$ is
\[
A^{*}(X)=\mathbb{Z}[B_1,B_2]/(B_1^2, B_2^3).
\]
Hence  we have
\[
B_1^2=0,\qquad B_2^3=0,
\]

Therefore, we obtain
\[
L^3 = (2B_1 + 3B_2)^3 = 3\cdot (2B_1)\cdot (3B_2)^2 = 54\, B_1 \cdot B_2^2.
\]

However, a general representative of $B_1$ is of the form $\{p\}\times \mathbb{P}^2$. If $D_2,D_2' \in |B_2|$ are two general divisors, then
\[
D_2 \cap D_2' = \mathbb{P}^1 \times \{\text{point}\}.
\]
Hence,
\[
B_1 \cdot B_2^2 = 1.
\]

Therefore,
\[
L^3 = 54 > 4.
\]

Let $H := L^{\otimes 2} = \mathcal{O}_X(4,6)$. We now compute the value of $p$ for which $H$ satisfies property $(N_p)$, and compare the bounds given by Corollary~\ref{Corollary_Np-property} and Corollary~\ref{theorem-fano-n-fold-index(n-2)}.

By Corollary~\ref{theorem-fano-n-fold-index(n-2)}, the line bundle $H=L^2$ satisfies $(N_p)$ provided
\[
2 \ge p+1,
\]
hence $p \le 1$. Therefore, $H$ satisfies $(N_1)$ by this result.

Next, using Example~\ref{PmxPn_example}, the minimal regularities of $\mathcal{O}_X$ are $(2,0)$ and $(0,1)$. Hence we may choose
\[
\vec{m}=(2,1)\quad \text{or}\quad (1,2),
\]
both satisfying the required $B^{\vec{m}-\vec{e_j}}$-regularity conditions for $j=1,2$. We choose $\vec{m}=(1,2)$, which gives a better bound.

Then Corollary~\ref{Corollary_Np-property} gives that $H$ satisfies $(N_p)$ provided
\[
(4,6) \ge (1,2) + (p-1,p-1) = (p,p+1),
\]
which implies $p \le 4$.

Thus, using Corollary~\ref{Corollary_Np-property}, we obtain that $H=L^{\otimes 2}$ satisfies Property $(N_p)$ for all $p\le 4$, which is a significant improvement over the bound $p\le 1$ obtained from Corollary~\ref{theorem-fano-n-fold-index(n-2)}.
\end{example}

The next example illustrates that the bound in Corollary \ref{Corollary_Np-property} improves the corresponding bound in Theorem \ref{theorem-fano-n-fold-index(n-3)}.

\begin{theorem}
    
[Gallego--Purnaprajna, {\cite[Theorem 3.2]{GallegoPurnaprajnaIV}}]
\label{theorem-fano-n-fold-index(n-3)}
Let $X$ be a Fano $n$-fold with index $m=(n-3)$. Assume that ${\omega_{X}}^{*}=H^{\otimes m}$ and $H$ a ample and base-point-free line bundle. Let $L=H^{\otimes k}$. Assume furthermore, $h^{0}(H)\ge n+2$ ,i.e., that $|H|$ does not map $X$ onto $\mathbb{P}^{n}$. If $k\ge p+2$ and $p\ge 1$, then $L$ satisfies property $(N_p)$.

\end{theorem}

\begin{example}
Let $X=\mathbb{P}^{1}\times \mathbb{P}^{1}\times \mathbb{P}^{2}$. Then
\[
\omega_{X}^{*}=\mathcal{O}_{X}(2,2,3).
\]
Since $\gcd(2,3)=1$, we may apply Theorem \ref{theorem-fano-n-fold-index(n-3)} with $H=\mathcal{O}_{X}(2,2,3)$, so that $\omega_{X}^{*}=H=H^{\otimes 1}$. Hence $X$ is a Fano $4$-fold of index $1=4-3$.

We first compute $h^{0}(H)$. By the K\"unneth formula,
\[
H^0\bigl(X,\mathcal{O}_{X}(2,2,3)\bigr)
= H^0\bigl(\mathbb{P}^{1},\mathcal{O}_{\mathbb{P}^{1}}(2)\bigr)
\otimes H^0\bigl(\mathbb{P}^{1},\mathcal{O}_{\mathbb{P}^{1}}(2)\bigr)
\otimes H^0\bigl(\mathbb{P}^{2},\mathcal{O}_{\mathbb{P}^{2}}(3)\bigr).
\]
Therefore,
\[
h^{0}(H)=\binom{3}{1}\cdot \binom{3}{1}\cdot \binom{5}{2}=90>4+2=6,
\]
so the hypothesis of Theorem \ref{theorem-fano-n-fold-index(n-3)} is satisfied.

Now take $L=H^{\otimes 3}$. The condition $3\ge p+2$ gives $p\le 1$. Hence, by Theorem \ref{theorem-fano-n-fold-index(n-3)}, $L$ satisfies property $(N_{1})$.

Next we compare this with the bound obtained from Corollary \ref{Corollary_Np-property}. By Theorem \ref{thm:minimal-regularity}, the minimal regularities of $\mathcal{O}_{X}$ are
\[
(3,2,0),\ (3,0,1),\ (2,3,0),\ (0,3,1),\ (0,1,2),\ (1,0,2).
\]
The optimal choice is $\vec{m}=(1,1,3)$. Then $\mathcal{O}_{X}$ is $\vec{m}-\vec{e}_{j}$-regular for all $j=1,2,3$. Hence, by Corollary \ref{Corollary_Np-property}, $L=H^{\otimes 3}$ satisfies property $(N_{p})$ provided
\[
3\cdot (2,2,3)\ge (1,1,3)+(p-1,p-1,p-1)=(p,p,p+2),
\]
which gives $p\le 6$. Therefore, $L$ satisfies $(N_{6})$, which is significantly sharper than $(N_{1})$.
\end{example}

\begin{remark}
In the above example, Corollary \ref{Corollary_Np-property} implies that $H$ satisfies Property $(N_{2})$ and $H^{\otimes 2}$ satisfies Property $(N_{4})$. In contrast, Theorem \ref{theorem-fano-n-fold-index(n-3)} does not provide any conclusion concerning Property $(N_p)$ for either $H$ or $H^{\otimes 2}$.
\end{remark}

\section{\texorpdfstring{\textbf{A General Reformulation of the Hering--Schenck--Smith Theorem on Property $\mathbf{(N_p)}$}}{A General Reformulation of the Hering--Schenck--Smith Theorem on Property (Np)}}
\label{section_9_reformulation_of_HSS}

In this section, we reformulate Theorem 1.1 of \cite{HeringSchenckSmith} in a more general setting and recover their original theorem on Property $(N_p)$ as a corollary. This reformulation also clarifies the role of the underlying regularity assumptions, which will be interpreted geometrically in the next section.

We first determine the most general setting in which the theorem of \cite{HeringSchenckSmith} remains applicable, thereby identifying the broadest scope of their result.

We then sketch how the proof of \cite{HeringSchenckSmith} extends to this more general setting.
\begin{theorem}
\label{theorem-HSS-general}
Let $X$ be a projective variety with a collection of base-point free line bundles
$B_1,\ldots,B_t$, and let $B^{\vec v}$ and $B^{\vec m}$ denote the
corresponding line bundles for $\vec v,\vec m\in\mathbb{N}^t$. Assume that
$\vec v\geq \vec 0$, $\vec m\geq \vec\delta$, and that $B^{\vec v}$ is
$B^{\vec m}$-regular. Let
\[
\vec w_1,\ldots,\vec w_{k+1}\in\mathbb{N}^t
\]
satisfy $\vec w_1\geq \vec v+\vec m$. Then
\[
H^1\left(M_{B^{\vec w_1}}\otimes M_{B^{\vec w_2}}\otimes\cdots
\otimes M_{B^{\vec w_{k+1}}}\otimes B^{\vec v}\otimes B^{\vec m_k}\right)=0
\]
whenever
\[
\vec m_k\geq \vec m+(k-1)\vec\delta .
\]
\end{theorem}

\begin{proof}
Set $L_i:=B^{\vec w_i}$ for $1\leq i\leq k+1$. By Lemma \ref{LemmaV}, 
$M_{L_1}\otimes B^{\vec v}$ is $B^{\vec m}$-regular. Therefore, by 
Proposition \ref{Proposition III}, we obtain that
\[
M_{L_1}\otimes\cdots\otimes M_{L_{k+1}}\otimes B^{\vec v}
\]
is $B^{\vec m+k\vec\delta}$-regular.

Now consider the short exact sequence
\[
\begin{tikzcd}
0 \arrow[r] & M_{L_{k+1}} \arrow[r] &
H^0(L_{k+1})\otimes \mathcal{O}_X \arrow[r] &
L_{k+1} \arrow[r] & 0 .
\end{tikzcd}
\]
Tensoring this sequence with
\[
M_{L_1}\otimes\cdots\otimes M_{L_k}\otimes B^{\vec v}\otimes B^{\vec m_k},
\]
and applying the same argument as in the proof of Theorem \ref{Theorem_6.1}, we obtain that
\[
M_{L_1}\otimes\cdots\otimes M_{L_k}\otimes B^{\vec v}
\]
is $B^{\vec m+(k-1)\vec\delta}$-regular. Hence, for
\[
\vec m_k\geq \vec m+(k-1)\vec\delta,
\]
we have
\[
H^1\left(M_{L_1}\otimes\cdots\otimes M_{L_k}
\otimes B^{\vec v}\otimes B^{\vec m_k}\right)=0.
\]
Moreover, the above regularity implies the surjectivity of the multiplication map
\[
H^0(L_{k+1})\otimes
H^0\left(M_{L_1}\otimes\cdots\otimes M_{L_k}\otimes B^{\vec v}
\otimes B^{\vec m+(k-1)\vec\delta}\right)
\]
\[
\longrightarrow
H^0\left(L_{k+1}\otimes M_{L_1}\otimes\cdots\otimes M_{L_k}
\otimes B^{\vec v}\otimes B^{\vec m+(k-1)\vec\delta}\right).
\]
Therefore, by Theorem \ref{Theorem_3.7}(ii), we conclude the desired vanishing for
\[
\vec m_k\geq \vec m+(k-1)\vec\delta .
\]
\end{proof}

\begin{remark}
The key point in the proof of the above theorem is the different role played by
$L_1$ compared to the remaining line bundles $L_i$. The choice of $L_1$ in the
initial step is essential for obtaining the bound
\[
\vec m_k\geq \vec m+(k-1)\vec\delta .
\]
Indeed, if we instead started with the short exact sequence
\[
0\to M_{L_1}\to H^0(L_1)\otimes\mathcal{O}_X\to L_1\to0
\]
and tensor it with
\[
M_{L_2}\otimes\cdots\otimes M_{L_{k+1}}\otimes B^{\vec v},
\]
then the same argument would yield the weaker bound
\[
\vec m_k\geq \vec m+k\vec\delta .
\]
\end{remark}

We recover Theorem 1.1 in \cite{HeringSchenckSmith} as a corollary of Theorem \ref{theorem-HSS-general}.

\begin{corollary}
\label{corollary-N_p-HSS}
Let $X$ be a projective variety, and let $B^{\vec w}$ and $B^{\vec m}$ be two line bundles on $X$ such that
$\vec w,\vec m\in\mathbb{N}^t$ and $\vec w\geq \vec m$. If $\mathcal{O}_X$ is
$B^{\vec m}$-regular, then, for every $p\geq 1$, the line bundle $B^{\vec w}$
satisfies Property $(N_p)$ whenever
\[
\vec w\geq \vec m+(p-1)\vec\delta .
\]
\end{corollary}

\begin{proof}
Set $L:=B^{\vec w}$. Recall that a globally generated non-special line bundle
$L$ satisfies Property $(N_p)$ if and only if
\[
H^1(X,M_L^{\otimes q}\otimes L^j)=0
\]
for $1\leq q\leq p+1$ and $j\geq 1$.

In Theorem \ref{theorem-HSS-general}, take $B^{\vec v}=\mathcal{O}_X$ and
$k=p$. Then we obtain
\[
H^1(X,M_L^{\otimes(p+1)}\otimes L)=0
\]
whenever
\[
\vec w\geq \vec m+(p-1)\vec\delta .
\]
Observe that the same inequality
\[
\vec w\geq \vec m+(p-1)\vec\delta
\]
immediately ensures that
\[
H^1(X,M_L^{\otimes q}\otimes L^j)=0
\]
for all $1\leq q\leq p+1$ and $j\geq1$. Therefore, since we are working over a
field of characteristic zero, $L=B^{\vec w}$ satisfies Property $(N_p)$ whenever
\[
\vec w\geq \vec m+(p-1)\vec\delta .
\]
\end{proof}

\begin{remark}
The crucial technical difference between our hypothesis and that of
Hering--Schenck--Smith is that they assume $\mathcal{O}_X$ is
$B^{\vec m}$-regular, whereas we work with the weaker assumption that
$\mathcal{O}_X$ is $B^{\vec m-\vec e_j}$-regular. If we were to assume
$B^{\vec m}$-regularity and apply our technique, we would end up obtaining a
weaker bound, namely
\[
\vec m+p\vec\delta,
\]
instead of
\[
\vec m+(p-1)\vec\delta,
\]
which is the bound obtained in \cite{HeringSchenckSmith} and in Corollary
\ref{Corollary_Np-property}. The geometric significance of this technical
difference will be discussed in the next section.
\end{remark}

\section{\textbf{Geometric Interpretation of the Regularity Conditions and Syzygy Bounds}}
\label{section_10_geometric_interpretation}

In this section, we interpret the conditions of our Corollary \ref{CorollaryM_q} through figures of convex subsets generates by line bundles satisfying Property $(M_{q})$  inside the Picard group.

\subsection{Geometric Interpretation of the regularity assumption in Main Theorem}

Let $X$ be a smooth projective variety, and let
$B_1,\ldots,B_t$ be fixed globally generated line bundles on $X$ that are
$\mathbb{R}$-linearly independent in $\operatorname{Pic}_{\mathbb{R}}(X):=\operatorname{Pic}(X)\otimes_{\mathbb Z}\mathbb R$. Set
$\vec{B}=(B_1,\ldots,B_t)$ and define
\[
\mathbb{B}=\{B^{\vec{w}}\mid \vec{w}\in\mathbb{Z}^{t}\}
\subseteq \operatorname{Pic}(X).
\]
The map
\[
\nu:\mathbb{Z}^{t}\longrightarrow \mathbb{B},\qquad
\nu(w_1,\ldots,w_t)=B_1^{w_1}\otimes\cdots\otimes B_t^{w_t}
\]
identifies $\mathbb{B}$ with the lattice $\mathbb{Z}^{t}$ inside
$\operatorname{Pic}_{\mathbb{R}}(X)
$.
Thus, we may view $\mathbb{B}$ as a lattice in $\mathbb{R}^{t}$, where
each $B_i$ corresponds to the $i$-th coordinate axis. The semigroup
\[
\mathfrak{B}=\{B^{\vec{w}}\mid \vec{w}\in\mathbb{N}^{t}\}
\]
corresponds to the lattice points in the first quadrant, whose closure is
the convex cone $\overline{\mathfrak{B}}\subseteq\mathbb{R}^{t}$.
This is illustrated in the figure below for $t = 2$.
\begin{footnotesize}
\begin{figure}[ht!]
\begin{tikzpicture}[row sep = small, column sep = small]
\fill[gray!30!] (0,0) rectangle (2.2,2.2);
    \draw[<->] (-3.5, 0) -- (3.5, 0);
    \draw[<->] (0, -3.5) -- (0, 3.5);
    \filldraw[color=gray] (0, 0) circle (4pt);
    \draw[color=black] (0, 0) circle (4pt);
    \filldraw[color=gray] (1, 0) circle (4pt);
    \draw[color=black] (1, 0) circle (4pt);
    \filldraw[color=gray] (2, 0) circle (4pt);
    \draw[color=black] (2, 0) circle (4pt);
    \filldraw[color=gray] (1, 1) circle (4pt);
    \draw[color=black] (1, 1) circle (4pt);
    \filldraw[color=gray] (2, 1) circle (4pt);
    \draw[color=black] (2, 1) circle (4pt);
    \filldraw[color=gray] (1, 2) circle (4pt);
    \draw[color=black] (1, 2) circle (4pt);
    \filldraw[color=gray] (2, 0) circle (4pt);
    \draw[color=black] (2, 0) circle (4pt);
    \filldraw[color=gray] (2, 1) circle (4pt);
    \draw[color=black] (2, 1) circle (4pt);
    \filldraw[color=gray] (2, 2) circle (4pt);
    \draw[color=black] (2, 2) circle (4pt);
    \filldraw[color=gray] (0, 1) circle (4pt);
    \draw[color=black] (0, 1) circle (4pt);
    \filldraw[color=gray] (0, 2) circle (4pt);
    \draw[color=black] (0, 2) circle (4pt);

    \filldraw[color=white] (-1, 0) circle (4pt);
    \draw[color=black] (-1, 0) circle (4pt);
    \filldraw[color=white] (-2, 0) circle (4pt);
    \draw[color=black] (-2, 0) circle (4pt);
    \filldraw[color=white] (-1, 1) circle (4pt);
    \draw[color=black] (-1, 1) circle (4pt);
    \filldraw[color=white] (-2, 1) circle (4pt);
    \filldraw[color=white] (-1, 2) circle (4pt);
    \draw[color=black] (-1, 2) circle (4pt);
    \filldraw[color=white] (-2, 2) circle (4pt);
    \filldraw[color=white] (-2, 0) circle (4pt);
    \draw[color=black] (-2, 0) circle (4pt);
    \filldraw[color=white] (-2, 1) circle (4pt);
    \draw[color=black] (-2, 1) circle (4pt);
    \draw[color=black] (-2, 2) circle (4pt);
    \draw[color=black] (-2, 2) circle (4pt);

    \filldraw[color=white] (-1, -1) circle (4pt);
    \draw[color=black] (-1, -1) circle (4pt);
    \filldraw[color=white] (-2, -1) circle (4pt);
    \filldraw[color=white] (-1, -2) circle (4pt);
    \draw[color=black] (-1, -2) circle (4pt);
    \filldraw[color=white] (-2, -2) circle (4pt);
    \filldraw[color=white] (-2, 0) circle (4pt);
    \draw[color=black] (-2, 0) circle (4pt);
    \filldraw[color=white] (-2, -1) circle (4pt);
    \draw[color=black] (-2, -1) circle (4pt);
    \draw[color=black] (-2, -2) circle (4pt);
    \draw[color=black] (-2, -2) circle (4pt);
    \filldraw[color=white] (0, -1) circle (4pt);
    \draw[color=black] (0, -1) circle (4pt);
    \filldraw[color=white] (0,-2) circle (4pt);
    \draw[color=black] (0, -2) circle (4pt);

    \filldraw[color=white] (1, -1) circle (4pt);
    \draw[color=black] (1, -1) circle (4pt);
    \filldraw[color=white] (2, -1) circle (4pt);
    \filldraw[color=white] (1, -2) circle (4pt);
    \draw[color=black] (1, -2) circle (4pt);
    \filldraw[color=white] (2, -1) circle (4pt);
    \draw[color=black] (2, -1) circle (4pt);
    \filldraw[color=white] (2, -2) circle (4pt);
    \draw[color=black] (2, -2) circle (4pt);
    \node[scale=1] at (3.8, 0.3) {$w_{1}$};
    \node[scale=1] at (0.3, 3.8) {$w_{2}$};
    \node[scale=1.2] at (2.5, 2.5) {$\mathfrak{B}$};
    \node[scale=1.2] at (-2.5, -2.5) {$\mathbb{B}$};
    \node[scale=1.2] at (-0.5, 2.5) {$\overline{\mathfrak{B}}$};
    \draw[->] (-0.2, 2.2) -- (0.5, 0.5);
\end{tikzpicture}
\caption{Lattice $\mathbb{B}$ and semigroup $\mathfrak{B}$: all circles correspond to points of $\mathbb{B}$ and closed circles correspond to points in $\mathfrak{B}$ and the shaded region (first quadrant) is $\overline{\mathfrak{B}}$}
\end{figure}
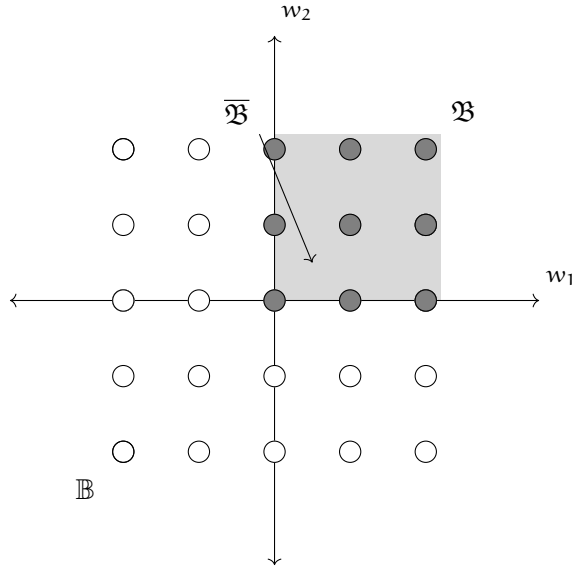
\end{footnotesize}
\textbf{ } \par

Now, for some $\vec{m}\in\mathbb{Z}^{t}$, the regions corresponding to
$B^{\vec{m}}$-regularity and $B^{\vec{m}-\vec{e}_{j}}$-regularity for all
$j$ are illustrated in the following figures.
\begin{footnotesize}
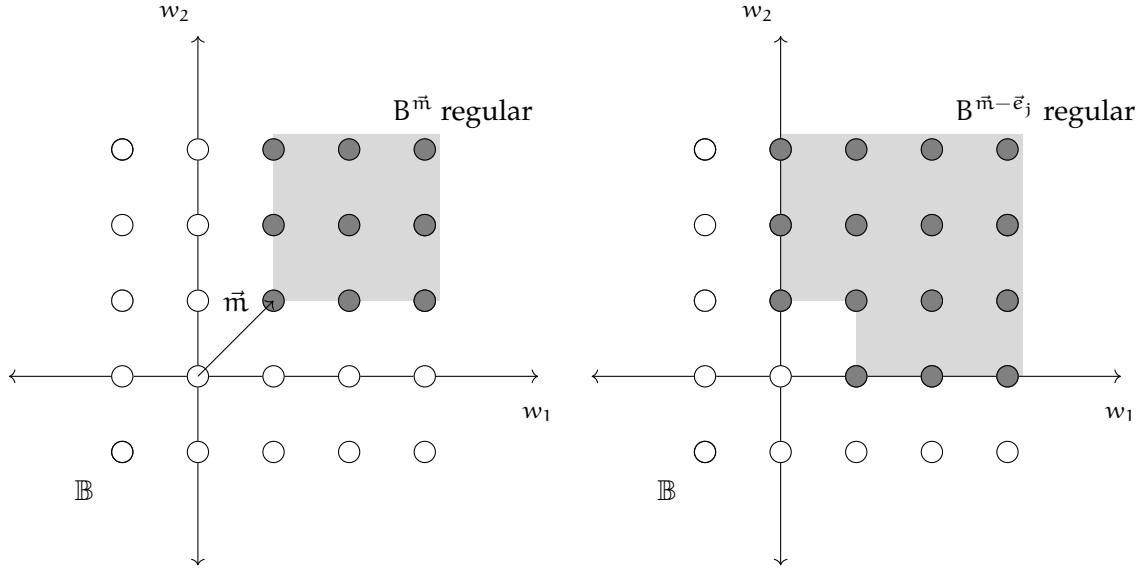
\begin{figure}[ht!]
\hspace{-3.5cm}
\begin{minipage}{0.2\textwidth}
\begin{tikzpicture}[row sep = small, column sep = small]
\fill[gray!30!] (0,0) rectangle (2.2,2.2);
    \draw[<->] (-3.5, -1) -- (3.5, -1);
    \draw[<->] (-1, -3.5) -- (-1, 3.5);
    \filldraw[color=gray] (0, 0) circle (4pt);
    \draw[color=black] (0, 0) circle (4pt);
    \filldraw[color=gray] (1, 0) circle (4pt);
    \draw[color=black] (1, 0) circle (4pt);
    \filldraw[color=gray] (2, 0) circle (4pt);
    \draw[color=black] (2, 0) circle (4pt);
    \filldraw[color=gray] (1, 1) circle (4pt);
    \draw[color=black] (1, 1) circle (4pt);
    \filldraw[color=gray] (2, 1) circle (4pt);
    \draw[color=black] (2, 1) circle (4pt);
    \filldraw[color=gray] (1, 2) circle (4pt);
    \draw[color=black] (1, 2) circle (4pt);
    \filldraw[color=gray] (2, 0) circle (4pt);
    \draw[color=black] (2, 0) circle (4pt);
    \filldraw[color=gray] (2, 1) circle (4pt);
    \draw[color=black] (2, 1) circle (4pt);
    \filldraw[color=gray] (2, 2) circle (4pt);
    \draw[color=black] (2, 2) circle (4pt);
    \filldraw[color=gray] (0, 1) circle (4pt);
    \draw[color=black] (0, 1) circle (4pt);
    \filldraw[color=gray] (0, 2) circle (4pt);
    \draw[color=black] (0, 2) circle (4pt);

    \filldraw[color=white] (-1, 0) circle (4pt);
    \draw[color=black] (-1, 0) circle (4pt);
    \filldraw[color=white] (-2, 0) circle (4pt);
    \draw[color=black] (-2, 0) circle (4pt);
    \filldraw[color=white] (-1, 1) circle (4pt);
    \draw[color=black] (-1, 1) circle (4pt);
    \filldraw[color=white] (-2, 1) circle (4pt);
    \filldraw[color=white] (-1, 2) circle (4pt);
    \draw[color=black] (-1, 2) circle (4pt);
    \filldraw[color=white] (-2, 2) circle (4pt);
    \filldraw[color=white] (-2, 0) circle (4pt);
    \draw[color=black] (-2, 0) circle (4pt);
    \filldraw[color=white] (-2, 1) circle (4pt);
    \draw[color=black] (-2, 1) circle (4pt);
    \draw[color=black] (-2, 2) circle (4pt);
    \draw[color=black] (-2, 2) circle (4pt);

    \filldraw[color=white] (-1, -1) circle (4pt);
    \draw[color=black] (-1, -1) circle (4pt);
    \filldraw[color=white] (-2, -1) circle (4pt);
    \filldraw[color=white] (-1, -2) circle (4pt);
    \draw[color=black] (-1, -2) circle (4pt);
    \filldraw[color=white] (-2, -2) circle (4pt);
    \filldraw[color=white] (-2, 0) circle (4pt);
    \draw[color=black] (-2, 0) circle (4pt);
    \filldraw[color=white] (-2, -1) circle (4pt);
    \draw[color=black] (-2, -1) circle (4pt);
    \draw[color=black] (-2, -2) circle (4pt);
    \draw[color=black] (-2, -2) circle (4pt);
    \filldraw[color=white] (0, -1) circle (4pt);
    \draw[color=black] (0, -1) circle (4pt);
    \filldraw[color=white] (0,-2) circle (4pt);
    \draw[color=black] (0, -2) circle (4pt);

    \filldraw[color=white] (1, -1) circle (4pt);
    \draw[color=black] (1, -1) circle (4pt);
    \filldraw[color=white] (2, -1) circle (4pt);
    \filldraw[color=white] (1, -2) circle (4pt);
    \draw[color=black] (1, -2) circle (4pt);
    \filldraw[color=white] (2, -1) circle (4pt);
    \draw[color=black] (2, -1) circle (4pt);
    \filldraw[color=white] (2, -2) circle (4pt);
    \draw[color=black] (2, -2) circle (4pt);
    \node[scale=1] at (3.5, -1.5) {$w_{1}$};
    \node[scale=1] at (-1.3, 3.8) {$w_{2}$};
    \node[scale=1.2] at (2.5, 2.5) {$B^{\vec{m}}$ regular};
    \node[scale=1.2] at (-2.5, -2.5) {$\mathbb{B}$};

    \draw[->] (-1, -1) -- (0, 0);
    \node[scale=1.2] at (-0.5, 0) {$\vec{m}$};
\end{tikzpicture}
\end{minipage}
\hspace{4.5cm}
\begin{minipage}{0.2\textwidth}
\begin{tikzpicture}[row sep = small, column sep = small]
\fill[gray!30!] (0,-1) rectangle (2.2,2.2);
\fill[gray!30!] (-1,0) rectangle (2.2,2.2);
    \draw[<->] (-3.5, -1) -- (3.5, -1);
    \draw[<->] (-1, -3.5) -- (-1, 3.5);
    \filldraw[color=gray] (0, 0) circle (4pt);
    \draw[color=black] (0, 0) circle (4pt);
    \filldraw[color=gray] (1, 0) circle (4pt);
    \draw[color=black] (1, 0) circle (4pt);
    \filldraw[color=gray] (2, 0) circle (4pt);
    \draw[color=black] (2, 0) circle (4pt);
    \filldraw[color=gray] (1, 1) circle (4pt);
    \draw[color=black] (1, 1) circle (4pt);
    \filldraw[color=gray] (2, 1) circle (4pt);
    \draw[color=black] (2, 1) circle (4pt);
    \filldraw[color=gray] (1, 2) circle (4pt);
    \draw[color=black] (1, 2) circle (4pt);
    \filldraw[color=gray] (2, 0) circle (4pt);
    \draw[color=black] (2, 0) circle (4pt);
    \filldraw[color=gray] (2, 1) circle (4pt);
    \draw[color=black] (2, 1) circle (4pt);
    \filldraw[color=gray] (2, 2) circle (4pt);
    \draw[color=black] (2, 2) circle (4pt);
    \filldraw[color=gray] (0, 1) circle (4pt);
    \draw[color=black] (0, 1) circle (4pt);
    \filldraw[color=gray] (0, 2) circle (4pt);
    \draw[color=black] (0, 2) circle (4pt);

    \filldraw[color=white] (-1, 0) circle (4pt);
    \draw[color=black] (-1, 0) circle (4pt);
    \filldraw[color=white] (-2, 0) circle (4pt);
    \draw[color=black] (-2, 0) circle (4pt);
    \filldraw[color=white] (-1, 1) circle (4pt);
    \draw[color=black] (-1, 1) circle (4pt);
    \filldraw[color=white] (-2, 1) circle (4pt);
    \filldraw[color=white] (-1, 2) circle (4pt);
    \draw[color=black] (-1, 2) circle (4pt);
    \filldraw[color=white] (-2, 2) circle (4pt);
    \filldraw[color=white] (-2, 0) circle (4pt);
    \draw[color=black] (-2, 0) circle (4pt);
    \filldraw[color=white] (-2, 1) circle (4pt);
    \draw[color=black] (-2, 1) circle (4pt);
    \draw[color=black] (-2, 2) circle (4pt);
    \draw[color=black] (-2, 2) circle (4pt);

    \filldraw[color=white] (-1, -1) circle (4pt);
    \draw[color=black] (-1, -1) circle (4pt);
    \filldraw[color=white] (-2, -1) circle (4pt);
    \filldraw[color=white] (-1, -2) circle (4pt);
    \draw[color=black] (-1, -2) circle (4pt);
    \filldraw[color=white] (-2, -2) circle (4pt);
    \filldraw[color=white] (-2, 0) circle (4pt);
    \draw[color=black] (-2, 0) circle (4pt);
    \filldraw[color=white] (-2, -1) circle (4pt);
    \draw[color=black] (-2, -1) circle (4pt);
    \draw[color=black] (-2, -2) circle (4pt);
    \draw[color=black] (-2, -2) circle (4pt);
    \filldraw[color=white] (0, -1) circle (4pt);
    \draw[color=black] (0, -1) circle (4pt);
    \filldraw[color=white] (0,-2) circle (4pt);
    \draw[color=black] (0, -2) circle (4pt);

    \filldraw[color=white] (1, -1) circle (4pt);
    \draw[color=black] (1, -1) circle (4pt);
    \filldraw[color=white] (2, -1) circle (4pt);
    \filldraw[color=white] (1, -2) circle (4pt);
    \draw[color=black] (1, -2) circle (4pt);
    \filldraw[color=white] (2, -1) circle (4pt);
    \draw[color=black] (2, -1) circle (4pt);
    \filldraw[color=white] (2, -2) circle (4pt);
    \draw[color=black] (2, -2) circle (4pt);

    \filldraw[color=gray] (-1, 0) circle (4pt);
    \draw[color=black] (-1, 0) circle (4pt);
    \filldraw[color=gray] (-1, 1) circle (4pt);
    \draw[color=black] (-1, 1) circle (4pt);
    \filldraw[color=gray] (-1, 2) circle (4pt);
    \draw[color=black] (-1, 2) circle (4pt);

    \filldraw[color=gray] (0, -1) circle (4pt);
    \draw[color=black] (0, -1) circle (4pt);
    \filldraw[color=gray] (1, -1) circle (4pt);
    \draw[color=black] (1, -1) circle (4pt);
    \filldraw[color=gray] (2, -1) circle (4pt);
    \draw[color=black] (2, -1) circle (4pt);
    
    \node[scale=1] at (3.5, -1.5) {$w_{1}$};
    \node[scale=1] at (-1.3, 3.8) {$w_{2}$};
    \node[scale=1.2] at (2.5, 2.5) {$B^{\vec{m}-\vec{e}_{j}}$ regular};
    \node[scale=1.2] at (-2.5, -2.5) {$\mathbb{B}$};
\end{tikzpicture}
\end{minipage}
\caption{Lattice description of $B^{\vec{m}}$ regularity and $B^{\vec{m}-\vec{e}_{j}}$ regularity for all $j = 1, 2$ and $\vec{m}=(1,1)$}
\end{figure}
\end{footnotesize} \newpage

Thus, the condition that $B^{\vec{m}-\vec{e}_{j}}$ is regular for all
$j$ is stronger than $B^{\vec{m}}$-regularity but weaker than
$B^{\vec{m}-\vec{\delta}}$-regularity. Our objective to use $B^{\vec{m}-\vec{e}_{j}}$ regularity for all $j$ is to make use of this interpolation.

\subsection{Geometric Interpretation of the Difference between Our Approach and the Hering--Schenck--Smith Approach for Property $(N_p)$}

We conclude this section by providing a geometric interpretation of the difference between the regularity assumptions underlying Theorem~\ref{Theorem_6.1} and the corresponding criterion of Hering--Schenck--Smith~\cite{HeringSchenckSmith} for Property~$(N_p)$.\par

As noted in Remark~6.2, for $p \ge 1$, replacing the assumption of $B^{\vec{m}-\vec{e_{j}}}$-regularity on $\mathcal{F}$ with the stronger condition of $B^{\vec{m}-\vec{\delta}}$-regularity yields the identical bound $\vec{m}+(p-1)\vec{\delta}$.  So, if we denote $E_{\vec{\lambda}}$ as the region of permissible line bundles that are $B^{\vec{\lambda}}$-regular. Then :
\begin{equation}
E_{\vec{m}-\vec{\delta}} \subseteq \bigcup_{j=1}^{t} E_{\vec{m}-\vec{e_{j}}} \subseteq E_{\vec{m}}.
\end{equation}

  But the crucial point in the approach of \cite{HeringSchenckSmith} is that they completely rely on the fact that regularity of $\mathcal{O}_{X}$ is non-negative, namely 
\[
\operatorname{reg}(\mathcal{O}_{X})\ge \vec{0},
\]
they obtain the same bound
\[
\vec{m}+(p-1)\vec{\delta}
\]
throughout the larger region $E_{\vec{m}}$. Importantly, if the $\mathrm{Pic}(X)$ is simply generated by the base-point free generators $B_{1},\cdots, B_{t}$, then their region $E_{\vec{m}}$ is a convex set, whereas our region \[
\bigcup_{j=1}^{t} E_{\vec{m}-\vec{e}_{j}}
\] is not convex.

However, this technique in \cite{HeringSchenckSmith} exhibits two key limitations:
\begin{enumerate}
    \item \textbf{Sheaf Specificity:} The drawback of this technique is that it applies only to line bundles of the form $B^{\vec{v}}$ where $\vec{v}\ge \vec{0}$, in particular for $\mathcal{O}_{X}$; but their approach is not being extendable for arbitrary vector bundle $\mathcal{F}$.

    \item \textbf{Range of the Index $p$:}  Another distinction of their technique is, their Theorem for $(N_{p})$ is applicable for $p\ge 1$ and in our set up, if we start with the stronger assumption of $B^{\vec{m}-\vec{\delta}}$, it allows our theorem to extend the bound for $(N_{p})$ down to $p\ge 0$. 
    
    \end{enumerate}

Finally, under our standing hypothesis of $B^{\vec{m}-\vec{e}_{j}}$-regularity in each coordinate direction, the resulting choice of $\vec{m}$ differs from that arising in the Hering--Schenck--Smith framework by at most a translation by $\vec{\delta}$.

\subsection{Geometric Interpretation of the bound for Property $\mathbf{(M_{q})}$}

The two parts of Figure~\ref{Bound_M_q} illustrate the convex regions in $\overline{\mathfrak{B}}$ consisting of line bundles $L=B^{\vec w}$ that satisfy Property $(M_q)$ for a fixed pair of globally
generated line bundles $\vec{B}=(B_1,B_2)$ where $B_{1}, B_{2}$ are $\mathbb{Z}$-linearly independent.
\begin{figure}[ht!]
    \hspace{-0.5cm}
    \begin{minipage}{0.2\textwidth}
\begin{tikzpicture}[row sep = 1.3, column sep = 0.7]
\fill[gray!30!] (1,2) rectangle (4, 4);
\draw[<->] (-.7, 1) -- (4, 1);
\node[scale=0.8] at (4, 0.5) {$B_{1}$};
\draw[<->] (0, 1-.7) -- (0, 4);
\node[scale=0.8] at (-0.5, 4) {$B_{2}$};
\draw[color=black] (1+0.05, 2+0.05) circle (2pt);
\filldraw[color=blue] (1+0.05, 2+0.05) circle (2pt);
\node[scale=0.8] at (4.3, 1+0.5) {$\vec{m}+(q-n-1)\vec{\delta} = (m_{1}+q-n-1, m_{2}+q-n-1)$};
\draw[->] (1, 2) -- (1, 4);
\draw[->] (1, 2) -- (4, 2);
\draw[->, dashed] (0, 1) -- (1, 2);
\draw[->] (0, 1) -- (1.5, 3.5);
\draw[color=black] (1.55, 3.55) circle (2pt);
\filldraw[color=red] (1.55, 3.55) circle (2pt);
\filldraw[color=green] (0.55, 3.55) circle (2pt);
\node[scale=0.8] at (2.3, 3.2) {$\vec{w} = (a, b)$};
\node[scale=0.8] at (3, 2.5) {$(M_{q})$};
\node[scale=0.8] at (2, 0.5) {$(I)$};
\end{tikzpicture}
\end{minipage}
\hspace{5.5cm}
\begin{minipage}{0.2\textwidth}
\begin{tikzpicture}[row sep = 1.3, column sep = 0.7]
\fill[gray!30!] (2-0.5,3-0.5) rectangle (4, 4);
\fill[gray!60!] (2-1,3-1) rectangle (2-0.5, 4);
\fill[gray!60!] (2-1,3-1) rectangle (4, 3-0.5);
\fill[gray!90!] (2-1.5,3-1.5) rectangle (2-1, 4);
\fill[gray!90!] (2-1.5,3-1.5) rectangle (4, 3-1);
\draw[<->] (-.7, 1) -- (4, 1);
\node[scale=0.8] at (4, 0.5) {$B_{1}$};
\draw[<->] (0, 1-.7) -- (0, 4);
\node[scale=0.8] at (-0.5, 4) {$B_{2}$};
\draw[->] (2-0.5, 3-0.5) -- (2-0.5, 4);
\draw[->] (2-0.5, 3-0.5) -- (4, 3-0.5);
\draw[->] (2-1, 3-1) -- (2-1, 4);
\draw[->] (2-1, 3-1) -- (4, 3-1);
\draw[->] (0.5, 2-0.5) -- (0.5, 4);
\draw[->] (0.5, 2-0.5) -- (4, 2-0.5);
\node[scale=0.8] at (3, 3.5) {$(M_{q})$};
\node[scale=0.8] at (2, 0.5) {$(II)$};
\node[scale=0.5] at (4.75, 3-0.5) {$q = n+2$};
\node[scale=0.5] at (4.75, 3-1) {$q = n+1$};
\node[scale=0.5] at (4.75, 3-1.5) {$q = n$};
\draw[->, dashed] (0, 1) -- (1-0.5, 2-0.5);
\node[scale=0.75] at (1-0.2, 2-0.75) {$\vec{m}-\vec{\delta}$};
\end{tikzpicture}
\end{minipage}
\caption{(I) Convex set for $(M_{q})$ for fixed $q.$ (II) Various convex sets for $(M_{q})$ as $q$ varies}.
\label{Bound_M_q}
\end{figure}
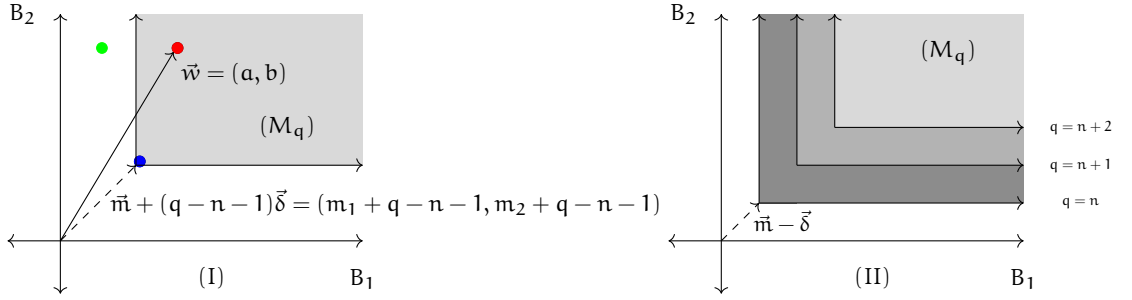
\textbf{ } \par

The first part of Figure~\ref{Bound_M_q} illustrates the convex region corresponding to Property $(M_q)$ for a fixed value of $q$. The red point $\vec{w}=(a,b)$ represents an arbitrary vector in this region; equivalently, the line bundle $L=B^{\vec{w}}$ satisfies Property $(M_q)$. The green point represents a vector $\vec{u}=(c,d)$ outside this region, indicating that the line bundle $B^{\vec{u}}$ does not satisfy Property $(M_q)$ in general (as illustrated by the examples in the next section that show sharpness of our result). The blue point corresponds to the minimal vector in this region,
\[
\vec{m}+(q-n-1)\vec{\delta}
=(m_1+q-n-1,\;m_2+q-n-1).
\]

The second part of Figure~\ref{Bound_M_q} illustrates how this admissible convex region changes as $q$ varies. In particular, the figures corresponding to $q=n,\;n+1,\;n+2$ show how the lower bound
\[
\vec{m}+(q-n-1)\vec{\delta}
\]
changes with $q$, and consequently how the region of line bundles satisfying Property $(M_q)$ varies.

\section{\texorpdfstring{\textbf{Sharpness of the Bound for Property $\mathbf{(M{q})}$}}{Sharpness of the Bound for Property (Mq)}}
\label{sharpness_of_property_Mq}

In this section, we will provide examples that show the optimality of our result on Property $(M_{q})$ ( Corollary \ref{CorollaryM_q}) that we have proved in Section \ref{section_7_mainbody_vanishingtheorem}. 

We begin with a classical example showing that the bound in Corollary \ref{CorollaryM_q} is sharp.
\label{examples_section}
\begin{example}
    Consider the Veronese embedding $\mathbb{P}^{2} \hookrightarrow \mathbb{P}^{\begin{pmatrix}
    w+2 \\
    2
\end{pmatrix}-1}$ of $\mathbb{P}^{2}$ defined by $\mathcal{O}(w) = (\mathcal{O}(1))^{\otimes w}.$ Take
$B = \left\{ H = \mathcal{O}(1) \right\}.$
Then $K_{\mathbb{P}^{2}} = \mathcal{O}_{\mathbb{P}^{2}}(-3)$ is $H^{\otimes 3}-$regular w.r.t. $B.$ That is it is $B^{4\vec{\delta}-\vec{e}_{j}}-$regular where $\delta = e_{j} = (1).$ Bound from Corollary \ref{CorollaryM_q}: $\mathcal{O}(w) = B^{\vec{w}}$ satisfies Property$-(M_{q})$ for $\vec{w} = (w) \geq \vec{m}+(q-n-1)\vec{\delta} = (4+(q-3)) = (q+1).$ In other words, the highest $q$ for which $\mathcal{O}(w)$ satisfies $(M_{q})$ is $q_{\max} = w-1.$ By Noether's theorem, $\mathrm{gon}(C) = w-1$ for any $C$ in $|wH|.$ Thus, by \cite{Basu} this bound is sharp. From the syzygy database \href{https://syzygydata.com/p2/}{Syzygy Data},  on syzygies of $\mathbb{P}^{2},$ we obtain the following Betti tables for $2 \leq w \leq 8.$
\flushleft{
        \scriptsize{
\begin{table}[!htb]
    \caption{Betti tables of Veronese embeddings $\mathcal{O}(w)$ for $\mathbb{P}^{2}$ $\equiv [t_{0}: t_{1}: t_{2}] \mapsto [t^{w}_{0} : t^{w-1}_{0}t_{1} : \cdots : t^{w}_{2}]$}
    \begin{minipage}{.5\linewidth}
      \centering
        \begin{tabular}{c | c c c c c}
        \cline{2-6}
        & \multicolumn{4}{c}{Betti table of $\mathcal{O}(2)$} \\
        \cline{1-6}
$j  \diagdown i$ & 0 & 1 & 2 & 3 & 4 \\
\cline{1-6}
&  &  &  & \\
0 & \cellcolor[gray]{0.9} \color{blue}{$1$} & $0$ & $0$ & $0$ & $0$  \\
1 & \cellcolor[gray]{0.9} \color{blue}{$0$} & $6$ & $8$ & $3$ & $0$ \\
2 & \cellcolor[gray]{0.9} \color{blue}{$0$} & \cellcolor[gray]{0.9} \color{blue}{$0$} & \cellcolor[gray]{0.9} \color{blue}{$0$} & \cellcolor[gray]{0.9} \color{blue}{$0$} & $0$ \\
\hline
& $E_{0}$ & $E_{1}$ & $E_{2}$ & $E_{3}$ & $E_{4}$
\end{tabular}
    \end{minipage}%
    \begin{minipage}{.5\linewidth}
      \centering
        \begin{tabular}{c | c c c c c c c c c}
        \cline{2-10}
        & \multicolumn{8}{c}{Betti table of $\mathcal{O}(3)$} \\
        \cline{1-10}
$j  \diagdown i$ & 0 & 1 & 2 & 3 & 4 & 5 & 6 & 7 & 8\\
\cline{1-10}
&  &  &  &  &  &  & & & \\
0 & \cellcolor[gray]{0.9} \color{blue}{$1$} & $0$ & $0$ & $0$ & $0$ & $0$ & $0$ & $0$ & $0$ \\
1 & \cellcolor[gray]{0.9} \color{blue}{$0$} & $27$ & $105$ & $189$ & $189$ & $105$ & $27$ & \cellcolor{yellow} {\color{red} $0$} & \cellcolor{yellow} \color{red} $0$ \\
2 & \cellcolor[gray]{0.9} \color{blue}{$0$} & \cellcolor[gray]{0.9} \color{blue}{$0$} & \cellcolor[gray]{0.9} \color{blue}{$0$} & \cellcolor[gray]{0.9} \color{blue}{$0$} & \cellcolor[gray]{0.9} \color{blue}{$0$} & \cellcolor[gray]{0.9} \color{blue}{$0$} & \cellcolor[gray]{0.9} \color{blue}{$0$} & $1$ & $0$
\end{tabular}
    \end{minipage}
\end{table}}}
\vspace{0.2cm}
The above two Betti tables correspond to the resolutions
{\scriptsize\begin{equation}
    \mathcal{O}_{\mathbb{P}^{2}}(2) \equiv 0 \rightarrow S(-4)^{\oplus 3}\rightarrow S(-3)^{\oplus 8} \rightarrow S(-2)^{\oplus 6} \rightarrow S \rightarrow S_{X} \rightarrow 0
\end{equation}
\vspace{-0.2cm}
\begin{equation}
   \mathcal{O}_{\mathbb{P}^{2}}(3) \equiv 0 \rightarrow {\color{red} 0 \ } \oplus S(-9)^{\oplus 1}  \rightarrow S(-7)^{\oplus 27} \rightarrow \cdots \rightarrow S(-3)^{\oplus 105} \rightarrow S(-2)^{\oplus 27} \rightarrow S \rightarrow S_{X} \rightarrow 0
\end{equation}}\par

The Betti tables below, illustrate the sharpness of the bound obtained in Corollary \ref{CorollaryM_q}. In particular, they exhibit the vanishing pattern predicted by the condition $q_{\max}=w-1$, confirming that Property $(M_q)$ holds precisely up to the predicted boundary.

\begin{small}
 \begin{table}[!htb]
        \begin{tabular}{c | c c c c c c c c c c c c c c}
        \cline{2-15}
        & \multicolumn{8}{c}{Betti table of $\mathcal{O}(4)$} \\
        \cline{1-15}
$j  \diagdown i$ & $0$ & $1$ & $2$ & $3$ & $4$ & $5$ & $6$ & $7$ & $8$ & $9$ & $10$ & $11$ & $12$ & $13$ \\
\cline{1-15}
&  &  &  &  &  &  & & & & & & & & \\
 $0$ & \cellcolor[gray]{0.9} \color{blue} $1$ & $0$ & $0$ & $0$ & $0$ & $0$ & $0$ & $0$ & $0$ & $0$ & $0$ & $0$ & $0$ & $0$ \\
$1$ & \cellcolor[gray]{0.9} \color{blue} $0$ & $75$ & $536$ & $1947$ & $4488$ & $7095$ & $7920$ & $6237$ & $3344$ & $1089$ & $120$ & \cellcolor{yellow} \color{red} $0$ & \cellcolor{yellow} \color{red} $0$ & \cellcolor{yellow} \color{red} $0$ \\
$2$ & \cellcolor[gray]{0.9} \color{blue} $0$ & \cellcolor[gray]{0.9} \color{blue} $0$ & \cellcolor[gray]{0.9} \color{blue} $0$ & \cellcolor[gray]{0.9} \color{blue} $0$ & \cellcolor[gray]{0.9} \color{blue} $0$ & \cellcolor[gray]{0.9} \color{blue} $0$ & \cellcolor[gray]{0.9} \color{blue} $0$ & \cellcolor[gray]{0.9} \color{blue} $0$ & \cellcolor[gray]{0.9} \color{blue} $0$ & \cellcolor[gray]{0.9} \color{blue} $0$ & $55$ & $24$ & $3$ & $0$
\end{tabular}
\end{table}
\end{small}
\end{example} \par

\begin{example}
    In this example, we show that the bound for Property $(M_{q})$ in Corollary \ref{CorollaryM_q} is optimal for any globally generated and ample line bundle on $X = \mathbb{P}^{1} \times \mathbb{P}^{1}$. \\
    For this, we fix the notation of Example \ref{PmxPn_example}. 
    That is, $B_{1} = p^{\ast}_{1}(\mathcal{O}_{\mathbb{P}^{1}}(1))$ and $B_{2} = p^{\ast}_{2}(\mathcal{O}_{\mathbb{P}^{1}}(1))$ are pullbacks of hyperplanes along the projections of the corresponding factors. Also, let $B = (B_{1}, B_{2})$. Then, by Example \ref{PmxPn_example} 
    any line bundle $L= B^{(a, b)} = \mathcal{O}_{\mathbb{P}^{1} \times \mathbb{P}^{1}}(a, b)$ is $B^{(-a, -b)}-$regular with respect to $B.$ Thus, $\omega_{X} = \mathcal{O}_{\mathbb{P}^{1} \times \mathbb{P}^{1}}(-2, -2)$ is $(2, 2)-$regular with respect to $B$. Moreover, we saw in \ref{PmxPn_example} that $\vec{m} = (3, 3)$ is the smallest integer vector such that $\omega_{X}$ is $B^{\vec{m}-\vec{e}_{j}}-$regular for $j = 1, 2.$ Then, by Corollary \ref{corollaryMq} we have, for $\vec{w} = (a, b) > \vec{0},$  line bundle $L = B^{\vec{w}}$ is ample and globally generated and it satisfies Property $(M_{q})$ if $(n-1)\vec{w} \geq \vec{m}+(q-n-1)\vec{\delta}$ ; i.e, if $(a, b) \geq (3, 3)+(q-3, q-3) = (q, q).$ Therefore, for a fixed $q$,  ample line bundle ${B}^{(a,b)}$ satisfies Property$-(M_{q})$ if $a\ge q$ and $b\ge q$. The figure below demonstrates this nicely. 

\begin{figure}[ht!]
\begin{minipage}{0.3\textwidth}
\begin{tikzpicture}
\fill[gray!30!] (1,2) rectangle (4, 4);
\draw[<->] (-.7, 1) -- (4, 1);
\node[scale=0.8] at (4, 0.5) {$B_{1}$};
\draw[<->] (0, 1-.7) -- (0, 4);
\node[scale=0.8] at (-0.5, 4) {$B_{2}$};
\draw[color=black] (1+0.05, 2+0.05) circle (2pt);
\filldraw[color=blue] (1+0.05, 2+0.05) circle (2pt);
\node[scale=0.8] at (2.3, 1+0.5) {$\vec{m}+(q-n-1)\vec{\delta} = (q, q)$};
\draw[->] (1, 2) -- (1, 4);
\draw[->] (1, 2) -- (4, 2);
\draw[->, dashed] (0, 1) -- (1, 2);
\draw[->] (0, 1) -- (1.5, 3.5);
\draw[color=black] (1.55, 3.55) circle (2pt);
\filldraw[color=red] (1.55, 3.55) circle (2pt);
\filldraw[color=green] (0.55, 3.55) circle (2pt);
\node[scale=0.8] at (2.3, 3.2) {$\vec{w} = (a, b)$};
\node[scale=0.8] at (3, 2.5) {$(M_{q})$};
\node[scale=0.8] at (2, 0.5) {$(I)$};
\end{tikzpicture}
\end{minipage}
\hspace{2cm}
\begin{minipage}{0.2\textwidth}
\begin{tikzpicture}
\fill[gray!70!] (2,1.5) rectangle (4, 4);
\draw[<->] (-.7, 1) -- (4, 1);
\node[scale=0.8] at (4, 0.5) {$B_{1}$};
\draw[<->] (0, 1-.7) -- (0, 4);
\node[scale=0.8] at (-0.5, 4) {$B_{2}$};
\draw[color=black] (2+0.05, 1.5+0.05) circle (2pt);
\filldraw[color=blue] (2+0.05, 1.5+0.05) circle (2pt);
\node[scale=0.8] at (1.3, 1.6) {$(q, q)$};
\draw[->] (2, 1.5) -- (2, 4);
\draw[->] (2, 1.5) -- (4, 1.5);
\draw[->, dashed] (0, 1) -- (2, 1.5);
\draw[->] (0, 1) -- (1.5, 3.5);
\draw[color=black] (1.55, 3.55) circle (2pt);
\filldraw[color=red] (1.55, 3.55) circle (2pt);
\filldraw[color=green] (0.55, 3.55) circle (2pt);
\node[scale=0.8] at (2.3, 3.2) {$\vec{w} = (a, b)$};
\node[scale=0.8] at (3, 2.5) {$(M_{q})$};
\node[scale=0.8] at (2, 0.5) {$(II)$};
\end{tikzpicture}
\label{IntBound2}
\end{minipage}
\caption{Different Convex Regions for $(M_{q})$ as q changes}
\end{figure}
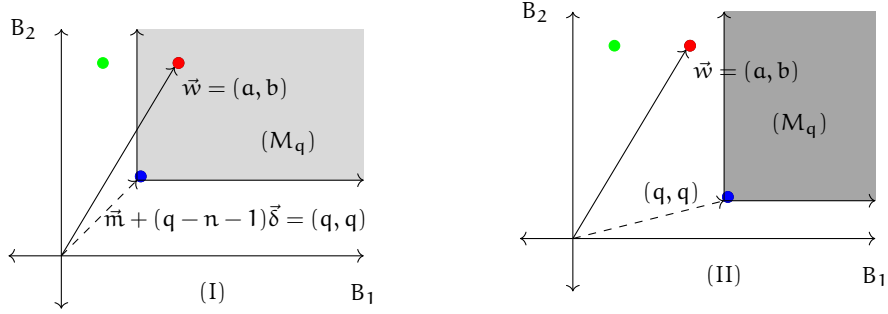

In the above figures, the blue circle represents the lattice point $(q,q)$. The shaded region consists of all lattice points $(a,b)$ satisfying $a\geq q$ and $b\geq q$, corresponding to all globally generated line bundles $B^{(a,b)}$ that satisfy Property $(M_{q})$ for the fixed value of $q$. The green circle represents another base-point-free line bundle; however, since it lies outside the shaded region, it does not satisfy Property $(M_{q})$.

We now show that the bound obtained from Corollary \ref{CorollaryM_q} is sharp. 
More precisely, we prove that the boundary of the above region is optimal by showing that Property $(M_q)$ fails immediately outside this region.

    By symmetry, we may assume that $b\geq a$. Then, by the result of Martens in \cite{Martens}, for any smooth curve $C\in |aC_{0}+bf|$ on $\mathbb{P}^{1}\times \mathbb{P}^{1}$, we have
\[
\mathrm{gon}(C)=a.
\]
Recently, D. Basu in \cite{Basu} generalized the Gonality Conjecture for curves to rational surfaces. In particular, he proved that for a base-point-free and ample line bundle $L$ on a rational surface $X$ satisfying
\[
-\omega_{X}\cdot L\geq q+2,
\]
the line bundle $L$ satisfies Property $(M_q)$ if and only if
\[
q\leq \mathrm{gon}_{\max}(L),
\]
where $\mathrm{gon}_{\max}(L)$ denotes the maximum gonality among smooth curves in the linear series $|L|$. 

Applying these results to the present example, we obtain that for $b\geq a$, the maximal value of $q$ for which $L=B^{(a,b)}$ satisfies Property $(M_q)$ is
\[
q_{\max}=a,
\]
which is precisely the gonality of a smooth curve in $|L|$. Therefore, when $b\geq a$, the line bundle $L=B^{(a,b)}$ satisfies Property $(M_a)$ but does not satisfy Property $(M_{a+1})$. Hence, the condition obtained from Corollary \ref{corollaryMq} is optimal: $B^{(a,b)}$ satisfies Property $(M_q)$ for $q\leq a$, but fails for $q=a+1$. This shows that the bound is sharp for every ample and base-point-free line bundle $L$ on $\mathbb{P}^{1}\times\mathbb{P}^{1}$.
    \par

Next, we record the Betti tables of 
$L=\mathcal{O}_{\mathbb{P}^{1}\times\mathbb{P}^{1}}(a,b)=B^{(a,b)}$
for $(a,b)=(2,2), (2, 3),(2, 4)$ and $(3,3)$.
These examples illustrate the sharpness of the bound obtained in Corollary \ref{corollaryMq}.

The Betti tables below are taken from \cite{JBruceP1xP1} for the line bundles 
$L\equiv aC_{0}+bf$ on $\mathbb{P}^{1}\times\mathbb{P}^{1}$. 
They exhibit the vanishing pattern predicted by Corollary \ref{CorollaryM_q}. 
More precisely, for the listed values of $(a,b)$, the last nonzero syzygies occur exactly at the boundary predicted by the condition $q\leq a$, confirming that Property $(M_q)$ holds for $q\leq a$ and fails for $q>a$.
     
    \flushleft{
    \begin{scriptsize}
\begin{table}[!htb]
    \caption{Betti tables of $L = aC_{0}+bf$ for $(a, b) = (2, 2), (2, 3), (2, 4) \text{ and } (3, 3)$}
    \begin{minipage}{.5\linewidth}
      \centering
        \begin{tabular}{c | c c c c c c c c c}
        \cline{2-9}
        & \multicolumn{7}{c}{Betti table of $(2, 2)$} \\
        \cline{1-9}
$j  \diagdown i$ & $0$ & $1$ & $2$ & $3$ & $4$ & $5$ & $6$ & $7$ \\
\cline{1-9}
&  &  &  &  &  &  & & \\
$0$ &  \cellcolor[gray]{0.9} \color{blue} $1$ & $0$ & $0$ & $0$ & $0$ & $0$ & $0$ & $0$ \\
$1$ & \cellcolor[gray]{0.9} \color{blue} $0$ & $20$ & $64$ & $90$ & $64$ & $20$ & \cellcolor{yellow}  \color{red} $0$ & \cellcolor{yellow}  \color{red} $0$ \\
$2$ & \cellcolor[gray]{0.9} \color{blue} $0$ & \cellcolor[gray]{0.9} \color{blue} $0$ & \cellcolor[gray]{0.9} \color{blue} $0$ & \cellcolor[gray]{0.9} \color{blue} $0$ & \cellcolor[gray]{0.9} \color{blue} $0$ & \cellcolor[gray]{0.9} \color{blue} $0$ & $1$ & $0$
\end{tabular}
    \end{minipage}%
    \begin{minipage}{.5\linewidth}
    \begin{tiny}
      \centering
        \begin{tabular}{c | c c c c c c c c c c c}
        \cline{2-12}
        & \multicolumn{8}{c}{Betti table of $(2, 3)$} \\
        \cline{1-12}
$j  \diagdown i$ & $0$ & $1$ & $2$ & $3$ & $4$ & $5$ & $6$ & $7$ & $8$ & $9$ & $10$ \\
\cline{1-12}
&  &  &  &  &  &  & & & & & \\
$0$ & \cellcolor[gray]{0.9} \color{blue} $1$ & $0$ & $0$ & $0$ & $0$ & $0$ & $0$ & $0$ & $0$ & $0$ & $0$ \\
$1$ & \cellcolor[gray]{0.9} \color{blue} $0$ & $43$ & $222$ & $558$ & $840$ & $798$ & $468$ & $147$ & $8$ & \cellcolor{yellow}  \color{red} $0$ & \cellcolor{yellow}  \color{red} $0$ \\
$2$ & \cellcolor[gray]{0.9} \color{blue} $0$ & \cellcolor[gray]{0.9} \color{blue} $0$ & \cellcolor[gray]{0.9} \color{blue} $0$ & \cellcolor[gray]{0.9} \color{blue} $0$ & \cellcolor[gray]{0.9} \color{blue} $0$ & \cellcolor[gray]{0.9} \color{blue} $0$ & \cellcolor[gray]{0.9} \color{blue} $0$ & \cellcolor[gray]{0.9} \color{blue} $0$ & $9$ & $2$ & $0$
\end{tabular}
\end{tiny}
    \end{minipage} 
\end{table}
\begin{table}[!htb]
        \begin{tabular}{c | c c c c c c c c c c c c c c}
        \cline{2-15}
        & \multicolumn{8}{c}{Betti table of $(2, 4)$} \\
        \cline{1-15}
$j  \diagdown i$ & $0$ & $1$ & $2$ & $3$ & $4$ & $5$ & $6$ & $7$ & $8$ & $9$ & $10$ & $11$ & $12$ & $13$ \\
\cline{1-15}
&  &  &  &  &  &  & & & & & & & & \\
$0$ & \cellcolor[gray]{0.9} \color{blue} $1$ & $0$ & $0$ & $0$ & $0$ & $0$ & $0$ & $0$ & $0$ & $0$ & $0$ & $0$ & $0$ & $0$ \\
$1$ & \cellcolor[gray]{0.9} \color{blue} $0$ & $75$ & $536$ & $1947$ & $4488$ & $7095$ & $7920$ & $6237$ & $3344$ & $1089$ & $120$ & $11$ & \cellcolor{yellow}  \color{red} $0$ & \cellcolor{yellow}  \color{red} $0$ \\
$2$ & \cellcolor[gray]{0.9} \color{blue} $0$ & \cellcolor[gray]{0.9} \color{blue} $0$ & \cellcolor[gray]{0.9} \color{blue} $0$ & \cellcolor[gray]{0.9} \color{blue} $0$ & \cellcolor[gray]{0.9} \color{blue} $0$ & \cellcolor[gray]{0.9} \color{blue} $0$ & \cellcolor[gray]{0.9} \color{blue} $0$ & \cellcolor[gray]{0.9} \color{blue} $0$ & \cellcolor[gray]{0.9} \color{blue} $0$ & \cellcolor[gray]{0.9} \color{blue} $0$ & $66$ & $24$ & $3$ & $0$
\end{tabular}
\end{table}
\end{scriptsize}}
 \par
\begin{tiny}
\begin{flushleft}
\resizebox{\textwidth}{!}{
        \begin{tabular}{c | c c c c c c c c c c c c c c c}
    \cline{2-16}
    & \multicolumn{3}{c}{ } & \multicolumn{8}{c}{Betti table of $(3,3)$} & \multicolumn{3}{c}{ } \\
    \cline{1-16}
    $j \diagdown i$ & 0 & 1 & 2 & 3 & 4 & 5 & 6 & 7 & 8 & 9 & 10 & 11 & 12 & 13 & 14 \\
    \cline{1-16}
    \\[-6pt]
    0 & \cellcolor[gray]{0.9} \color{blue} $1$ & $0$ & $0$ & $0$ & $0$ & $0$ & $0$ & $0$ & $0$ & $0$ & $0$ & $0$ & $0$ & $0$ & $0$  \\
    1 & \cellcolor[gray]{0.9} \color{blue} $0$ & $87$ & $676$ & $2691$ & $6864$ & $12155$ & $15444$ & $14157$ & $9152$ & $3861$ & $780$ & $22$ & \cellcolor{yellow}  \color{red} $0$ & \cellcolor{yellow}  \color{red} $0$ & \cellcolor{yellow}  \color{red} $0$ \\
    2 & \cellcolor[gray]{0.9} \color{blue} $0$ & \cellcolor[gray]{0.9} \color{blue} $0$ & \cellcolor[gray]{0.9} \color{blue} $0$ & \cellcolor[gray]{0.9} \color{blue} $0$ & \cellcolor[gray]{0.9} \color{blue} $0$ & \cellcolor[gray]{0.9} \color{blue} $0$ & \cellcolor[gray]{0.9} \color{blue} $0$ & \cellcolor[gray]{0.9} \color{blue} $0$ & \cellcolor[gray]{0.9} \color{blue} $0$ & \cellcolor[gray]{0.9} \color{blue} $0$ & $165$ & $144$ & $39$ & $4$ & $0$ \\
\end{tabular}}
\end{flushleft}
\end{tiny}
\end{example} 
\vspace{0.25cm}

\section{\textbf{Hierarchy of Koszul Cohomology: Interpretation of Corollary 7.3}}
\label{section_12_hierarchy_and_slope}

In this section, we discuss further applications and interpretations of Corollary \ref{CorollaryKoszulCohomology}. It turns out that the scope of this corollary extends well beyond the applications presented earlier in this article. First, we show that Theorem 2.2 of \cite{GreenII} follows from Corollary \ref{CorollaryKoszulCohomology}.

\begin{theorem}[Mark L. Green, {\cite[Theorem 2.2]{GreenII}}]
\label{theorem:green-koszul-cohomology-projective-space}
For $k,d \in \mathbb{Z}$ with $d \ge 1$,
\[
K_{p,q}(\mathbb{P}^{r}, \mathcal{O}_{\mathbb{P}^{r}}(k), \mathcal{O}_{\mathbb{P}^{r}}(d))=0
\quad\text{if}\quad
k+(q-1)d\ge p.
\]
\end{theorem}

\begin{proof}
Set $B:=\mathcal{O}_{\mathbb{P}^{r}}(1)$, $L:=B^{\otimes d}=\mathcal{O}_{\mathbb{P}^{r}}(d)$, and $F:=\mathcal{O}_{\mathbb{P}^{r}}(k)$. We verify that $F$ satisfies the regularity hypothesis required to invoke Corollary~\ref{CorollaryKoszulCohomology}.

To this end, we show that $F$ is $B^{\otimes (-k)}$-regular. Indeed,
\[
H^{i}\!\left(F \otimes B^{\otimes (-k-i)}\right)
=
H^{i}\!\left(\mathcal{O}_{\mathbb{P}^{r}}(-i)\right)=0
\quad\text{for all } i>0,
\]
by the standard cohomology vanishing on projective space. Thus $F$ is $B^{\otimes (-k)}$-regular. Rewriting this in the form required by Corollary~\ref{CorollaryKoszulCohomology}, we set $m=-k+1$, so that $F$ is $B^{\otimes (m-1)}$-regular.

Applying Corollary~\ref{CorollaryKoszulCohomology}, we obtain
\[
K_{p,q}\bigl(\mathbb{P}^{r},F,L\bigr)
=
K_{p,q}\bigl(\mathbb{P}^{r},\mathcal{O}_{\mathbb{P}^{r}}(k),\mathcal{O}_{\mathbb{P}^{r}}(d)\bigr)=0
\]
whenever
\[
(q-1)d\ge m+(p-1).
\]
Substituting $m=-k+1$ gives
\[
(q-1)d\ge -k+1+(p-1),
\]
which is equivalent to
\[
k+(q-1)d\ge p.
\]
This is precisely the claimed vanishing.
\end{proof}

\begin{remark}
The above argument differs substantially from the original proof of \cite[Theorem 2.2]{GreenII}. Green's approach is more representation-theoretic and homological in nature: in \cite{GreenII}, the Koszul cohomology groups are analyzed directly via the Koszul complex, using exact sequences, spectral sequences, and an induction on the homological index. In contrast, the present proof bypasses this analysis entirely by encoding the vanishing into a regularity condition and invoking a general vanishing theorem. More precisely, the argument proceeds via Castelnuovo--Mumford regularity and follows formally from Corollary~\ref{CorollaryKoszulCohomology}. This provides a uniform and conceptually streamlined approach to the result.
\end{remark}

We present two interpretations of Corollary~\ref{CorollaryKoszulCohomology}: a hierarchical viewpoint and a slope-based geometric interpretation.

\paragraph{\textbf{Hierarchical interpretation:}}
We now describe the geometric hierarchy of line bundles arising from Corollary~\ref{CorollaryKoszulCohomology}, which reveals an important structural feature of the vanishing theorem.

Suppose $F$ is a vector bundle on $X$ which is $B^{\vec{m}-\vec{e_{j}}}$-regular. Fix a homological index $p$. Then Corollary~\ref{CorollaryKoszulCohomology} asserts that for a line bundle $L=B^{\vec{w}}$, one has
\[
K_{p,q}(X,F,L)=0 \quad \text{whenever} \quad (q-1)\vec{w}\ge \vec{m}+(p-1)\vec{\delta}.
\]
Unwinding this condition for successive values of $q$, we obtain the nested sequence of vanishing regions
\[
\vec{w}\ge \frac{1}{q-1}\bigl(\vec{m}+(p-1)\vec{\delta}\bigr),
\qquad q=2,3,4,\ldots .
\]
This pattern makes the underlying structure transparent: as the weight $q$ increases, the corresponding vanishing condition becomes progressively weaker. Equivalently, the region in the parameter space of $\vec{w}$ for which $K_{p,q}(X,F,L)=0$ enlarges with $q$. In particular, the vanishing region for weight $q$ contains that for weight $q-1$, which in turn contains that for weight $q-2$, and so forth. On the other hand, for a fixed weight $q$, the hierarchy is reversed with respect to the homological position $p$ in the resolution: the vanishing region corresponding to stage $p$ is contained in that corresponding to stage $p-1$, and so on. Thus, the vanishing regions exhibit a two-directional hierarchy, as illustrated in Figure~\ref{IntBound} below. We refer to this hierarchical pattern of vanishing as the \emph{$K_{p,q}$-hierarchy}.

 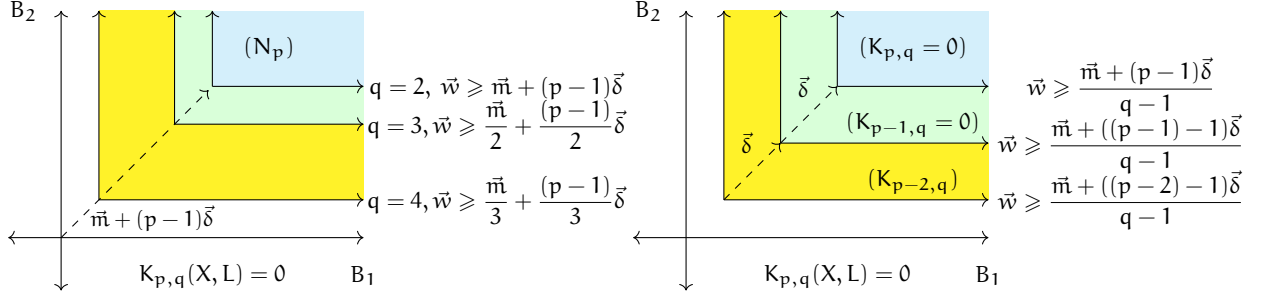
\begin{figure}[!ht]
  \begin{adjustwidth*}{}{-0.5em} 
  \begin{minipage}{0.3\textwidth}
\begin{tikzpicture}[row sep = 1.3, column sep = 0.7]
\fill[cyan!15!] (2,3) rectangle (4, 4);
\fill[green!15!] (2-0.5,3-0.5) rectangle (2, 4);
\fill[green!15!] (2-0.5,3-0.5) rectangle (4, 3);
\fill[yellow!90!] (2-1.5,3-1.5) rectangle (2-0.5, 4);
\fill[yellow!90!] (2-1.5,3-1.5) rectangle (4, 3-0.5);
\draw[<->] (-.7, 1) -- (4, 1);
\node[scale=0.8] at (4, 0.5) {$B_{1}$};
\draw[<->] (0, 1-.7) -- (0, 4);
\node[scale=0.8] at (-0.5, 4) {$B_{2}$};
\draw[->] (2, 3) -- (2, 4);
\draw[->] (2, 3) -- (4, 3);
\draw[->] (2-0.5, 3-0.5) -- (2-0.5, 4);
\draw[->] (2-0.5, 3-0.5) -- (4, 3-0.5);
\draw[->] (0.5, 2-0.5) -- (0.5, 4);
\draw[->] (0.5, 2-0.5) -- (4, 2-0.5);
\node[scale=0.8] at (2.75, 3.5) {$(N_{p})$};
\node[scale=0.8] at (2, 0.5) {$K_{p, q}(X, L) = 0$};
\node[scale=0.8] at (5.75, 3) {$q = 2, \ \vec{w} \geq \vec{m}+(p-1)\vec{\delta}$};
\node[scale=0.8] at (5.75, 3-0.5) {$  q = 3,  \vec{w} \geq \dfrac{\vec{m}}{2}+\dfrac{(p-1)}{2}\vec{\delta}$};
\node[scale=0.8] at (5.75, 3-1.5) {$q = 4, \vec{w} \geq \dfrac{\vec{m}}{3}+\dfrac{(p-1)}{3}\vec{\delta}$};
\draw[->, dashed] (0, 1) -- (2-0.05, 3-0.05);
\node[scale=0.75] at (1+0.2, 2-0.75) {$\vec{m}+(p-1)\vec{\delta}$};
\end{tikzpicture}
\end{minipage}
\hspace{3.5cm}
 \begin{minipage}{0.3\textwidth}
\begin{tikzpicture}[row sep = 1.3, column sep = 0.7]
\fill[cyan!15!] (2,3) rectangle (4, 4);
\fill[green!15!] (2-0.75,3-0.75) rectangle (2, 4);
\fill[green!15!] (2-0.75,3-0.75) rectangle (4, 3);
\fill[yellow!90!] (2-1.5,3-1.5) rectangle (2-0.75, 4);
\fill[yellow!90!] (2-1.5,3-1.5) rectangle (4, 3-0.75);
\draw[<->] (-.7, 1) -- (4, 1);
\node[scale=0.8] at (4, 0.5) {$B_{1}$};
\draw[<->] (0, 1-.7) -- (0, 4);
\node[scale=0.8] at (-0.5, 4) {$B_{2}$};
\draw[->] (2, 3) -- (2, 4);
\draw[->] (2, 3) -- (4, 3);
\draw[->] (2-0.75, 3-0.75) -- (2-0.75, 4);
\draw[->] (2-0.75, 3-0.75) -- (4, 3-0.75);
\draw[->] (2-1.5, 3-1.5) -- (2-1.5, 4);
\draw[->] (2-1.5, 3-1.5) -- (4, 3-1.5);
\node[scale=0.8] at (3, 3.5) {$(K_{p, q} = 0)$};
\node[scale=0.8] at (3, 2.5) {$(K_{p-1, q} = 0)$};
\node[scale=0.8] at (3, 1.75) {$(K_{p-2, q})$};
\node[scale=0.8] at (2, 0.5) {$K_{p, q}(X, L) = 0$};
\node[scale=0.8] at (5.75, 3) {$ \vec{w} \geq \dfrac{\vec{m}+(p-1)\vec{\delta}}{q-1}$};
\node[scale=0.8] at (5.75, 3-0.75) {$\vec{w} \geq \dfrac{\vec{m}+((p-1)-1)\vec{\delta}}{q-1}$};
\node[scale=0.8] at (5.75, 3-1.5) {$\vec{w} \geq \dfrac{\vec{m}+((p-2)-1)\vec{\delta}}{q-1}$};
\draw[->, dashed] (2-1.5, 3-1.5) -- (2-0.75, 3-0.75);
\node[scale=0.75] at (2-0.75-0.45, 3-0.75) {$\vec{\delta}$};
\draw[->, dashed] (2-0.75, 3-0.75) -- (2, 3);
\node[scale=0.75] at (2-0.45, 3) {$\vec{\delta}$};
\end{tikzpicture}
\end{minipage}

\caption{Hierarchy of vanishing regions for Koszul cohomology groups
$K_{p,q}(X,L)$: variation with respect to the weight $q$ for a fixed stage
of the resolution $p$, and variation with respect to the stage $p$ for a
fixed weight $q$, with $L=B^{\vec{w}}$ and $F=\mathcal{O}_X$.}
\label{IntBound2b}
\end{adjustwidth*}
\end{figure}

\paragraph{\textbf{Slope Interpretation:}}
We now describe the above hierarchy from a geometric viewpoint by reducing the vanishing condition to a one-dimensional setting. 

The vanishing condition 
\begin{equation}\label{VanishingInequality}
(q-1)\vec{w}\ge \vec{m}+(p-1)\vec{\delta}
\end{equation}
 in Corollary \ref{CorollaryKoszulCohomology}, 
defines a family of linear inequalities in the $(w_{i},p)$-plane for each $i=1 \cdots t$, together with the corresponding half-planes of vanishing. As the weight $q$ increases, the slopes of the boundary lines $w_{i}=\frac{m}{q-1}+\frac{p-1}{q-1}$, $\frac{1}{q-1}$, decrease, and the associated vanishing regions expand. In the limit as $q\to\infty$, the slope approaches zero for each component, reflecting the fact that higher weight syzygies require progressively weaker conditions. In contrast, the case $q=2$ gives the steepest boundary and corresponds to the strongest condition in the hierarchy, namely property $(N_p)$.

In order to actually describe this interpretation, it is therefore enough to look at a single component of the vector notation. By symmetry, the rest follows. Hence, we restrict to the case of Castelnuovo--Mumford regularity with respect to a single base point free line bundle $L$. In this setting, the multigraded notation reduces to the usual one-dimensional situation, and we may work with scalars. Let $L=B^{\otimes w}$, and consider the $(w,p)$-plane.

This behavior is illustrated in the figure below.
\begin{figure}[H]
\centering
\begin{tikzpicture}[scale=0.8]


\draw[->, thick] (-6,0) -- (10,0) node[right] {$p$};
\draw[->, thick] (0,0) -- (0,8) node[above] {$w$};

\draw[thick] (0,0) -- (0,8);
\begin{scope}
\clip (-6,0) rectangle (10,8);

\def\XMIN{-6}
\def\XMAX{10}
\def\YTOP{12}


\draw[thick, green!80!black] (\XMIN,{\XMIN+4}) -- (4,8);

\draw[thick, yellow!85!black] (\XMIN,{(\XMIN+4)/2}) -- (\XMAX,{(\XMAX+4)/2});

\draw[thick, blue!85!black] (\XMIN,{(\XMIN+4)/3}) -- (\XMAX,{(\XMAX+4)/3});

\filldraw (0,4) circle (2pt);
\node[left] at (0,4.25) {$(0,6)$};

\filldraw (0,2) circle (2pt);
\node[left] at (0,2.25) {$(0,3)$};

\filldraw (0,{4/3}) circle (2pt);
\node[left] at (0,{2/3}) {$(0,2)$};

\coordinate (P) at (-4,0);
\filldraw (P) circle (2.2pt);
\node[below left] at (P) {$P(-4,0)$};


\fill[blue!20, opacity=0.35]
(\XMIN,{(\XMIN+4)/3}) --
(\XMAX,{(\XMAX+4)/3}) --
(\XMAX,\YTOP) --
(\XMIN,\YTOP) -- cycle;

\fill[red!20, opacity=0.35]
(\XMIN,{(\XMIN+4)/2}) --
(2*\YTOP-4,\YTOP) --
(\XMAX,{(\XMAX+4)/2}) --
(\XMAX,\YTOP) --
(\XMIN,\YTOP)
(\XMIN,\YTOP) -- cycle;

\fill[green!20, opacity=0.35]
(\XMIN,{\XMIN+4}) --
(\YTOP-4,\YTOP) --
(\XMIN,\YTOP)
(4,8)
(\XMIN,\YTOP) -- cycle;

\node[rotate=25, font=\bfseries] at (3,5) {\small  $R_2:\; w\ge \tfrac{1}{2}p+3$};
\node[rotate=25, font=\bfseries] at (5.5,4) {\small $R_3:\; w\ge \tfrac{1}{3}p+2$};

\node[rotate=25, font=\bfseries] at (-2.5,5) {\small $R_1:\; w\ge p+6,$ \hspace{0.1cm} $q = 2$};
\end{scope}
\end{tikzpicture}

\caption{Slope interpretation of the $K_{p,q}$-hierarchy for $m=7$. Each line is the boundary of the inequality $
w \ge \frac{p+(m-1)}{q-1} = \frac{p+6}{q-1},$ and the corresponding vanishing region is the half-plane lying above the line. As $q$ increases, the slope $\frac{1}{q-1}$ decreases, causing the vanishing regions to enlarge. The shaded regions overlap, illustrating the inclusions $\operatorname{Region}(q=2)\subset \operatorname{Region}(q=3)\subset \operatorname{Region}(q=4).$
 }

\end{figure}
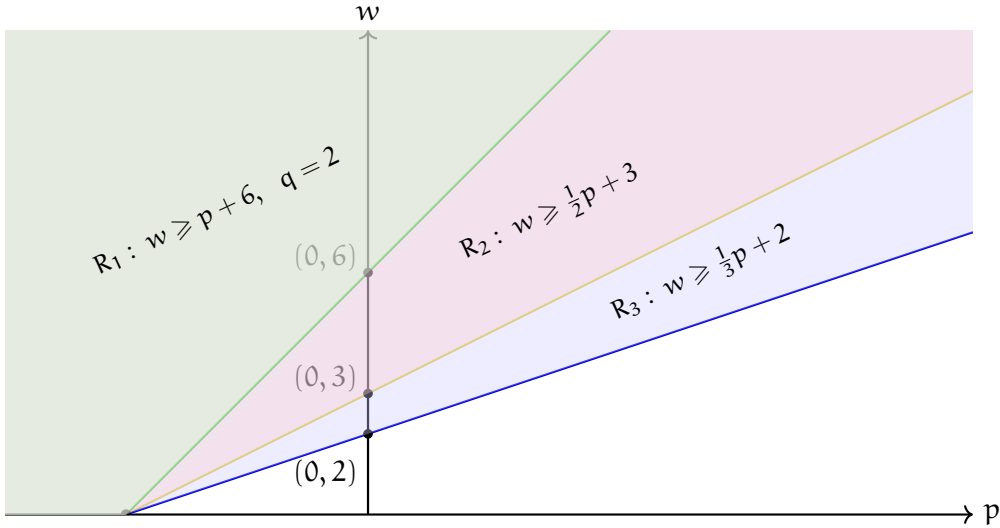

\begin{remark}
Although the figures above are drawn in the $(w,p)$-plane, the natural geometric object in the multigraded setting lives in the $(w_1,\ldots,w_t,p)$-space. Each component determines a boundary hyperplane, $(q-1)w_i\ge m_i+(p-1)$ for, $i=1,\ldots,t.$ and the corresponding vanishing region is the higher-dimensional analogue of the two-dimensional half-plane illustrated above. Thus, the figures should be viewed as coordinate projections of this higher-dimensional geometric object. While  such higher-dimensional regions are difficult to visualize, the same hierarchical behavior with respect to both the weight $q$ and the stage $p$ persists in every coordinate direction.
\end{remark}

We now illustrate the sharpness of the vanishing bound in Corollary~\ref{CorollaryKoszulCohomology} through the following example.

\begin{example}
\label{Example: Sharpness of Higher Koszul Cohomology}
Let $C$ be a hyperelliptic curve of genus $g \ge 2$. Then $C$ admits a basepoint-free $g^1_2$, i.e.\ a line bundle $L$ of degree $2$ with $h^0(C,L)=2$, inducing a degree $2$ morphism
\[
\varphi_L : C \longrightarrow \mathbb{P}^1.
\]
In particular, the canonical bundle satisfies
\[
\omega_C \cong L^{\otimes (g-1)},
\]
and the canonical map factors as
\[
C \xrightarrow{\ \varphi_L\ } \mathbb{P}^1 \xrightarrow{\ \nu_{g-1}\ } \mathbb{P}^{g-1},
\]
where $\nu_{g-1}$ denotes the $(g-1)$-st Veronese embedding.

We now place ourselves in the setting of Corollary~\ref{CorollaryKoszulCohomology}, taking $L$ as the reference line bundle. We first compute the Castelnuovo--Mumford regularity of $\mathcal{O}_C$ with respect to $L$. Let $r$ denote this regularity. Then
\[
H^{1}(L^{\otimes (r-1)})=0.
\]
By Serre duality,
\[
H^{1}(L^{\otimes (r-1)}) \cong H^{0}(\omega_{C} \otimes L^{\otimes (1-r)}),
\]
which vanishes whenever
\[
\deg(\omega_{C} \otimes L^{\otimes (1-r)}) = 2g-2 + 2(1-r) = 2(g-r) < 0,
\]
i.e.\ for $r \ge g+1$. Hence $\mathcal{O}_C$ is $L^{\otimes (g+1)}$-regular.

Next, consider the line bundle $L^{\otimes w}$. By Corollary~\ref{CorollaryKoszulCohomology}, one has
\[
K_{p,q}(C,L^{\otimes w})=0 \quad \text{whenever} \quad (q-1)w \ge g+p+1,
\]
equivalently,
\[
w \ge \frac{g+p+1}{q-1}.
\]

In particular, for $w=g-1$, so that $L^{\otimes w}=\omega_C$, we obtain
\[
K_{p,3}(C,\omega_C)=0 \quad \text{if} \quad g-1 \ge \frac{g+p+1}{2},
\]
i.e.\ for $p \le g-3$.

We now show that this bound is sharp. Consider $K_{g-2,3}(C,\omega_C)$. By Green's duality theorem (cf.\ \cite[Theorem~(2.c.6) and Corollary~(2.c.10)]{GreenI}), one has
\[
K_{p,q}(C,\omega_C) \cong K_{g-2-p,\,3-q}(C,\omega_C)^{*}.
\]
Since $h^0(\omega_C)=g$, taking $(p,q)=(g-2,3)$ yields
\[
K_{g-2,3}(C,\omega_C) \cong K_{0,0}(C,\omega_C)^{*}.
\]

It remains to compute $K_{0,0}(C,\omega_C)$. For this, consider the Koszul complex
\[
\cdots 
\to
\wedge^{p+1} V \otimes H^0\!\big(B^{\otimes (q-1)}\big)
\xrightarrow{\;\delta_{p+1,q-1}\;}
\wedge^{p} V \otimes H^0\!\big(B^{\otimes q}\big)
\xrightarrow{\;\delta_{p,q}\;}
\wedge^{p-1} V \otimes H^0\!\big(B^{\otimes (q+1)}\big)
\to \cdots,
\]
where $ B= \omega_{C}, V = H^0(C,\omega_C)$ and $\delta_{p,q}$ denotes the Koszul differential.

Setting $p=q=0$, the relevant segment becomes
\[
0 \longrightarrow k \longrightarrow 0,
\]
so that
\[
K_{0,0}(C,\omega_C) \cong k.
\]
Consequently,
\[
K_{g-2,3}(C,\omega_C) \cong k,
\]
and the vanishing bound obtained above is sharp.

\end{example}

\begin{remark}
The complete graded Betti table of the canonical image of a hyperelliptic curve of genus $g$ is presented in the Appendix \ref{Appendix_B}.
\end{remark}

\section{\textbf{Appendix}}
\label{Appendix}

\subsection{Further Remarks on the $(M_{q})$-property for Gorenstein Rational Singularities:}
\label{Appendix_A}

The purpose of this section is to explore a possible future generalization of the $(M_q)$-property to singular projective varieties. We record several observations that may be useful in this direction, together with an open problem that naturally arises from them.

We observe that Lemma \ref{LemmaAprodu}, proved in \cite{AproduLombardi} for smooth projective varieties, remains valid for projective varieties with Gorenstein rational singularities. Moreover, the hypothesis that $L\otimes B^{-1}$ is ample can be weakened to the assumption that $L\otimes B^{-1}$ is big and nef.

The following observations explain why the proof of Lemma \ref{LemmaAprodu} continues to work in this setting.

\begin{itemize}
\item[(1)] If $X$ is Gorenstein, then $\omega_X$ is a line bundle. Moreover, in the Gorenstein setting, rational, canonical and klt singularities are equivalent. The fact that $\omega_X$ is a line bundle is essential for preserving the exactness of the short exact sequences of vector bundles appearing in the proof. In contrast, the canonical sheaf of a variety with rational singularities is, in general, only a reflexive sheaf of rank one, which need not be locally free and is not necessarily flat.

\item[(2)] A variety with rational singularities is Cohen--Macaulay. Since Serre duality holds for Cohen--Macaulay projective varieties, the duality arguments used in the proof remain available in this setting.

\item[(3)] Since the assumption that $L\otimes B^{-1}$ is ample is replaced by the weaker condition that it is big and nef, the Kodaira Vanishing Theorem may be replaced by the Kawamata--Viehweg Vanishing Theorem.

\item[(4)] It is important to note that the Kawamata--Viehweg Vanishing Theorem does not hold in general for varieties with rational singularities over a field of characteristic $p>0$. Therefore, throughout this discussion, we implicitly assume that the underlying field is $\mathbb{C}$.
\end{itemize}

Recall that Lemma \ref{LemmaAprodu} together with Green's $K_{p,1}$-Theorem (Theorem (3.c.1) in \cite{GreenI}) yields the $(M_q)$-property. Therefore, in order to extend the $(M_q)$-property to varieties with Gorenstein rational singularities, it is natural to investigate whether Green's $K_{p,1}$-Theorem admits a corresponding extension to this singular setting.

This leads to the following open problem.

\bigskip

\textbf{Open Problem.} Does Green's $K_{p,1}$-Theorem remain valid for projective varieties with Gorenstein rational singularities? More generally, can the proof of the $K_{p,1}$-Theorem for smooth projective varieties be extended to projective varieties with klt singularities whose canonical sheaf is a line bundle?

\subsection{Betti-Table of the Canonical Image of Hyperelliptic Curves of genus $\mathbf{g}$ :} \label{Appendix_B}

The purpose of this subsection is to determine the complete graded Betti table of the canonical image of a hyperelliptic curve of genus g.

 We resume by investigating the Koszul cohomology groups $K_{p,2}(C,\omega_C)$. Since Corollary~\ref{CorollaryKoszulCohomology} does not provide any information regarding the vanishing of the weight-$2$ groups, we appeal to the results of \cite{GreenI}.

We first determine whether $K_{p,2}(C,\omega_C)$ vanishes for $1\le p\le g-3$. By the Duality Theorem of Green \cite{GreenI}, we have
\[
K_{p,2}(C,\omega_C)
\cong
K_{g-2-p,1}(C,\omega_C)^{*}.
\]

We now invoke the converse to the Noether--Enriques--Petri theorem, stated as Conjecture~(5.1) in \cite{GreenI} and proved in the appendix by Green and Lazarsfeld. It asserts that if a smooth curve $C$ of genus $g$ carries a $g^r_d$ satisfying
\[
d\le g-1,\qquad r\ge 1,
\qquad\text{and}\qquad
d-2r\le g-2-p,
\]
then
\[
K_{p,1}(C,\omega_C)\neq 0.
\]

Since $C$ is hyperelliptic, it carries a $g^1_2$. Moreover,
\[
2-2(1)=0\le g-2-p
\]
for every $1\le p\le g-2$. Hence the above result yields
\[
K_{p,1}(C,\omega_C)\neq 0
\qquad
\text{for all }
1\le p\le g-2.
\]

Applying Green's duality, we obtain
\[
K_{g-2-p,2}(C,\omega_C)\neq 0
\qquad
\text{for all }
1\le p\le g-2.
\]

Furthermore, $\omega_C$ is ample and basepoint-free, but it is not very ample since $C$ is hyperelliptic. Consequently, $\omega_C$ cannot satisfy property $(N_0)$. It follows that $\omega_{C}$
does not satisfy property $(N_p)$ for any
\[
0\le p\le g-2.
\]

We now begin an explicit analysis of the Betti table associated to $(C,\omega_C)$.

Since $\omega_C$ fails property $(N_0)$, there exists some $j\ge2$ for which
\[
K_{0,j}(C,\omega_C)\neq 0.
\]
Moreover, Green's duality gives
\[
K_{0,j}(C,\omega_C)
\cong
K_{g-2,\,3-j}(C,\omega_C)^{*}.
\]

We first compute $K_{0,2}(C,\omega_C)$. Let
\[
V:=H^0(C,\omega_C),
\]
and consider the factorization of the canonical morphism
\[
C
\xrightarrow{\;\pi\;}
\mathbb P^1
\xrightarrow{\;\nu_{g-1}\;}
\mathbb P^{g-1}.
\]

By definition,
\[
K_{0,2}(C,\omega_C)
=
\operatorname{Coker}
\Bigl(
V^{\otimes 2}
\xrightarrow{\rho}
H^0(C,\omega_C^{\otimes 2})
\Bigr).
\]

Since
\[
\pi_*\mathcal O_C
\cong
\mathcal O_{\mathbb P^1}
\oplus
\mathcal O_{\mathbb P^1}(-g-1)
\]
and
\[
\omega_C
\cong
\pi^*\mathcal O_{\mathbb P^1}(g-1),
\]
we obtain
\[
\pi_*\omega_C
\cong
\mathcal O_{\mathbb P^1}(g-1)
\oplus
\mathcal O_{\mathbb P^1}(-2).
\]
Hence
\[
H^0(C,\omega_C)
\cong
H^0\!\bigl(\mathbb P^1,\mathcal O_{\mathbb P^1}(g-1)\bigr).
\]

Since the canonical image of a hyperelliptic curve is the rational normal curve $\nu_{g-1}(\mathbb P^1)$, the projective normality of the Veronese embedding implies that
\[
\operatorname{Im}(\rho)
\cong
H^0\!\bigl(\mathbb P^1,\mathcal O_{\mathbb P^1}(2g-2)\bigr).
\]

Therefore,
\[
\dim K_{0,2}(C,\omega_C)
=
h^0(C,\omega_C^{\otimes 2})
-
h^0\!\bigl(\mathbb P^1,\mathcal O_{\mathbb P^1}(2g-2)\bigr).
\]

Since
\[
h^0(C,\omega_C^{\otimes 2})=3g-3
\]
and
\[
h^0\!\bigl(\mathbb P^1,\mathcal O_{\mathbb P^1}(2g-2)\bigr)=2g-1,
\]
it follows that
\[
\dim K_{0,2}(C,\omega_C)
=
3g-3-(2g-1)
=
g-2.
\]

Next, we consider the group $K_{0,3}(C,\omega_C)$. Since Theorem~\ref{Theorem_6.1} applies when $p=0$, we obtain
\[
K_{0,3}(C,\omega_C)=0
\]
provided
\[
2(g-1)\ge g+1,
\]
which is equivalent to $g\ge3$.

Consequently,
\[
K_{p,3}(C,\omega_C)=0
\qquad
\text{for all }
0\le p\le g-3.
\]

Thus the failure of property $(N_0)$ is entirely explained by the non-vanishing of
\[
K_{0,2}(C,\omega_C).
\]

At this stage, the fourth row of the Betti table of $(C,\omega_C)$ is determined. By Green's duality, the first row is determined as well. Since the second and third rows are dual to one another, it remains to compute only one of them in order to recover the entire Betti table.

We therefore focus on the second row, namely the groups
\[
K_{p,1}(C,\omega_C),
\qquad
1\le p\le g-2.
\]
As a first step, we compute
\[
K_{1,1}(C,\omega_C).
\]

 To compute $K_{1,1}(C,\omega_C)$, consider the relevant portion of the Koszul complex:
\[
\cdots
\longrightarrow
\wedge^{p+1}V\otimes H^0\!\bigl(\omega_C^{\otimes(q-1)}\bigr)
\xrightarrow{\;\beta\;}
\wedge^pV\otimes H^0\!\bigl(\omega_C^{\otimes q}\bigr)
\xrightarrow{\;\rho\;}
\wedge^{p-1}V\otimes H^0\!\bigl(\omega_C^{\otimes(q+1)}\bigr)
\longrightarrow
\cdots,
\]
where $V:=H^0(C,\omega_C)$.

By definition,
\[
K_{p,q}(C,\omega_C)
=
\ker(\rho)\big/\operatorname{Im}(\beta).
\]

Specializing to $(p,q)=(1,1)$, we obtain
\[
\wedge^2V
\xrightarrow{\;\beta\;}
V\otimes V
\xrightarrow{\;\rho\;}
H^0(C,\omega_C^{\otimes 2}).
\]

From the computation of $K_{0,2}(C,\omega_C)$ above, we know that
\[
\dim \operatorname{Im}(\rho)=2g-1.
\]
Since $\dim(V\otimes V)=g^2$, it follows that
\[
\dim\ker(\rho)
=
g^2-(2g-1)
=
g^2-2g+1.
\]

Next, the map
\[
\beta:\wedge^2V\longrightarrow V\otimes V
\]
identifies $\operatorname{Im}(\beta)$ with the subspace of skew-symmetric tensors. Consequently,
\[
(V\otimes V)/\operatorname{Im}(\beta)
\cong
\operatorname{Sym}^2(V),
\]
and therefore
\[
\dim\operatorname{Im}(\beta)
=
g^2-\dim\operatorname{Sym}^2(V).
\]

Since
\[
\dim\operatorname{Sym}^2(V)
=
\binom{g+1}{2}
=
\frac{g^2+g}{2},
\]
we obtain
\[
\dim\operatorname{Im}(\beta)
=
g^2-\frac{g^2+g}{2}.
\]

Hence
\[
\begin{aligned}
\dim K_{1,1}(C,\omega_C)
&=
\dim\ker(\rho)-\dim\operatorname{Im}(\beta)\\
&=
\left(g^2-2g+1\right)
-
\left(g^2-\frac{g^2+g}{2}\right)\\
&=
\frac{g^2-3g+2}{2}\\
&=
\frac{(g-1)(g-2)}{2}.
\end{aligned}
\]

Now, we compute the dimension of $K_{p,1}(C,\omega_{C})$ for an arbitrary $p\ge 1$. We only outline the calculation here.

\par

Rewriting the relevant portion of the Koszul complex for $q=1$, we obtain

\[
\cdots
\longrightarrow
\wedge^{p+1}V
\xrightarrow{\;\beta\;}
\wedge^pV\otimes H^0\!(\omega_C)
\xrightarrow{\;\rho\;}
\wedge^{p-1}V\otimes H^0\!(\omega_C^{\otimes 2})
\longrightarrow
\cdots,
\]
where the map $\rho$ is defined by
\[
\rho \bigl( (x_{1}\wedge \cdots \wedge x_{p})\otimes z\bigr)
=
\Sigma (-1)^i
\bigl(
(x_{1}\wedge \cdots \wedge \hat{x_{i}}\wedge \cdots \wedge x_{p})
\otimes x_{i}z
\bigr).
\]
From the previous calculation,
\[
\pi_{*}(\omega_{C}^{\otimes 2})
\cong
\mathcal{O}_{\mathbb{P}^{1}}(2g-2)
\oplus
\mathcal{O}_{\mathbb{P}^{1}}(g-3).
\]

Consequently,
\[
H^0(\omega_{C}^{\otimes 2})
\cong
H^{0}\!\bigl(\mathbb{P}^{1},
\mathcal{O}_{\mathbb{P}^{1}}(2g-2)\bigr)
\oplus
H^{0}\!\bigl(\mathbb{P}^{1},
\mathcal{O}_{\mathbb{P}^{1}}(g-3)\bigr).
\]

Now, from the computation of $K_{0,2}(C,\omega_{C})$, we observe that

\[
\mathrm{Im}\bigl(V\otimes V \to H^{0}(\omega_{C}^{\otimes 2})\bigr)
=
H^{0}\!\bigl(\mathbb{P}^{1},
\mathcal{O}_{\mathbb{P}^{1}}(2g-2)\bigr).
\]

Hence,
\[
\mathrm{Im}(\rho)
=
\wedge^{p-1}(V)
\otimes
H^{0}\!\bigl(\mathbb{P}^{1},
\mathcal{O}_{\mathbb{P}^{1}}(2g-2)\bigr).
\]
Since
\[
H^{0}(C,\omega_{C})
=
H^{0}\!\bigl(\mathbb{P}^{1},
\mathcal{O}_{\mathbb{P}^{1}}(g-1)\bigr),
\]
it follows that
\[
K_{p,1}(C,\omega_{C})
=
K_{p,1}\!\bigl(\mathbb{P}^{1},
\mathcal{O}_{\mathbb{P}^{1}}(g-1)\bigr),
\]
which corresponds to the rational normal curve. Therefore, it suffices to compute the weight one Koszul cohomology groups for the rational normal curve in order to determine the corresponding groups for hyperelliptic curves.
\par
But
\[
K_{p,1}\!\bigl(\mathbb{P}^{1},
\mathcal{O}_{\mathbb{P}^{1}}(g-1)\bigr)
=
p\binom{g-1}{p+1},
\] see \cite{TGS}.
\par
Therefore, we obtain the entire Betti table of the hyperelliptic curve $C$ with respect to the line bundle $\omega_{C}$.

\begin{table}[ht]
\centering

\[
\begin{array}{c|cccccc}
 & 0 & 1 & 2 & \cdots & g-3 & g-2\\
\hline
0 & 1 & 0 & 0 & \cdots & 0 & 0\\[2mm]

1 &
0 &
\binom{g-1}{2} &
2\binom{g-1}{3} &
\cdots &
(g-3)\binom{g-1}{g-2} &
(g-2)\binom{g-1}{g-1}
\\[2mm]

2 &
g-2 &
(g-3)\binom{g-1}{1} &
(g-4)\binom{g-1}{2} &
\cdots &
\binom{g-1}{g-3} &
0
\\[2mm]

3 &
0 &
0 &
0 &
\cdots &
0 &
1
\end{array}
\]

\caption{The Betti table of $(C,\omega_C)$, where $C$ is a hyperelliptic curve of genus $g$.}
\label{tab:hyperelliptic-betti}
\end{table}

\begin{remark}
It is intriguing to observe that
\[
\dim K_{0,2}(C,\omega_C)
=
g-2
=
\operatorname{codim}_{M_g}(\mathcal H_g),
\]
where $\mathcal H_g\subset M_g$ denotes the hyperelliptic locus.
Moreover, as shown above, the non-vanishing of $K_{0,2}(C,\omega_C)$ is precisely the obstruction to the canonical model of a hyperelliptic curve satisfying property $(N_0)$. Whether the numerical coincidence
\[
\dim K_{0,2}(C,\omega_C)
=
\operatorname{codim}_{M_g}(\mathcal H_g)
\]
and all other non-zero Betti numbers appearing in Table~\ref{tab:hyperelliptic-betti}, each expressed explicitly in terms of the genus $g$, admit a deeper geometric explanation remains an interesting question. 
\end{remark}

\section*{\textbf{Acknowledgements}}

First, I would like to thank my PhD advisor, Purnaprajna Bangere, for introducing me to the problem that eventually led to this work. He was the first to suggest investigating the complementary version of the result of Hering--Schenck--Smith, namely the property $(M_q)$ in the multigraded setting. I am grateful for his constant support, encouragement, and inspiration throughout the course of this research.

A name that is indispensable throughout this work is my academic sibling, Debjit Basu. He was the one who advised me to prove Lemma~\ref{Lemma II}, which ultimately led to the general vanishing theorem (Theorem~\ref{Theorem_6.1}) and significantly broadened the scope of this paper. Beyond this, I had the privilege of discussing all of my ideas in this article with him. His constant willingness to engage in mathematical discussions, despite his own commitments, has been invaluable. Every discussion with him helped refine my ideas, deepen my understanding. 

I am grateful to Robert Lazarsfeld for generously taking the time to discuss with me the possible generalization of Green's $K_{p,1}$-Theorem. Our conversation was not only valuable for this work but also provided a broader perspective for my future research. I am especially thankful for his thoughtful advice to take additional time to further refine and strengthen the results before posting the preprint. This advice subsequently led to several important developments and ideas in this paper. I would also like to thank Mihnea Popa for his insightful question during my talk at the Hodge Theory Workshop at Ashoka University, India. His question regarding the syzygies of products of projective spaces inspired me to further investigate their multigraded regularity.

Finally, I am grateful for the opportunity to work on this problem and to the universe for bringing me to it at the right time.

\raggedright
\providecommand{\bysame}{\leavevmode\hbox 
  to3em{\hrulefill}\thinspace}

\end{document}